\pdfoutput=1
\documentclass[a4paper]{book}
\usepackage[T1]{fontenc}
\usepackage{ae,aecompl}
\usepackage{amsmath,amsfonts,amssymb,amsthm,amscd,verbatim}
\usepackage{hyperref}
\usepackage{cite}
\usepackage[english]{babel}

\def\RR{{\mathbb R}}
\def\CC{{\mathbb C}}
\def\HH{{\mathbb H}}
\def\OO{{\mathbb O}}
\def\KK{{\mathbb K}}
\def\FF{{\mathcal F}}
\def\ZZ{{\mathbb Z}}

\def\Sym{{S}}
\def\Lam{{\Lambda}}
\def\End{{\rm End}}
\def\sGL{{\sf GL}}
\def\sSO{{\sf SO}}
\def\sO{{\sf O}}
\def\sSL{{\sf SL}}
\def\sSU{{\sf SU}}
\def\sSp{{\sf Sp}}
\def\sU{{\sf U}}
\def\sE{{\sf E}}
\def\sF{{\sf F}}
\def\sG{{\sf G}}
\def\fso{\frak{so}}
\def\fsu{\frak{su}}
\def\fsp{\frak{sp}}
\def\fu{\frak{u}}
\def\fe{\frak{e}}
\def\ff{\frak{f}}
\def\fg{\frak{g}}
\def\cG{{\mathcal G}}

\def\cg{{\mathcal G'}}

\def\cH{{\mathcal H}}
\def\cJ{{\mathcal J}}

\def\hx{\frak{h}_3}
\def\sx{\frak{sh}_3}
\def\hxk{\frak{h}_3\KK}
\def\sxk{\frak{sh}_3\KK}
\def\fxk{\FF(\hxk)}

\def\tr{{\rm tr}}
\def\der{{\rm der\ }}
\def\derx{{\rm der}}
\def\deg{{{\rm der}_0\ }}
\def\dev{{{\rm der}_1\ }}
\def\Aut{{\rm Aut}}

\def\cQ{{\mathcal Q}}
\def\id{{\rm id}}
\def\LC{L^\times}
\def\LF{L^\bullet}

\def\LT{L^\tau}
\def\lF{l^\bullet}
\def\tlF{{\tilde l}^\bullet}
\def\cD{{\mathcal D }}

\def\cB{{\mathcal B }}
\def\cE{{\mathcal E }}
\def\cH{{\mathcal H }}
\def\cJ{{\mathcal J }}
\def\Im{{\rm Im}}
\def\Re{{\rm Re}}
\def\ad{{\rm ad}}
\def\Ad{{\rm Ad}}
\def\Span{{\rm Span}}
\def\fm{\frak{M}}
\def\cW{{\mathcal U}}

\def\Sec{{\rm Sec}}
\def\ft{\frak{t}}

\def\fk{\frak{k}}
\def\fh{\frak{h}}
\def\fl{\frak{l}}
\def\pr{{\rm pr}}

\def\Y{{\mathcal Y}}
\def\UpsilonM{{\underline \Upsilon}}
\def\XiM{{\underline \Xi}}
\def\mhoM{{\underline \mho}}
\def\Str{{\rm Str}}
\def\str{\frak{str}}
\def\cV{\mathcal{V}}
\def\V{{V_c}}
\def\YM{{\underline\Y}}
\theoremstyle{plain}
\newtheorem{thm}{Theorem}

\newtheorem{lem}{Lemma}
\newtheorem{pro}{Proposition}
\newtheorem{cor}{Corollary}
\theoremstyle{definition}
\newtheorem{defn}{Definition}
\newtheorem{eg}{Example}

\theoremstyle{remark}
\newtheorem{note}{Remark}

\begin{document} 
\setcounter{tocdepth}{1}
\sloppy
\author{Jan Gutt \\ \small Stony Brook University }

\title{Special Riemannian geometries and the Magic Square of Lie Algebras}

\frontmatter

\maketitle
\section*{Summary}
The present work is a slightly revised version of the author's M.Sc. thesis presented at the Physics department of Warsaw University, investigating nonintegrable Riemannian geometries modelled after
certain symmetric spaces related to the Freudenthal-Tits Magic Square. The collection of four
such geometries investigated by Nurowski \cite{nurowski-2006} has been extended by further eight, together with a unified description provided in terms of rank three Jordan algebras and associated constructions. In particular,
symmetric tensors reducing the orthogonal group to adequate structure groups have been found and used to describe geometric properties of corresponding $G$-structures on manifolds. The results obtained this way include: conditions for existence of a natural complex or quaternionic K\"ahler structure; equivalence of the existence of a characteristic connection and the Killing condition on the tensor defining a $G$-structure, which holds for four of the new geometries, thus extending an analogous result of Nurowski. Moreover, the geometries which admit a characteristic connection have been classified. The paper is concluded by an algebraic construction of locally reductive examples.

\tableofcontents

\mainmatter
\chapter*{Introduction}
\addcontentsline{toc}{chapter}{Introduction}
\section*{Special geometries with characteristic torsion}
It is common in differential geometry to express a specific geometric structure on a manifold
in terms of a reduction of the frame bundle to a subbundle with some structure group $G\subset
\sGL(n)$ ($n$ being the dimension of the manifold), i.e. a $G-$structure. Such a reduction
naturally distinguishes a class of compatible connections, namely those connections on the
frame bundle, which restrict to a connection on the $G-$structure. We call
the latter integrable iff it admits a \emph{torsion-free} compatible connection.

In particular,
an $\sO(n)-$structure is always integrable, with \emph{unique} torsion-free connection, namely the Levi-Civita connection of the corresponding Riemannian metric. 
If we thus wish to reduce the orthonormal frame bundle to a $G-$structure with $G\subset\sO(n),$ we readily notice
that the latter is integrable iff the Levi-Civita connection is compatible with respect to it. It then
follows that integrability of a $G-$structure implies reduction of the Riemannian holonomy group
to a subgroup of $G.$

Such reduced Riemannian holonomy groups are classified by the celebrated Berger's theorem,
stating that the holonomy of an irreducible, simply connected, and not (locally) symmetric
Riemannian manifold is either the entire special orthogonal group, or one of the groups corresponding
to a Calabi-Yau, K\"{a}hler, hyper-K\"{a}hler, quaternion-K\"{a}hler, $\sG_2,$ or ${\sf Spin}(7)$
structure. These are usually referred to as (integrable) special Riemannian geometries and have been
extensively studied (in the context of holonomy) throughout last fifty years.

Willing to consider more general geometries, one needs to relax the torsion triviality condition. 
Observe
that the bundle $\Lam^2 TM \otimes TM,$ the torsion of a metric connection on a Riemannian manifold $M$ is a section of,
decomposes into $\sO(n)$ irreducibles as follows:
$$ \Lam^2 TM \otimes TM = \Lam^3 TM \oplus TM \oplus \mathcal{T}. $$
One thus sees 
that there are in general $2^3=8$ classes of metric connections with respect to their torsion.
In what follows, we shall focus on $G-$structures admitting a compatible connection
whose torsion is completely skew, i.e. a section of $\Lam^3 TM.$

The first structures studied in this context where those present on Berger's list, and it
is a common feature of these, that a $G-$connection with skew torsion, provided it exists,
is \emph{unique}. It has been thus called \emph{the} characteristic connection, and its
torsion tensor (determining the connection itself) -- the characteristic torsion. 

Let us now return for a moment to a general $G-$structure on a manifold $M.$ Being able
to describe $G \subset \sGL(n)$ as a stabilizer of certain set of $\RR^n$ tensors, one can
introduce the $G-$structure by means of an analogous set of tensors on $M:$ the subbundle
is then defined to consist of frames mapping the distinguished $\RR^n$ tensors into 
the distinguished tensors on $M$ (point-wise). 

To be specific, let us for the sake of simplicity assume
that $G$ is the stabilizer of a single tensor $\Y \in \otimes^p (\RR^n)^*$; a $G-$structure on
an $n-$dimensional manifold $M$ is then defined by a tensor $\Y_M \in \Sec(\otimes^p T^*M)$ such
that (locally) in some adapted coframe $\theta : TM \to \RR^n$ one has
$$ \Y_M(X_1,\dots,X_p) = \Y(\theta(X_1),\dots,\theta(X_p)) $$
for all  $X_1,\dots,X_p\in TM.$
The fibre above $x\in M$ of the corresponding frame subbundle is the set of frames $e_x : T_x M \to \RR^n$ such that $e_x^* \Y = \Y_M(x).$ The structure group is clearly the stabilizer of $\Y,$ i.e. $G.$
Finally, the compatible connections, viewed as connections in the tangent bundle, are those
with respect to which $\Y_M$ is parallel.

Proceeding again to the Riemannian case, with $G\subset \sO(n)$
being the orthogonal stabilizer of $\Y,$
we easily see
that the $G-$structure is integrable iff $\Y_M$ is parallel with respect to the Levi-Civita connection
$\nabla^{LC}.$
One is then tempted to ask whether a weaker condition on $\nabla^{LC}\Y_M$ would guarantee existence of a $G-$connection with skew torsion. It is not difficult to check,
as it has been noticed by Nurowski \cite{nurowski-2006},
that the existence of such a connection implies vanishing of the symmetric part of the derivative\footnote{
i.e. demanding that $\Y_M$ be a Killing tensor.
}:
\begin{equation} \label{nearly-int} (\nabla^{LC}_X \Y_M)(X,\dots,X) = 0 \quad \forall X \in TM, \end{equation}
a condition we shall call the \emph{nearly-integrability} of (the $G$-structure defined by) $\Y_M.$
One may hope that the converse would also hold in some cases (whether it does, is a purely algebraic
question referring to $\Y$).
As only the symmetric part of
$\Y_M$ enters the latter equation, it is clear that we should restrict our attention to \emph{symmetric} tensors $\Y$ and their orthogonal stabilizers.

\section*{Examples of geometries defined by a symmetric tensor}
The simplest interesting example, thoroughly investigated by Bobienski and Nurowski \cite{bobienski-2005}, involves an irreducible $\sSO(3)$ structure on a five-dimensional Riemannian
manifold $(M,g_M).$ 
The authors first consider the symmetric space $\sSU(3) / \sSO(3)$ and the corresponding symmetric pair:
$$ \fsu(3) = \fso(3) \oplus V,\quad \dim V = 5. $$
The adjoint action of $G=\sSO(3)$ on $V$ defines an irreducible 5-dimensional representation
of the group, which is moreover self-adjoint, so that $V\simeq V^*$ as $G-$modules via an invariant
scalar product $g : V\to V^*.$

It can be further shown, that $\sSO(3) \subset \sO(5)$ is the stabilizer of a tensor $\Upsilon
\in\Sym^3 V$ satisfying the relation
\begin{equation}\label{eqn-nur}
\Upsilon_{m(ij}\Upsilon_{kl)m} = g_{(ij} g_{kl)},
\end{equation}
where abstract index notation is assumed, together with the identification of $V$ with its dual.\footnote{ Equation (\ref{eqn-nur}) is equivalent to demanding that the trivial representation of $G$
appear only once in the decomposition of $\Sym^4 V.$
} Since the representation is irreducible, the tensor is also obviously required to satisfy
$ \Upsilon_{imm} = 0.$

As it has been already described, one defines the $G-$structure by means of a tensor $\Upsilon_M$ on $M$ such that in a (local) adapted coframe $\theta : TM \to V$ one has
$$ \Upsilon_M(X,X,X) = \Upsilon(\theta(X),\theta(X),\theta(X))\quad \&\quad
g_M(X,X) = g(\theta(X),\theta(X)).$$ Equivalently, one may simply demand that $\Upsilon_M$
satisfy the analog of equation (\ref{eqn-nur}) with respect to $g_M.$

One then finds that such five-dimensional $\sSO(3)$ geometries indeed behave in the way we are looking for: nearly integrability of $\Upsilon_M$ (recall equation (\ref{nearly-int})) implies existence of a compatible $\sSO(3)$ connection
with skew torsion, and the latter is moreover unique (that is, characteristic) \cite{bobienski-2005}.

It is now natural to ask, whether a similar setting can be found in other dimensions. Guided
by the defining identity on $\Upsilon,$ i.e. equation (\ref{eqn-nur}), Nurowski checked that the latter
can be satisfied for a symmetric third-rank tensor exactly in four distinguished
dimensions, namely: 5, 8, 14 and 26 \cite{nurowski-2006}. The corresponding symmetric spaces are:
\begin{equation}\label{family1} \frac{\sSU(3)}{\sSO(3)},\quad \frac{\sSU(3)\times\sSU(3)}{\sSU(3)}, \quad \frac{\sSU(6)}{\sSp(3)}, \quad \frac{\sE_{6(-78)}}{\sF_{4(-52)}} \end{equation}
(all four appearing on Cartan's list of irreducible symmetric Riemannian spaces, cf. \cite{loos-1969,besse-1987}). The result
of Nurowski is then that nearly
integrability implies existence of compatible connection with skew torsion in dimensions 5, 8 and
14, while such a connection is unique in dimensions 5, 14 and 26.

The next step would thus be to consider geometries modelled after other symmetric spaces from
Cartan's list. These in particular include symmetric spaces related to a construction known as
the Magic Square of Lie algebras, and investigation of corresponding special geometries is
the task proposed in \cite{nurowski-2006}.
As we shall soon see, there are three families of such
spaces, the first one being exactly (\ref{family1}).
One may expect the geometries modelled after all of these spaces
to exhibit similar properties regarding nearly integrability and characteristic connection.\footnote{
Actually, the problem of uniqueness of the characteristic connection has been recently completely solved
by Nagy \cite{nagy-2007}.
}

\section*{Freudenthal-Tits Magic Square and related symmetric spaces}

The Magic Square of Lie algebras is an outcome of the constructions developed by Freudenthal and Tits
mainly in effort to provide a direct construction of the exceptional groups. When performed over the reals, the Tits construction yields algebras corresponding to the following `magic' square
of compact Lie groups:
$$ \begin{CD}
\sSO(3) @>>> \sSU(3) @>>> \sSp(3) @>>> \sF_4 \\
@VVV @VVV @VVV @VVV \\
\sSU(3) @>>> \sSU(3)\times\sSU(3) @>>> \sSU(6) @>>> \sE_6 \\
@VVV @VVV @VVV @VVV \\
\sSp(3) @>>> \sSU(6) @>>> \sSO(12) @>>> \sE_7 \\
@VVV @VVV @VVV @VVV \\
\sF_4 @>>> \sE_6 @>>> \sE_7 @>>> \sE_8 
\end{CD} $$
together with natural inclusions denoted by the arrows
(the word `magic' referring to the symmetry of the table, which is not explicit in its  original construction, see later). 

Let us now consider the `quotient' of the second column by the first one (with respect
to the inclusions), namely the homogeneous
spaces:
$$ \frac{\sSU(3)}{\sSO(3)},\quad \frac{\sSU(3)\times\sSU(3)}{\sSU(3)}, \quad \frac{\sSU(6)}{\sSp(3)}, \quad \frac{\sE_6}{\sF_4}. $$
Observe that these are exactly the spaces we have already considered in context of a symmetric third rank tensor.
Repeating this procedure for the next pair of columns (2\&3), we obtain spaces which are unfortunately not irreducible: their isotropy representations possess a one-dimensional invariant subspace. This can be resolved by augmenting the subgroup by an extra $\sU(1),$ so that the corresponding four spaces are:
$$\frac{\sSp(3)}{\sU(3)},\quad \frac{\sSU(6)}{S(\sU(3)\times\sU(3))}, \quad \frac{\sSO(12)}{\sU(6)}, \quad \frac{\sE_7}{\sE_6\times\sU(1)}. $$
These are another four irreducible symmetric Riemmanian spaces from Cartan's list, and thus the second of three advertised families (the way this extra $\sU(1)$ sits in the groups is to be explained in the sequel). Moreover, the generator of the additional $\sU(1)$ defines a complex structure making
these symmetric spaces into K\"{a}hler manifolds.

Similar situation occurs for the last pair of columns (3\&4). Here however one has to add extra
$\sSp(1),$ so that the third family cosists of the following irreducible symmetric spaces\footnote{
The product with $\sSp(1)$ is taken with respect to the adjoint (i.e. isotropy) representation.
In fact, $G_0\sSp(1) \simeq (G_0 \times \sSp(1))/{\mathbb{Z}_2}$
}:
$$ \frac{\sF_4}{\sSp(3)\sSp(1)}, \quad \frac{\sE_6}{\sSU(6)\sSp(1)}, \quad \frac{\sE_7}{\sSO(12)\sSp(1)}, \quad \frac{\sE_8}{\sE_7\sSp(1)}. $$
In this  case the generators of the additional $\sSp(1)$ define a quaternionic structure making
these symmetric spaces into quaternion-K\"{a}hler manifolds.

These are the three families of irreducible compact symmetric Riemannian spaces. Nurowski \cite{nurowski-2006} worked out the first one, and proposed the task of investigating the next two:
\begin{quote}
`It is interesting if all these geometries admit characteristic connection. Also, we do not know what objects in $\RR^{\dim M}$ reduce the orthogonal groups $\sSO(\dim M)$ to the above mentioned structure groups. Are these symmetric tensors as it was in the case of the groups $H_k$ [i.e. $\sSO(3)$, $\sSU(3)$, $\sSp(3)$ and $\sF_4$]?'
\end{quote}
We answer these questions and provide a systematic approach rooted in the theory of Jordan algebras.\footnote{
The problem of symmetric defining invariants is addressed in a very original and self-contained way by Cvitanovi\'c in his remarkable book on Group Theory \cite{cvitanovic-2008}, where he also discusses the Magic Square. While we do not use his approach and results, we have to acknowledge that the latter work has been an important source of inspiration. In particular, many of the technical calculations contained in the present paper have been initially carried using Cvitanovi\'c's graphical notation, only later to be translated to the conventional one.
}

\section*{Overview}
The work is split into two parts (chapters),
which we have tried to make possibly independent of each other.

The first one begins citing known results on Jordan algebras and the Tits construction, providing
the minimal theory needed to make sense of the symmetric pairs corresponding to the spaces of our interest, fitting them into a single structure. We then review some further constructions, mainly due to Freudenthal, again introducing just the minimum set of objects needed to describe the isotropy representations
of the symmetric spaces. Having the latter done, we finally find the symmetric invariants giving
the desired reductions (on a Lie-algebraic level) and prove some of their properties.

The second part, building on the results outlined above, deals directly with the main subject of the work, i.e. \emph{geometries} related to the symmetric spaces. We first summarize the results on isotropy representations and the invariant tensors: we claim their existence and properties they satisfy, in such a way that no reference to the Jordan-related objects is needed. Then, we review the subject of intrinsic torsion and characteristic torsion of $G$-structures: these results are known, but not particularily accessible in the literature, so that we prove most of them, also in order to make the reader familiar with the general setting. Finally, $\fg(\KK,\KK')$-geometries are defined and their properties are investigated.
Additionally, we provide a classification of $\fg(\KK,\KK')$-geometries with characteristic torsion and
prove existence of naturally reductive examples, whose characteristic torsion does not vanish.

The most important original results obtained in the present work are: 
\begin{itemize}
\item Proposition \ref{pro-isoreps}, describing the isotropy representations.
\item Proposition \ref{pro-invs}, describing symmetric invariants giving desired reductions.
\item Propositions \ref{pro-tor1}, \ref{pro-tor2} and \ref{pro-tor3}, expressing intrinsic torsion
of $G(\KK,\KK')$-structures in terms of geometric data.
\item Propositions \ref{pro-kaehler} and \ref{pro-qkaehler}, providing a geometric condition
on existence of a natural (quaternion-)K\"ahler structure on the geometries.
\item Theorem \ref{thm-nearly}, stating the equivalence of nearly-integrability of the symmetric tensor
defining a second-family geometry and existence of a characteristic connection.
\item Theorem \ref{thm-red} and Proposition \ref{pro-incl}, providing a simple criterion for existence
of naturally reductive examples with nontrivial characteristic torsion.
 \end{itemize}

\section*{Strominger's superstrings with torsion}
Having argued about the geometric significance of studying special geometries with characteristic torsion, one should not overlook an inspiring motivation coming from physics, found by Strominger in his
1986 paper \cite{strominger-1986}.

Strominger considers the geometric setting for compactification of the common sector of type II Superstring Theory, i.e. a 6-dimensional spin manifold $(M,g)$ equipped with a global nonvanishing spinor field $\epsilon$. In absence of Yang-Mills fields and with constant dilaton, the conditions for $\epsilon$ to generate supersymmetry transformations read \cite{strominger-1986}
\begin{eqnarray}
\nabla^{LC}_X \epsilon + \frac{1}{4} H(X) \cdot \epsilon &=& 0\quad \textrm{for each}\ X\in TM \label{eq-strom} \\
H \cdot \epsilon &=& 0, \nonumber
\end{eqnarray}
where $H\in\Omega^3(M)$ is the strength of the Kalb-Ramond $B$-field, i.e. the background field coupling to massless skew-symmetric excitations of the string, and the dot indicates Clifford action of differential forms on spinor fields
($H$ moreover obeys certain integrability condition involving the Riemannian curvature).

Strominger now introduces a tangent bundle connection $\nabla^H$ whose torsion is $H$, i.e.
$$ \nabla^H_X Y = \nabla^{LC}_X Y + \frac{1}{2} H(X,Y) $$
for each $X\in TM.$ In terms of this new connection, equation (\ref{eq-strom}) reads simply
$$ \nabla^H \epsilon = 0, $$
i.e. $\epsilon$ is parallel with respect to $\nabla^H$. This in turn is equivalent to a reduction
of the holonomy\footnote{ Indeed, existence of a parallel spinor in dimension six implies holonomy reduction to $\sSU(3)$. In the conventional approach this condition is applied to the Levi-Civita connection, leading to compactification on Calabi-Yau threefolds.}
of $\nabla^H$ to (a subgroup of) $\sSU(3)$
(in its $6$-dimensional representation, not to be confused with the $8$-dimensional one mentioned earlier). We thus end up with a $\sSU(3)$-structure
with characteristic torsion, namely $H$ (indeed a compatible connection with skew torsion for a $\sSU(3)$-structure is unique). Moreover, all further equations can be cast in terms of the $\sSU(3)$ structure (i.e. an almost hermitian structure and a complex $3$-form), so that the latter describes both supersymmetry and the $B$-field \cite{strominger-1986}.

Thus, Strominger found that 6-dimensional backgrounds for supersymmetric compactification are naturally described by $\sSU(3)$-structures admitting characteristic torsion. Just as a Yang-Mills field strength is interpreted as the curvature of a connection on a principal bundle, the Kalb-Ramond field strength is then identified with a torsion of the unique connection associated with the $\sSU(3)$-structure\footnote{
A much more elaborate approach to a geometric interpretation of the $B$ field has been developed 
in the last decade  in terms of gerbes with connection \cite{segal-2001} 
}.

This development attracted the attention of mathematicians, who started studying the problem of characteristic torsion, parallel spinors and relation of torsion to curvature, initially for $G$-structures related to Berger's classical list of irreducible Riemannian holonomies (see \cite{friedrich-2003} and references therein). 
In the present work we address
only the question of characteristic torsion. 
Moreover, one should note that the dimensions of the geometries we wish to investigate situate them rather remotely from the usual area of interest of fundamental theories.\footnote{
Nevertheless, the homogeneous spaces related to the Magic Square also find their place in supergravity-related physics \cite{pioline-2006,bellucci-2006}, although without direct relation to the characteristic torsion problem.
}

\section*{Conventions}

All the vector spaces and algebras considered in this work are over the reals, unless stated otherwise.
All manifolds and maps are assumed to be smooth. Stating general results, we shall often use $\RR^n$
as a generic $n$-dimensional vector space, possibly equipped with a generic positive definite scalar product $\langle\cdot,\cdot\rangle$.

We will often make use of the abstract index notation \cite{penrose-1986}, using various sets of letters to index both spaces and tensors, so that a homomorphism of vector spaces $ f : E \to F $ can be written as
$ f^a{}_i \in E^*_i \otimes F^a, $ while its contraction with a vector $ X \in E$ is $ f(X)^a = f^a{}_i X^i \in F^a.$ These indices \emph{never} refer to any specific frame. They are simply labels, which do not assume any (numerical) values.

Most of the spaces we deal with are equipped with a symmetric scalar product, and we explicitly declare that the latter is used to identify a space with its dual. Accordingly, has such an identification been performed, the position of indices becomes irrelevant. The map $f$ given as an example above is then
written $f_{ai} \in F_a \otimes E_i $ and as such it needn't be distinguished from the adjoint $f^* : F\to E$ (of course if we have identified $E\simeq E^*$ and $F\simeq F^*$).

We will often encounter tensors $\Y \in \otimes^p E,$ with $E^*\simeq E.$ Then by $\Y(X_1,\dots,X_p),$
where $X_1,\dots,X_p\in E,$ we mean $\Y_{i_1\dots i_p} X_1^{i_1}\cdots X_p^{i_p} \in \RR$ 
and the order of the
vectors $\Y$ is contracted with does matter (unless $\Y$ is symmetric). We sometimes use also partial application as in $$\Y(X_1,\dots,X_q) \in \otimes^{p-q} E,$$ which means $$\Y_{i_1\dots i_p} X_1^{i_1}\cdot X_q^{i_q} \in E_{i_{q+1}} \otimes \cdots \otimes E_{i_p}$$
in this very order (i.e. the components of a tensor product are contracted from the left).

Being given a group $G \subset \sGL(E),$ with a Lie algebra $\fg \subset \End E,$ we denote both the action of $G$ and $\fg$ on arbitrary tensor products of $E$ and $E^*$ simply by $ g(\cdot)$ and $A(\cdot),$
where  $g\in G$ and $A \in \fg.$ This should be clarified by the example of $\Y\in\otimes^p E^*,$ where
$$ g(\Y)(X_1,\dots, X_p) = \Y(g^{-1}(X_1),\dots,g^{-1}(X_p) $$
$$ A(\Y)(X_1,\dots, X_p) = \Y(-A(X_1),X_2,\dots,X_p) + \dots + \Y(X_1,\dots,X_{p-1},-A(X_p)). $$
Whether a map is to be considered as acting via a group action, or Lie algebra action, should be clear from context. An exception is e.g. a complex structure, which can be viewed both as an orthogonal map, and as spanning a $\fu(1)$ algebra: in this case, we shall by default assume the group action.

Special indexing conventions are introduced for spaces that are to be considered over $\CC.$ These are
described in Remark \ref{not-cpx}, which shall be recalled explicitly whenever needed.

Finally, being given a manifold $M$ we will also use letters to index tensor products of $TM$ and $T^* M$ (being identified when a metric tensor is given). This is extended to the $C^\infty(M)$-modules of
tangent-bundle-valued differential forms, so that for example a local metric connection form can be written as
$ \Gamma_{ab} \in \Omega^1(M,(\Lam^2 TM)_{ab}).$

\section*{Acknowledgements}

I would like to thank my supervisor, Pawel Nurowski, for suggesting the subject of my research, introducing me into the field, his assistance, help and sharing his knowledge and experience with me throughout nearly two years of our cooperation. In particular, his critical remarks were important for the form of the present paper.

At the early stage I benefited from discussions with A. Rod Gover, who spent numerous hours patiently listening to my exposition of some of the initial results I obtained in the first months of 2007.

In autumn 2007 I had the opportunity to give a talk to the group of Ilka Agricola and Thomas Friedrich at HU Berlin. I am grateful for their hospitality, support and valuable remarks. In particular, Ilka Agricola made me aware of her lecture notes \cite{agricola-2006} containing a comprehensive presentation of the general theory of $G$-structures with skew torsion.

Finally, I would like to thank Andrzej Trautman, both for agreeing to write a review of my thesis and, even more importantly, for the beautiful lectures he gave each year in Warsaw. The latter, unique for their clarity and elegance, were a major influence on my interest in differential geometry.

\chapter{Algebraic part}

\section{The Tits construction and symmetric pairs}\label{sec-tits}
The construction of Magic Square\footnote{
Interesting generalisations can be found in \cite{barton-2000, cvitanovic-2008}
} algebras given by Tits \cite{tits-1966} produces a Lie algebra out of two basic `building
blocks': a normed division algebra, i.e. $\RR,\ \CC,\ \HH$ or $\OO,$ and
a Jordan algebra of 3x3 hermitian matrices with entries in a second
normed division algebra.
Eventually thus, it gives a Lie algebra for every pair of normed division algebras, fitting into 
a 4x4 table with rows and columns labelled by $\RR,\ \CC,\ \HH$ and $\OO.$

We shall first introduce the usual algebraic structures on aforementioned `building blocks' (in particular,
their automorphism groups and derivation algebras), whose combination will -- in a fairly natural way -- lead
to a Lie algebra structure on the outcome of Tits' construction.

All algebras and representations used in this section are to be considered over the reals.

\subsection{Four normed division algebras $\KK$} 
As a warm-up we will now recall some well-known facts which hold in general for all the four algebras
(see e.g. ~\cite{schafer-1995,baez-2001}).
Although these are rather obvious, we expose them quite carefully to show that analogous structure appears
in the Jordan case.
As the octonions are the most complex, with their noncommutativity and nonassociativity, it 
is useful to have this algebra in mind when reading subsequent statements. For their simpler subalgebras,
some terms and spaces become trivial, however the formulas, given in their most general form, are still
 correct.

We thus let $\KK$ be one of the four algebras: $\RR,\ \CC,\ \HH$ and $\OO.$ 
We introduce the commutator
$$ [p,q] = pq - qp, $$
nontrivial for the quaternions and octonions, and the associator
$$ [p,q,r] = p(qr) - (pq)r, $$
nontrivial only for the octonions (where alternativity of the latter implies that the associator is
antisymmetric in its three arguments).

These maps are useful when describing the automorphism group $\Aut(\KK),$ i.e. a subgroup of $\sGL(\KK)$
preserving the product. Its Lie algebra, the derivations $\der\KK,$ i.e. a subalgebra
of maps in $\End\KK$ satisfying the Leibniz rule, is then made accessible by the following Lemma
(see \cite{schafer-1995} for a proof):
\begin{lem}\label{lem-dmap} Let us introduce a map
$$ \cD : \Lam^2 \KK \to \End\KK $$
$$ \cD_{p, q}(r) = [[p,q],r] - 3 [p,q,r]\quad\textrm{for}\ p,q,r\in\KK. $$
Then the algebra $\der\KK$ is exactly the image of $\cD.$ Moreover, $\cD$ is equivariant with respect to
$\Aut(\KK).$
\end{lem}

The automorphism groups are found to be ~\cite{baez-2001}:
\begin{eqnarray*}
\Aut(\RR) &\simeq& \{ e\} \\
\Aut(\CC) &\simeq& \ZZ_2 \\
\Aut(\HH) &\simeq& \sSO(3) \\
\Aut(\OO) &\simeq& \sG_2.
\end{eqnarray*}

More structure on $\KK$ is provided by a natural trace, namely the real part $ \Re : \KK \to \RR. $
Combined with the product, it gives rise to a positive definite scalar product
$$ \langle\cdot,\cdot\rangle : \KK \times \KK \to \RR $$
$$ \langle p,q\rangle = \Re(\bar p q). $$
One can decompose $\KK$ into the real line, spanned by the unit, and its orthogonal complement,
namely the imaginary subspace (trivial
for the reals): 
\begin{equation} \label{reim} \KK = \RR 1 \oplus \Im \KK.  \end{equation}
It is then easy to check that:
\begin{lem} The automorphisms $\Aut(\KK)$ preserve the scalar product and the decomposition (\ref{reim}),
acting \emph{irreducibly} on $\Im\KK.$
\end{lem}
Observe moreover, that the map $\cD$ is nontrivial only on $\Lam^2\Im\KK.$
Finally, it is convenient to introduce an antisymmetric product on the imaginary subspace, being
simply a restriction of the usual one:
$$ \times : \Im\KK \times \Im\KK \to \Im\KK $$
$$ p\times q = \frac{1}{2} [p,q].$$
Clearly, this map is also preserved by the automorphisms of $\KK.$

\subsection{Four Jordan algebras $\hxk$} 
Jordan algebras emerged in an attempt to axiomatize the properties of hermitian operators representing observables in quantum mechanics. While the original program turned out to be unsuccessful, it was realized that these algebras are interesting in their own right, allowing a rich theory as somewhat a counterpart of Lie algebras \cite{jacobson-1968,albert-1948}. Indeed, they are characterized by a \emph{commutative} product satisfying an identity which may be thought of as playing a role analogous to that of the Jacobi identity for a Lie bracket:
\begin{defn}
A commutative (yet possibly not associative) algebra $\cJ$ is called Jordan iff for arbitrary two elements $a,b\in\cJ$ the following holds:
$$ (ab)a^2 = a(ba^2) \quad\textrm{
(Jordan identity).}$$
\end{defn}
In particular, 
the algebras of $n$ by $n$ hermitian matrices with entries in either the reals,
complex numbers or quaternions, equipped with anticommutator as the product, are Jordan. The three infinite families are denoted $\frak{h}_n\RR,\ \frak{h}_n\CC$ and $\frak{h}_n\HH.$ For the octonions, however, there is
only a single one, called \emph{the} exceptional Jordan algebra, or Albert algebra, namely $\hx\OO.$

In what follows we will only use 3x3 matrices, so that a corresponding Jordan algebra exists for all four
normed division algebras. Let thus $\KK$ be one of $\RR,\ \CC,\ \HH$ and $\OO.$ We define
$$ \hxk = \{ \left(\begin{matrix}\alpha &a &b \\ \bar a&\beta&c \\ \bar b& \bar c&\gamma\end{matrix}\right)
\ |\ a,b,c\in\KK;\ \alpha,\beta,\gamma\in\RR \}, $$
together with a product
$$ X \circ Y = \frac{1}{2}(XY+YX) $$
where the right hand side features usual matrix multiplication. We then have:
\begin{lem}[Jordan \cite{jordan-1934}] The algebra $(\hxk,\circ)$ is Jordan.
\end{lem}

We shall now follow  the way we used to describe the algebra $\KK$ itself. We thus begin with the automorphism group $\Aut(\hxk)$ and its Lie algebra, the derivations
$\der\hxk.$ For the latter, we again have a convenient map, denoted with a slight abuse of notation by the same letter:
\begin{lem}[cf. \cite{albert-1948}]\label{lem-dmapj}
Let us introduce a map
$$ \cD : \Lam^2 \hxk \to \End\hxk $$
$$ \cD_{X, Y}(Z) = X\circ(Y\circ Z) - Y\circ(X\circ Z)\quad\textrm{for}\ X,Y,Z\in\hxk.$$
Then the algebra $\der\hxk$ is exactly the image of $\cD.$ Moreover, $\cD$ is equivariant with
respect to $\Aut(\hxk).$
\end{lem}

The automorphism groups are found to be \cite{freudenthal-1964,chevalley-1949}:
\begin{eqnarray*}
\Aut(\hx\RR) &\simeq& \sSO(3) \\
\Aut(\hx\CC) &\simeq& \sSU(3) \\
\Aut(\hx\HH) &\simeq& \sSp(3) \\
\Aut(\hx\OO) &\simeq& \sF_{4(-52)}.
\end{eqnarray*}

More structure on $\hxk$ is given by a natural trace, namely the usual matrix trace $\tr: \hxk\to\RR.$
Combined with the product, it gives rise to a positive definite scalar product
$$ \langle\cdot,\cdot\rangle : \hxk\times\hxk \to \RR $$
$$ \langle X,Y \rangle = \tr (X\circ Y). $$
Multiplication in the algebra is then symmetric with respect to the scalar product:
$$ \langle X\circ Y, Z \rangle = \langle Y, X\circ Z \rangle. $$
One can decompose $\hxk$ into the real line, spanned by the unit, and its orthogonal complement,
namely the traceless subspace:
\begin{equation}\label{r1sxk} \hxk = \RR1 \oplus \sxk, \end{equation}
where $\sxk = \ker \tr.$ One then finds that:
\begin{lem}
The automorphisms $\Aut(\hxk)$ preserve the scalar product and the decomposition (\ref{r1sxk}),
acting \emph{irreducibly} on $\sxk.$
\end{lem}

Observe moreover, that the map $\cD$ is nontrivial only on $\Lam^2 \sxk.$ Finally, it is convenient to
introduce a symmetric product on the traceless subspace, being simply a restriction of the usual one:
$$ \times : \sxk \times \sxk \to \sxk $$
$$ X\times Y = X\circ Y - \frac{1}{3}\langle X,Y \rangle. $$
Clearly, this map is also preserved by the automorphisms of $\hxk.$

\subsection{The Magic Square of Lie algebras}
We are now ready to construct the Magic Square. As the algebras to be constructed are parametrized by
a pair of normed division algebras, let us introduce two symbols, $\KK$ and $\KK',$ allowing each of them
to be one of $\RR,\ \CC,\ \HH$ and $\OO.$

Taking $\KK'$ and $\hxk$ as our building blocks, we can form the product algebra $\KK' \otimes \hxk.$
The automorphism group of the latter is simply the product of $\Aut(\KK')$ and $\Aut(\hxk),$ so that
the corresponding derivation algebra is $$
\der(\KK'\otimes\hxk) = \der\KK' \oplus \der\hxk.
$$
Recalling the decompositions of $\KK'$ and $\hxk,$ we have in particular the largest irreducible subspace
$ \Im\KK' \otimes \sxk \subset \KK'\otimes\hxk. $
Let us consider the direct sum of the derivation algebra and its irreducible module:
$$ \fm(\KK,\KK') = \der\KK' \oplus \der\hxk \oplus \Im\KK'\otimes\sxk. $$
Our aim will be to equip the latter with a Lie algebra structure. We keep $\der\KK'\oplus\der\hxk$
as a subalgebra with its original Lie bracket. Its action on the module provides the mixed bracket:
\begin{equation} \label{bktdx} [ d + D, p\otimes X ] = d(p)\otimes X + p\otimes D(X) \end{equation}
for $d\in\der\KK',\ D\in\der\hxk,\ p\in\Im\KK'$ and $X\in\sxk.$

We still need a bracket of two elements of $\Im\KK'\otimes\sxk.$ Recalling the structure introduced so far,
one sees that there are three natural (equivariant w.r.t. the derivations) maps form $\Lam^2 (\Im\KK'\otimes\sxk)$ to $\fm(\KK,\KK'):$
\begin{eqnarray*}
(p\otimes X)\wedge(q\otimes Y) &\mapsto& \langle X, Y \rangle \cD_{p, q} \\
(p\otimes X)\wedge(q\otimes Y) &\mapsto& \langle p, q \rangle \cD_{X, Y} \\
(p\otimes X)\wedge(q\otimes Y) &\mapsto& (p\times q) \otimes (X\times Y),
\end{eqnarray*}
and the most general derivation-equivariant
bracket is a linear combination thereof. The factors are determined by demanding
that the Jacobi identity be satisfied, so that finally we arrive at the following
\begin{lem}[Tits \cite{tits-1966}, cf. \cite{jacobson-1971}]
The space $\fm(\KK,\KK')$ becomes a Lie algebra with the bracket defined by:
\begin{enumerate}
\item{The natural bracket on $\der\KK'\oplus\der\hxk$}
\item{Equation (\ref{bktdx}) for an element of $\der\KK'\oplus\der\hxk$ and an element of $\Im\KK'\otimes\sxk.$}
\item{The following bracket for two elements $p\otimes X$ and $q\otimes Y$ of $\Im\KK'\otimes\sxk:$
\begin{equation}\label{bkt-tits} [ p\otimes X, q\otimes Y] = 
\frac{1}{12}\langle X,Y \rangle \cD_{p,q}
-\langle p,q\rangle \cD_{X,Y} + 
 (p\times q)\otimes (X\times Y).\end{equation}}
\end{enumerate}
\end{lem}

The outcome of this procedure is summarized in the following
\begin{pro}[Tits \cite{tits-1966}]
The algebras $\fm(\KK,\KK')$ are isomorphic (over the reals) to the ones in the table, with columns indexed
by $\KK'$ and rows by $\KK:$
\begin{center}\begin{tabular}{r|cccc}
 & $\RR$ & $\CC$ & $\HH$ & $\OO$ \\
 \hline
$\RR$ & $\fso(3)$ & $\fsu(3)$ & $\fsp(3)$ & $\ff_4$ \\
$\CC$ & $\fsu(3)$ & $\fsu(3)\oplus\fsu(3)$ & $\fsu(6)$ & $\fe_6$ \\
$\HH$ & $\fsp(3)$ & $\fsu(6)$ & $\fso(12)$ & $\fe_7$ \\
$\OO$ & $\ff_4$ & $\fe_6$ & $\fe_7$ & $\fe_8$
\end{tabular}\end{center}
\end{pro}
Two remarks have to be made in this place. First, the algebras in the square appear in their compact
form, an important virtue of Tits' construction. Second, as we have warned before, certain parts of the
construction become trivial for $\KK'$ being $\RR$ or $\CC.$ In particular, we have $\der\RR = 0$ and $\Im\RR = 0,$ so that simply $\fm(\KK,\RR) = \der\hxk.$

\subsection{Cayley-Dickson and induced decompositions}
We shall now recall the Cayley-Dickson decomposition of the complex numbers, quaternions and octonions, namely:
$$ \CC = \RR \oplus i\RR,\quad \HH = \CC \oplus j\CC,\quad \OO = \HH \oplus l\HH, $$
and check that, when applied to $\KK',$ they naturally lead to decompositions of the magic square algebras $\fm(\KK,\KK')$ into symmetric pairs.
Let us then from now on assume $\KK'\neq\RR.$ One easily checks that the Cayley-Dickson construction defines a $\mathbb{Z}_2$ grading on $\KK',$ that is:
$$ \KK' = \KK'_0 \oplus \KK'_1 $$
$$ \KK'_i \cdot \KK'_j = \KK'_{i+j} $$
where $i,j\in\mathbb{Z}_2,$ and moreover $\KK'_i$ and $\KK'_j$ are orthogonal with respect to
$\langle\cdot,\cdot\rangle$ whenever $i\neq j.$

Using the latter and Lemma \ref{lem-dmap}, we can decompose the derivation algebra of
$\KK'$ into a symmetric pair:
\begin{lem} \label{lem-der-grad}
Let us define
\begin{eqnarray*} 
\deg\KK' &=& \cD(\Lam^2 \KK'_0 \oplus \Lam^2 \KK'_1 ) \\
\dev\KK' &=& \cD(\KK'_0 \wedge \KK'_1).
\end{eqnarray*}
Then
$$ \der\KK' = \deg\KK' \oplus \dev\KK' $$
with
\begin{equation}\label{der-grad-bkt} [\derx_i\ \KK',\derx_j\ \KK'] = \derx_{i+j}\ \KK' \end{equation}
\begin{equation}\label{der-grad-act} \derx_i\ \KK'(\KK'_j) \subset \KK'_{i+j} \end{equation}
where $i,j\in\mathbb{Z}_2.$
\end{lem}
\begin{proof} Using the formula for $\cD,$ we have
$$ \cD(\KK'_{i_1} \wedge \KK'_{i_2})(\KK'_j) \subset \KK'_{i_1+i_2+j}, $$
which implies (\ref{der-grad-act}). Then, using equivariance of $\cD,$ (\ref{der-grad-bkt})
follows immediately. 
Now, since $\der\KK'$ is the image of $\Lam^2\KK'$ under $\cD,$ it is clear
that $\deg\KK' + \dev\KK' = \der\KK'.$ It remains to check that the intersection of these
spaces is zero. But that already follows from (\ref{der-grad-act}).
\end{proof}

We are now interested in identifying the algebras $\deg\KK'.$ 
Note first, that it is a property of the Cayley-Dickson construction, that every element of the even subspace may be represented as a product of two elements of the odd subspace -- thus the representation of $\deg\KK'$ on $\KK'_1$ is necessarily faithful\footnote{
Indeed, assume that $d\in\deg\KK'$ acts trivially on $\KK'_1$. Then it acts trivially on entire $\KK'$
and thus $d=0$.
}
and there exists an injective homomorphism
$ \deg\KK' \to \fso(\KK'_1, \langle\cdot,\cdot\rangle).$ In fact, one finds that it is also
surjective:
\begin{lem} \label{lem-deg} $\deg\KK'\simeq\fso(\KK'_1,\langle\cdot,\cdot\rangle)$
\end{lem}

\begin{proof}
Trivial for $\KK'=\CC,$ and straightforward for $\KK'=\HH,$ with $\deg\HH$ generated by 
$\ad_i = \frac{1}{2}\cD_{j,k}.$

The octonionic case, however, is more involved. As $\fso(4) \simeq \fsp(1) \oplus \fsp(1),$
we shall construct two homomorphisms form $\fsp(1)$ to $\deg\OO$ and check that they combine
into an isomorphism from $\fso(4)$ to $\deg\OO.$

To this end,
we consider the usual Lie algebra $\fsp(1)$ of imaginary quaternions with the commutator as a bracket. This algebra acts with left and right multiplications on $\OO=\HH\oplus l\HH,$ preserving the decomposition ($l$
is some imaginary octonion orthogonal $\HH\subset\OO$). This gives rise to four irreducible four-dimensional, and necessarily equivalent, representations of the algebra,
namely:
$$ 
q \mapsto L_q|_\HH,\quad
q \mapsto L_q|_{l\HH},\quad
q \mapsto R_{\bar q}|_\HH,\quad
q \mapsto R_{\bar q}|_{l\HH}
$$
for $q\in\Im\HH\simeq\fsp(1).$
There in particular exists an orthogonal intertwiner of left multiplications
$$ \varphi : \HH \to l\HH $$
\begin{equation} \label{varphi} p\varphi(x) = \varphi(px) \end{equation}
for $p\in\Im\HH$ and $x\in\HH.$ Orthogonality implies  $\varphi(x)^2 = -|x|^2.$

We now define the maps $$\cE,\cE':\Im\HH \to \fso(\OO)$$
\begin{eqnarray*}
\cE(q)|_\HH = \ad_q & & \cE'(q)|_\HH = 0 \\
\cE(q)|_{l\HH} = L_q & & \cE'(q)|_{l\HH} = \varphi \circ R_{\bar q} \circ \varphi^{-1}.
\end{eqnarray*}
One readily checks that each of these is a Lie algebra homomorphism and that 
$\cE(q)$ commutes with $\cE'(q')$ for $q,q'\in\Im\HH$ (this is because left and right multiplications commute, and $\varphi$ is left-equivariant). 
It is also clear that the kernel of $\cE\oplus\cE' : \Im\HH \oplus \Im\HH \to \fso(\OO)$ is trivial.
Thus we find that $\cE\oplus\cE'$ is an injective Lie algebra homomorphism from $\fso(4)\simeq\fsp(1)\oplus\fsp(1)$ to $\fso(\OO).$ Moreover, the elements of its image explicitly preserve the decomposition $\OO = \HH \oplus l\HH.$

The last step is to show that $\cE_q$ and $\cE'_q$ for $q\in\Im\HH$ are derivations (for the moment we
write the first argument in a subscript). We have to check the Leibniz formula patiently
for the products $xy,$ $x\varphi(y)=-\varphi(y)x=\varphi(xy)$ and $\varphi(x)\varphi(y),$
where $x,y\in\HH.$ The first case is obvious (clearly, $\cE_q$ and $\cE_q'$ restricted to
$\HH$ are derivations of the quaternions). Evaluating the next one, we have:
\begin{eqnarray*}
\cE_q(\varphi(xy)) = &\varphi(qxy) & = [q,x]\varphi(y) + x\varphi(qy) = \cE_q(x)\varphi(y)
+ x\cE_q(\varphi(y)) \\
\cE'_q(\varphi(xy)) = &\varphi(xy\bar q)& = \cE'_q(x)\varphi(y) + x\cE'_q(\varphi(y)).
\end{eqnarray*}
To handle expressions like $\varphi(x)\varphi(y),$ we use the orthogonality of $\varphi$ and
multiply equation (\ref{varphi}) with $\varphi(px)$ on the right and $p^{-1}$ on the left to
obtain
$ \varphi(x) \varphi(px) = - p^{-1} |p|^2 |x|^2. $
Then, setting $p = yx^{-1},$ we get
$$ \varphi(x)\varphi(y) = - x \bar y. $$
We can now check the following:
\begin{eqnarray*}
\cE_q(-x\bar y) = & -[q,x\bar y] & = -qx\bar y + x \bar y q =\cE_q(\varphi(x)) \varphi(y) + \varphi(x)\cE_q(\varphi(y)) \\
\cE'_q(-x\bar y) = & 0 & = -x\bar q\bar y - x q\bar y = \cE'_q(\varphi(x))\varphi(y) + \varphi(x)\cE'_q(\varphi(y)).
\end{eqnarray*}
This way we have checked that the image of $\cE\oplus\cE'$ is indeed in $\deg\OO.$ Finally then,
there is an injective homomorphism
$$ \cE\oplus\cE' : \fso(4)\simeq\fsp(1) \oplus \fsp(1) \to \deg\OO, $$
and thus an isomorphism.
\end{proof}

The maps $\cE$ and $\cE'$ will be used once again in the sequel, so let us mention them in
a separate
\begin{cor} \label{cor-ee} There is an isomorphism of Lie algebras
$$ \cE \oplus \cE' : \fsp(1) \oplus \fsp(1) \to \deg\OO, $$
where $\cE$ and $\cE'$ have been introduced in the proof of Lemma \ref{lem-deg}. When restricted
to $\HH\subset\OO,$ the image of $\cE$ is $\der\HH\simeq\fsp(1)$ while $\cE'$ is trivial.
\end{cor}

\subsection{Symmetric pairs related to the Magic Square} \label{ss-symd}
Let us recall the definition of our magic square algebras:
$$ \fm(\KK,\KK') = \der\KK' \oplus \der\hxk \oplus \Im\KK'\otimes\sxk, $$
where, as before, $\KK'\neq\RR.$ We can now apply the Cayley-Dickson and induced decompositions
to $\KK'$ and $\der\KK'$ in the latter formula, and try to group the resulting subspaces into a
symmetric pair. A natural candidate for the latter is easily found, and we have the following
\begin{lem} The decomposition $\fm(\KK,\KK') = \fg(\KK,\KK') \oplus V(\KK,\KK'),$ with
\begin{eqnarray*}
\fg(\KK,\KK') =& \der\hxk\ \oplus& \deg\KK'\oplus\Im\KK'_0\otimes\sxk \\
V(\KK,\KK') =& & \dev\KK' \oplus \KK'_1\otimes\sxk
\end{eqnarray*}
yields a symmetric pair.
\end{lem}
\begin{proof} 
Let us for the sake of brevity omit $(\KK,\KK'),$ writing simply $\fg$ and $V.$ What
has to be checked is that:
$$ [\fg,\fg]\subset\fg,\quad [\fg,V]\subset V,\quad [V,V]\subset\fg. $$
Note first, that $\der\hxk$ commutes with $\der\KK'$ and acts only on the $\sxk$ factor in
$\Im\KK'\otimes\sxk,$ so we already have $$[\der\hxk,\fg]\subset\fg\quad\&\quad [\der\hxk,V]\subset V.$$
The other conditions basically follow from Lemma
\ref{lem-der-grad}. Indeed, the latter implies that $\deg\KK'\oplus\dev\KK'$ is
itself a symmetric pair. Moreover, as $\deg\KK'$ preserves the Cayley-Dickson decomposition,
we have $$[\deg\KK',\fg]\subset\fg\quad\&\quad[\deg\KK',V] \subset V.$$ 
Furthermore, as $\dev\KK'$ maps $\KK'_0$ into $\KK'_1,$ we have $[\Im\KK'_0\otimes\sxk,\dev\KK']
\subset\KK'_1\otimes\sxk.$ Using (\ref{bkt-tits}), the definitions of Lemma \ref{lem-der-grad}
and orthogonality of the Cayley-Dickson decomposition,
we also find that $[\Im\KK'_0\otimes\sxk,
\KK'_1\otimes\sxk] \subset \KK'_1\otimes\sxk.$ Combining these gives
$$[\Im\KK'_0\otimes\sxk,\fg]\subset\fg\quad\&\quad[\Im\KK'_0\otimes\sxk,V] \subset V.$$ 
We finally have to check the $[V,V]$ bracket. Recalling that $\dev\KK'$ is in the odd subalgebra of
the derivations, and that it maps $\KK'_1$ to $\Im\KK'_0,$ we have $[\dev\KK',V]\subset\fg.$ 
Using (\ref{bkt-tits}) once again, we also find that $[\KK'_1\otimes\sxk,\KK'_1\otimes\sxk]\subset\fg.$
Thus we have checked that $[V,V]\subset\fg,$ which completes the proof.
\end{proof}

What we claim now, is that the pairs $\fg(\KK,\KK')\oplus V(\KK,\KK')$ are exactly the
ones corresponding to the symmetric spaces described in the introduction. Recall, that we wanted the
subgroups in the three families to be, respectively, those of the first column, those of the second column times an additional $\sU(1),$ and those of the third column times an additional $\sSp(1).$ We will first show that
the algebras $\fg(\KK,\KK')$ are indeed of this form.

\begin{pro} There are isomorphisms of Lie algebras:
\begin{eqnarray*}
\fg(\KK,\CC) &\simeq& \fm(\KK,\RR) \\
\fg(\KK,\HH) &\simeq& \fm(\KK,\CC)\oplus \fu(1) \\
\fg(\KK,\OO) &\simeq& \fm(\KK,\HH)\oplus \fsp(1).
\end{eqnarray*}
\end{pro}
\begin{proof}
In the first case, we have immediately
$ \fg(\KK,\CC) = \der\hxk = \fm(\KK,\RR). $
In the second one, we recall that $\deg\HH \simeq \fu(1)$ is generated by $\ad_i,$ and thus
preserves $\Im\HH_0 = i\RR \subset\HH.$ It then follows that
$$ \fg(\KK,\HH) = \fu(1)\oplus \der\hxk \oplus i\RR\otimes\sxk = \fu(1) \oplus \fm(\KK,\CC),$$
with the rightmost expression being a direct sum of Lie algebras. Finally, in the last case,
we recall from corollary \ref{cor-ee}, that $\deg\OO \simeq \fsp(1)\oplus\fsp(1)$ with the first
$\fsp(1)$ acting on $\HH\subset\OO$ as the derivations $\der\HH,$ and the other one -- trivially.
We then have
$$ \fg(\KK,\OO) = \fsp(1)\oplus\der\HH\oplus\der\hxk\oplus\Im\HH\otimes\sxk =
\fsp(1) \oplus \fm(\KK,\HH)$$
with the rightmost expression being again a direct sum of Lie algebras.
\end{proof}

Our next goal is to describe the adjoint representations of $\fg(\KK,\KK')$ on $V(\KK,\KK'),$ showing,
in particular, that they are faithful and irreducible. It will then turn out, that the connected subgroups of $\sGL(V(\KK,\KK'))$ resulting from exponentiating the adjoint representation of $\fg(\KK,\KK')$ are indeed the magic square groups (admitting a nice description in terms of invariants), possibly augmented by an additional $\sU(1)$ or $\sSp(1),$ corresponding to a complex or quaternionic structure on $V(\KK,\KK').$

\section{More structure on Jordan algebras}\label{sec-fts}
Before we can perform what has just been indicated, we need to introduce some further structure on the Jordan algebras $\hxk,$ and their related Freudenthal Triple Systems (FTS, to be defined).

\subsection{The determinant and Freudenthal product on $\hxk.$}
If the dimension of a (commutative) algebra is finite, it is clear that subsequent powers of an element of the algebra cannot be linearily independent. Instead, they satisfy a characteristic equation, polynomial in the element of the algebra. In case of the algebras of our interest, $\hxk,$ the characteristic equation is of degree three:
\begin{lem}[cf. \cite{freudenthal-1964,jacobson-1968}] \label{lem-tsn}
There exist natural maps $T,S,N : \hxk\to\RR,$ respectively linear, quadratic and cubic, such that
for each $X\in\hxk$ the following holds:
$$ X^3 - T(X) X^2 + S(X,X) X - N(X,X,X) = 0. $$
They may be expressed using the product and trace, as follows:
\begin{eqnarray*}
T(X) &=& \tr X \\
S(X,X) &=& \frac{1}{2}(\tr X)^2 - \frac{1}{2}\tr X^2 \\
N(X,X,X) &=& \frac{1}{3}\tr X^3 - \frac{1}{2}\tr X^2 \tr X + \frac{1}{6} (\tr X)^3. 
\end{eqnarray*}
\end{lem}

The cubic form $N$ is usually referred to as the \emph{norm}. One can however check, that for all four $\KK$ it exactly conincides with the matrix determinant, and we will use this more familiar notion in the sequel. To avoid explaining the subtleties of an octonionic determinant and proving the identity, we simply \emph{define} the symbol $\det$ using the formula given in the Lemma:
$$ \det X := N(X,X,X) $$
for $X\in\hxk.$

We are now naturally interested in the isotropy groups of the determinant.
These are
related to what is called the \emph{structure alebra} $\str(\hxk)\subset\End(\hxk),$ namely
the Lie algebra generated by all multiplications
$$L_X : Y \mapsto X \circ Y $$ for $X\in\hxk.$ 

Recalling Lemma \ref{lem-dmapj}, one has $[L_X,L_Y]\in\der\hxk,$
and naturally $[d,L_X]=L_{d(x)}$ for a derivation $d\in\der\hxk.$ It thus follows, that
$$\str(\hxk) = \der\hxk \oplus L_{\hxk}, $$
where $ L_{\hxk} = \{ L_X\ |\ X\in\hxk \}.$ Multiplications by multiples of identity generate the centre, and one can decompose
$ \str(\hxk) = L_1 \RR \oplus \str_0(\hxk),$
where 
\begin{equation} \label{def-str0}
\str_0(\hxk) = \der\hxk \oplus L_{\sxk}, 
\end{equation}
and $ L_{\sxk} = \{ L_X\ |\ X\in\sxk \}.$ The latter is called the \emph{reduced structure algebra} of $\hxk.$

Exponentiating these algebras to connected subgroups of $\sGL(\hxk),$ one obtains respectively
the structure group $\Str(\hxk)$ and
reduced structure group $\Str_0(\hxk),$ with $\Str(\hxk) = \Str_0(\hxk)\times\RR.$
We then have:
\begin{lem}[cf. \cite{chevalley-1949,freudenthal-1964,jacobson-1971}] The stabilizer of the map $\det : \hxk\to\RR$ in $\sGL(\hxk)$ is precisely the reduced structure
group $\Str_0(\hxk).$
\end{lem}

These groups are found to be\footnote{
Wangberg \cite{wangberg} interestingly argues, that in certain sense $\sE_{6(-26)} =$ ``$\sSL(3,\OO)$''.
}\cite{chevalley-1949,freudenthal-1964}:
\begin{eqnarray*}
\Str_0(\hx\RR) & \simeq & \sSL(3,\RR) \\
\Str_0(\hx\CC) & \simeq & \sSL(3,\CC) \\
\Str_0(\hx\HH) & \simeq & \sSL(3,\HH) \\
\Str_0(\hx\OO) & \simeq & \sE_{6(-26)}
\end{eqnarray*}
(recall that $\sSL(3,\HH)$ is the group of quaternionic matrices with the \emph{quaternionic determinant} equal one, so that their Lie algebra $\frak{sl}(3,\HH)$ consists
of matrices with vanishing \emph{real part} of the trace).
Note that these are \emph{non-compact} forms of the second column of the magic square algebras.

Another incarnation of the natural cubic form, useful in formulas, is the symmetric Freudenthal product,
$$ \bullet : \hxk \times \hxk \to \hxk $$
$$ \langle X \bullet Y, Z \rangle = 3N(X,Y,Z) $$
for $X,Y,Z\in\hxk,$ together with a corresponding multiplication map
$$ \LF_X : Y \mapsto X\bullet Y $$
and a quadratic map
$$ \sharp : \hxk \to \hxk $$
$$ X^\sharp = X \bullet X. $$
Expressing the latter using powers of $X$ and their traces, and comparing with the formulae
of Lemma \ref{lem-tsn}, one finds that
$$ X^{\sharp\sharp} = (\det X) X. $$

\subsection{FTS $\fxk$ and the triple product.}
Freudenthal \cite{freudenthal-1954,freudenthal-1955,freudenthal-1959,freudenthal-1963}
originally constructed the group $\sE_{7(-25)}$ as the automorphism group of certain algebraic object, namely the space $\hx\OO \oplus \hx\OO \oplus \RR \oplus \RR$ equipped with a symmetric triple product, mapping \emph{three} elements of the space to a fourth one and satisfying certain identities with respect to a natural symplectic form on this space. This construction has been extended in many ways (see \cite{meyberg-1968,brown-1969,brown-1984,kantor-1978}). We shall be interested in four FTS associated with the four Jordan algebras $\hxk$ (where the case $\KK=\OO$ corresponds to the original construction).

Let us thus define the space of the FTS associated with $\hxk$ to be
$$ \fxk = (\RR\oplus\hxk) \otimes \RR^2 $$
and equip it with a natural symplectic form
$$ 
\omega(\left(\begin{matrix} x \\ X \end{matrix}\right) \otimes \left(\begin{matrix}\xi_1 \\ \xi_2 \end{matrix}\right),
\left(\begin{matrix} y \\ Y \end{matrix}\right) \otimes \left(\begin{matrix}\eta_1 \\ \eta_2 \end{matrix}\right)
) = (xy + \langle X,Y \rangle )(\xi_1\eta_2 - \xi_2\eta_1),
$$
where $X,Y\in\hxk$ and $x,\xi_1,\xi_2,y,\eta_1,\eta_2\in\RR.$

Next, we introduce a quartic map $\cQ : \fxk \to \RR,$ such that for
\begin{equation}\label{genf}
 F = 
\left(\begin{matrix} x \\ X \end{matrix}\right) \otimes \left(\begin{matrix}1 \\ 0 \end{matrix}\right)
+ \left(\begin{matrix} \tilde x \\ \tilde X \end{matrix}\right) \otimes \left(\begin{matrix}0 \\ 1 \end{matrix}\right),
\end{equation}
where $X,\tilde X\in\hxk$ and $x,\tilde x\in\RR,$
one has
$$ \cQ(F,F,F,F)= \langle X^\sharp,\tilde X^\sharp \rangle - x\det X - \tilde x \det\tilde X - \frac{1}{4}(\langle X,\tilde X\rangle - x\tilde x)^2.
$$
The symmetric triple product
$$\tau : \fxk\times\fxk\times\fxk\to\fxk $$ is then defined by
$$ \omega(F',\tau(F,F,F)) = \cQ(F',F,F,F) $$
for $F,F'\in\fxk.$

We are now interested in the automorphism groups $\Aut(\fxk),$ defined to be the ones preserving both
the symplectic form $\omega$ and triple product $\tau,$ so that
$$\tau(a(F),a(F),a(F)) = a(\tau(F,F,F))$$
for $a\in\Aut(\fxk)$ and $F\in\fxk.$ The corresponding Lie algebra, called the derivations $\der\fxk,$
is given by the following
\begin{lem}[cf. \cite{jacobson-1971,freudenthal-1954}]\label{lem-derfts}
The derivations of the FTS $\fxk$ are the following subspace of $\End\fxk:$
$$\der\fxk = \cH_0(\der\hxk\oplus\hxk) \oplus \cH_1(\hxk\oplus\hxk), $$
where the maps 
\begin{eqnarray*}
\cH_0 &:& \der\hxk\oplus\hxk \to \End\fxk \\
\cH_1 &:& \hxk\oplus\hxk \to \End\fxk
\end{eqnarray*}
are given by
$$ \cH_0(D,C) =
\left(\begin{matrix} -\tr C &  \\  & L_C \end{matrix}\right) \otimes
\left(\begin{matrix} 1 &  \\  & -1 \end{matrix}\right)
+
\left(\begin{matrix} 0 &  \\  & D \end{matrix}\right) \otimes
\left(\begin{matrix} 1 &  \\  & 1 \end{matrix}\right)
$$
$$ \cH_1(A,B) =
\left(\begin{matrix}  0& \langle A,\cdot\rangle \\ A & -2 \LF_A \end{matrix}\right) \otimes
\left(\begin{matrix}  & -1 \\ 1 &  \end{matrix}\right)
+
\left(\begin{matrix}  0& \langle B,\cdot\rangle \\ B & 2 \LF_B \end{matrix}\right) \otimes
\left(\begin{matrix}  & 1 \\ 1 &  \end{matrix}\right)
$$
as operators on $\fxk = (\RR\oplus\hxk)\otimes\RR^2.$ 
\end{lem}
\begin{proof}[Sketch of a Proof] This form of the derivation algebra of $\fxk$ is usually presented for $\KK=\OO,$ as a construction of the exceptional algebra $\fe_7.$ However, the same reasoning works equally
well for $\KK=\RR,\CC,\HH.$ Indeed, let us write a generic element of $\End\fxk$ as
$$ E = \sum_{k=1}^4 \left(\begin{matrix} 
\lambda_k & \langle H_k,\cdot\rangle \\ K_k & E_k
\end{matrix}\right) \otimes \epsilon_k, $$
for $\lambda_k\in\RR,$ $H_k,K_k\in\hxk$ and $E_k\in\End\hxk,$ where $\epsilon_k$ is a basis in $\End\RR^2$
consisting of matrices with one nonzero element.
Then, expanding the equation
$$ \cQ(E(F),F,F,F) = 0 \quad \textrm{for each}\ F\in\fxk,$$
and expressing $F$ as in (\ref{genf}), one can solve it using the fact that a general map in $\End\hxk$
stabilizing $\det$ is $D + L_C$ for $C\in\hxk, $ $D\in\der\hxk,$ and its conjugate w.r.t. $\langle\cdot,\cdot\rangle$ is $D - L_C.$ A general solution turns out to be the one given in the Lemma. One then
checks, that it also preserves $\omega,$ and thus is a derivation. Conversely, every derivation,
preserving $\omega$ and $\tau,$ must also preserve $\cQ.$
\end{proof}

The argument we have just sketched leads to the following
\begin{cor}[cf. \cite{jacobson-1971}]
The stabilizer of the map $\cQ : \fxk\to\RR$ in $\sGL(\fxk)$ is precisely the automorphism
group $\Aut(\fxk).$
\end{cor}

These groups are found to be \cite{gunaydin-1993}:
\begin{eqnarray*}
\Aut(\FF(\hx\RR)) &\simeq& \sSp(6,\RR) \\
\Aut(\FF(\hx\CC)) &\simeq& \sSU(3,3) \\
\Aut(\FF(\hx\HH)) &\simeq& \sSO^*(12) \\
\Aut(\FF(\hx\OO)) &\simeq& \sE_{7(-25)}.
\end{eqnarray*}
Note that these are \emph{non-compact} forms of the third column of magic square algebras.

For later use, we also equip $\fxk$ with a scalar product
$$ 
\langle\left(\begin{matrix} x \\ X \end{matrix}\right) \otimes \left(\begin{matrix}\xi_1 \\ \xi_2 \end{matrix}\right),
\left(\begin{matrix} y \\ Y \end{matrix}\right) \otimes \left(\begin{matrix}\eta_1 \\ \eta_2 \end{matrix}\right)\rangle
 = (xy + \langle X,Y \rangle )(\xi_1\eta_1 + \xi_2\eta_2),
$$
where $X,Y\in\hxk$ and $x,\xi_1,\xi_2,y,\eta_1,\eta_2\in\RR.$

\def\cA{\mathcal{A}}
\section{Identifying the isotropy representations}

We have so far encountered three families of isotropy groups for certain structures, namely: Jordan algebra automorphisms $\Aut(\hxk),$ stabilizer of the determinant $\Str_0(\hxk)$ and FTS automorphisms $\Aut(\fxk).$ While the first family forms exactly the first column of the magic square, the other two are \emph{non-compact} forms of the second and third column. In what follows, we shall construct their compact forms in terms of \emph{unitary} defining representations.

We will need the following simple fact:
\begin{lem}\label{lem-cpx}
Let $V$ be a real vector space equipped with a scalar product $\langle\cdot,\cdot\rangle$ and $\cA\subset\End V$ a semisimple Lie algebra acting irreducibly on $V$. The scalar product gives rise to a decomposition
$$ \cA = \cA_{\Lam} \oplus \cA_{\Sym}, $$
corresponding to $\End V \simeq \Lam^2 V \oplus \Sym^2 V.$

Consider now the complexification $V^\CC = \CC\otimes V$ 
equipped with a hermitian inner product given by a sesquilinear extension of $\langle\cdot,\cdot\rangle$, 
as well as the complexification of the algebra, $\cA_\CC = \CC\otimes\cA \subset \CC\otimes\End V\simeq \End_\CC(V^\CC).$ Let now $$\cA_\fu = \cA_\CC \cap \fu(V^\CC)$$ be the antihermitian subalgebra of $\cA_\CC.$
Then 
$$ \cA_\fu = \cA_{\Lam} \oplus i\cA_{\Sym} $$
and is the compact real form of $\cA_\CC.$ It moreover acts complex-irreducibly on $V^\CC.$
\end{lem}

\begin{proof} The form of $\cA_\fu$ follows simply from the consideration of $\cA_\CC = \cA_\Lam \oplus i\cA_\Sym \oplus i\cA_\Lam \oplus \cA_\Sym$ where first two summands are antihermitian, while the other two -- hermitian. Moreover, being a semisimple subalgebra of the compact algebra $\fu(V^\CC)$, $\cA_\fu$ is compact.

Finally, assume that there is a complex subspace of $V^\CC$
preserved by $\cA_\fu.$ The subspace is necessarily of the form
$W + iW$ with $W\subset V$. An element $f\in\cA_\Lam$ maps
$w+iw'$ to $f(w) + if(w')$ while $is\in i\cA_\Sym$ maps the same to
$-s(w') + is(w).$ Clearly, $f(w),s(w')\in W$ for each $w,w'\in W.$ Hence, $W\subset V$
is fixed by the original $\cA.$ But then $W$ is either empty of $V.$ Thus $V^\CC$ is irreducible.
\end{proof}

Our isotropy represntations will turn out to be described by the following spaces:
$$ \cV_1(\KK) = \sxk,\quad \cV_2(\KK) = \CC\otimes\hxk,\quad \cV_3(\KK) = \CC\otimes\fxk, $$
where $\cV_2(\KK)$ and $\cV_3(\KK)$ are complex vector spaces equipped with  \emph{hermitian}
inner products being sesquilinear extensions of the scalar products on, respectively, $\hxk$ and $\fxk,$ defined in the previous sections. The spcae $\cV_3(\KK)$ is moreover equipped with a 
symplectic form being the linear extension of $\omega.$
The interesting algebras acting on these spaces are:
\begin{eqnarray*}
\cg_1(\KK) &= &\der\hxk \subset \End\cV_1(\KK) \\
\cg_2(\KK) &= &[\str_0(\hxk)]_\fu \subset \fsu(\cV_2(\KK)) \\
\cg_3(\KK) &= &[\der\fxk]_\fu \subset \fsp(\cV_3(\KK),\omega),
\end{eqnarray*}
where $\cg_2(\KK)$ and $\cg_3(\KK)$ are the antihermitian subalgebras of, respectively, $\CC\otimes\str_0(\hxk)$ and $\CC\otimes\der\fxk,$ as described in Lemma \ref{lem-cpx}.

Using the Lemma and examining (anti)symmetry of elements of algebras defined in the previous section, we have the explicit forms:
\begin{eqnarray}
\cg_2(\KK) &=& \der\hxk \oplus iL_{\sxk} \\
\cg_3(\KK) &=& \cH_0(\der\hxk \oplus i\hxk) \oplus \cH_1(\hxk\oplus i\hxk),
\end{eqnarray}
where the maps $\cH_0$ and $\cH_1$ of Lemma \ref{lem-derfts} have been extended by linearity.

We now exponentiate $\cg_n(\KK)$, $n=1,2,3$, to obtain connected subgroups $\cG_n(\KK)\subset\sGL(\cV_n(\KK))$. In particular,
$\cG_2(\KK)$ and $\cG_3(\KK)$ are \emph{compact} forms of the ones described in the previous section, as given in the following table:
\begin{center}
\begin{tabular}{r|ccc}
$\KK$ & $\cG_1(\KK)$ & $\cG_2(\KK)$ & $\cG_3(\KK)$ \\
\hline
$\RR$ & $\sSO(3)$ & $\sSU(3)$ & $\sSp(1)$ \\
$\CC$ & $\sSU(3)$ & $\sSU(3)\times\sSU(3)$ & $\sSU(6)$ \\
$\HH$ & $\sSp(1)$ & $\sSU(6)$ & $\sSO(12)$ \\
$\OO$ & $\sF_4$ & $\sE_{6(-52)}$ & $\sE_{7(-133)}$ 
\end{tabular}\end{center}
with 
\begin{eqnarray*}
\cG_2(\KK)&\subset&\sSU(\cV_2(\KK)) \\ \cG_3(\KK)&\subset&\sSp(\cV_3(\KK),\omega).
\end{eqnarray*}

We now arrive at the main point of this section:
\begin{pro} \label{pro-isoreps} The isotropy representations of the magic square symmetric spaces are extensions of those of $\cG_n(\KK)$ on $\cV_n(\KK)$ with $n=1,2,3$ in, respectively, first, second and third family, i.e.:
\begin{enumerate}
\item{	There exist 
isomorphisms (respectively of Lie algebras and of vector spaces)
	\begin{eqnarray*}
	F_1^\fg &:& \cg_1(\KK) \to \fg(\KK,\CC) \\
	F_1^V &:& \cV_1(\KK) \to V(\KK,\CC)
	\end{eqnarray*}
	such that, for each $A\in\cg_1(\KK)$ and $X\in\cV_1(\KK),$
	$$ [ F_1^\fg(A), F_1^V(X) ] = F_1^V (A(X)). $$}
\item{	There exist
isomorphisms (respectively of Lie algebras and of vector spaces)
	\begin{eqnarray*}
	F_2^\fg &:& \cg_2(\KK) \to \fm(\KK,\CC)\subset\fg(\KK,\HH)  \\
	F_2^V &:& \cV_2(\KK) \to V(\KK,\HH)
	\end{eqnarray*}
	such that, for each $A\in\cg_2(\KK)$ and $X\in\cV_2(\KK),$
	$$ [ F_2^\fg(A), F_2^V(X) ] = F_2^V (A(X)). $$
	Moreover, as Lie algebras, $$\fg(\KK,\HH) = F_2^\fg(\cg_2(\KK)) \oplus \fu(1).$$
	Once we use $F_2^V$ to identify $\CC\otimes\hxk$ with $V(\KK,\HH),$
	a generator $i$ of this extra $\fu(1)\subset\fg(\KK,\HH)$ acts on $\CC\otimes\hxk$
	as $\sqrt{-1}\in\CC.$ 
	In other words, it defines a complex (hermitian) structure, with respect
	to which $\fg(\KK,\HH) \subset \fu(\dim\hxk).$}
\item{	There exist
isomorphisms (respectively of Lie algebras and of vector spaces)
	\begin{eqnarray*}
	F_3^\fg &:& \cg_3(\KK) \to \fm(\KK,\HH)\subset\fg(\KK,\OO) \\
	F_3^V &:& \cV_3(\KK) \to V(\KK,\OO)
	\end{eqnarray*}
	such that, for each $A\in\cg_3(\KK)$ and $X\in\cV_3(\KK),$
	$$ [ F_3^\fg(A), F_3^V(X) ] = F_3^V (A(X)). $$	
	Moreover, as Lie algebras, $$\fg(\KK,\OO) = F_3^\fg(\cg_3(\KK)) \oplus \fsp(1). $$
	Once we use $F_3^V$ to identify $\CC\otimes\fxk$ with $V(\KK,\OO),$
	the generators $i,j,k$ of this extra $\fsp(1)\subset\fg(\KK,\HH)$ act on $\CC\otimes\fxk$
	as $I,J,K$ given explicitly by:
	\begin{equation}\label{eq-ijk}
		I(X) = i X,\quad \langle\bar X, J Y\rangle = \omega(X,Y),\quad \langle\bar X, K Y\rangle = i\omega(X,Y).
	\end{equation}
	In other words, they define a quaternionic (hermitian) structure, with respect
	to which $\fg(\KK,\OO) \subset \fsp(\dim\hxk+1) \oplus \fsp(1).$}
\end{enumerate}
\end{pro}

\begin{cor} \label{cor-reps} $ $
\begin{enumerate}
\item The adjoint representation 
$$\ad : \fg(\KK,\KK') \to \End V(\KK,\KK')$$
is faithful and irreducible.
\item
Let $G(\KK,\KK')$ be the connected subgroup of $\sGL(V(\KK,\KK'))$ obtained by exponentiation
of the image of $\fg(\KK,\KK')$ under the adjoint representation on $V(\KK,\KK').$ Let us moreover
identify $V(\KK,\KK')$ with $\cV_n(\KK)$ using $F^V_n$ with $n=1,2,3$ for, respectively, $\KK'=\CC,\HH,\OO.$
Then: 
\begin{eqnarray*}
G(\KK,\CC) &=& \cG_1(\KK) \\
G(\KK,\HH) &=& \cG_2(\KK)\cdot\sU(1) \\
G(\KK,\OO) &=& \cG_2(\KK)\cdot\sSp(1)
\end{eqnarray*}
where the $\sU(1)$ is generated by the natural complex structure on $\cV_2(\KK)$ and the $\sSp(1)$
is generated by the natural quaternionic structure (induced by $\omega$ and the hermitian inner product) on $\cV_3(\KK).$
\end{enumerate}
\end{cor}



The rest of this section is devoted to proving the Proposition (and may be skipped by the inpatient reader). We first need to complete the discussion of $\der\OO$ started in Lemma \ref{lem-deg}. Recall that in Corollary \ref{cor-ee} we have stated that
the subalgebra $\deg\OO$ is isomorphic to $\fsp(1)\oplus\fsp(1) \simeq \Im\HH \oplus \Im\HH$ via the maps $\cE$ and $\cE'.$ 
To parametrize $\dev\OO,$ we introduce for each $q\in\Im\HH$ a map
$$ \cB_q : \HH \to \dev\OO $$
$$ \cB_q(x) = \cD_{q,q\varphi(x)} $$
for $x\in\HH.$

Recall that $\dev\OO$ maps $\HH$ to $l\HH.$ The action  on $\Im\HH$ is given in a
useful form by the following
\begin{lem} Let $r,s$ be orthogonal unit imaginary quaternions. Then for $x\in\HH$ the following
holds:
\begin{eqnarray*}
 \cB_r(x)(r) &=& 4 r\varphi(x) \\ 
 \cB_s(x)(r) &=& -2 r \varphi(x)
\end{eqnarray*}
\end{lem}
\begin{proof}
First, we have
$$ \cD_{s,s\varphi(x)}(s) = [[s,s\varphi(x)],s] = 2[ s(s\varphi(x)) , s] = -2[\varphi(x),s]
= 4 s\varphi(x), $$
where we used the alternativity of $\OO$ (i.e. antisymmetry of the associator, equivalent
to $a^2 b = a(ab)$) and the fact that elements of $\Im\HH$ anticommute with elements of $l\HH.$
Then, we have
\begin{eqnarray*}
\cD_{s,s\varphi(x)}(r) &=& [[s,s\varphi(x)],r] - 3[s,s\varphi(x),r] \\
&=& 4 r \varphi(x) - 3 r \varphi(x) - 3 s(r(s\varphi(x))) \\
&=& -2r\varphi(x),
\end{eqnarray*}
where we used the fact that orthogonal imaginary quaternions anticommute, and checked explicitly
that $s(r(sx'))) = rx' $ for $x'\in\OO.$ This proves the Lemma.
\end{proof}

It is clear that the maps $\cB_i,$ $\cB_j$ and $\cB_k$ cannot be independent.  Indeed, one can check that $ \cB_i + \cB_j + \cB_k = 0 $ (note that the map $q\mapsto \cB_q$ is quadratic in $q,$ so that
the latter sum is simply a polarised version of $\cB$ evaluated on the $\fsp(1)$-invariant quadratic element).
We choose two combinations of $\cB_i$ and $\cB_j$, adapted to the goal of constructing the intertwiner $F^V_3:$
\begin{lem}
Let us define two maps $\cB',\cB'' : \HH \to \dev\OO,$
$$ \cB' = \cB_i - \cB_j,\quad \cB'' = 3 \cB_k = -3 (\cB_i+\cB_j). $$

Then the map
$$ \cB'\times\cB'' : \HH \oplus \HH \to \dev\OO $$
is a bijection.
\end{lem}
The proof is by direct calculation with help of computer algebra.
Finally, the adjoint action of $\deg\OO$ on $\dev\OO$ is given by the following
\begin{lem}
Let us use $\cB'\times\cB''$ to identify $\dev\OO$ with $\HH\oplus\HH = \RR^2 \otimes\HH. $
Then the adjoint action of $\deg\OO$ on $\dev\OO \simeq \RR^2\otimes\HH$ is:
\begin{eqnarray*}
\ad_{\cE(i)} &=& \left(\begin{matrix}0&-3\\-1&-2\end{matrix}\right) \otimes L_i \\
\ad_{\cE(j)} &=& \left(\begin{matrix}0&3\\1&-2\end{matrix}\right)\otimes L_j \\
\ad_{\cE(k)} &=& \left(\begin{matrix}-3&0\\0&1\end{matrix}\right)\otimes L_k 
\end{eqnarray*}
and, for $q\in\Im\HH,$
$$ \ad_{\cE'(q)} = \id \otimes R_{\bar q}. $$
\end{lem}
\begin{proof}
Let us compute the commutator of $\cE(s)$ and $\cB_r(x)$ for $s,r$ being unit imaginary quaternions:
\begin{eqnarray*}
[ \cE(s), \cD_{r,r\varphi(x)} ] &=& \cD_{[s,r],r\varphi(x)} + \cD_{r,[s,r]\varphi(x)}
+ \cD_{r,r\varphi(sx)} 
\\ &=& 
2 \cD_{rs,rs\varphi(sx)} - \cD_{r,r\varphi(sx)} + 2\langle r,s\rangle \cD_{r,r\varphi(rx)}, 
\end{eqnarray*}
where we used the equivariance of $\cD$ and the identity $[r,s] = 2 rs + 2 \langle r,s\rangle. $
In particular, 
\begin{eqnarray*}
[ \cE(s), \cB_s(x) ] &=& \cB_s(sx) \\ {}
[ \cE(s), \cB_r(x) ] &=& 2\cB_{sr}(sx) - \cB_r(sx) \quad\textrm{for}\ s\bot r.
\end{eqnarray*}
The commutator of $\cE'(s)$ and $\cB_r(x)$ is simply 
$$[\cE'(s),\cD_{r,r\varphi(x)}] = \cD_{r,r\varphi(x\bar s)} = \cB_r(x\bar s).$$
This proves the Lemma.
\end{proof}


We are now ready to prove the Proposition.

\begin{proof}[Proof of Proposition \ref{pro-isoreps}] $ $
We first notice, that it is enough to check the intertwining property 
$$ [ F_{1,2,3}^\fg(A), F_{1,2,3}^V(X) ] = F_{1,2,3}^V (A(X)) $$
-- then
the fact that $F_{1,2,3}^V$ is an isomorphism between the vector spaces of two faithful representations
already implies that $F_{1,2,3}^\fg$ is an isomorphism of Lie algebras.

In the first family, the algebra is $\cg_1(\KK) = \der\hxk$ and
the representation space is $\cV_1(\KK) = \sxk.$
On the other hand, since $\der\CC$ is trivial, 
we have $V(\KK,\CC) = i\RR \otimes \sxk,$
which is naturally \emph{identified} with $\sxk$
and $\fg(\KK,\CC) = \der\hxk.$
The maps $F_1^V$ and $F_1^\fg$ can be simply set to identity,
and point 1 of the Proposition follows.

In the second family, the algebra is $\cg_2(\KK) = \der\hxk \oplus i L_{\sxk}$ and
the representation space is $\cV_2(\KK)=\CC\otimes\hxk.$ On the other hand, since
$$ \deg\HH\oplus\dev\HH = \Span\{\ad_i\} \oplus \Span\{\ad_j, \ad_k\}, $$
we have
$$ \fg(\KK,\HH) = \der\hxk \oplus \Span\{\ad_i\} \oplus i\RR \otimes \sxk $$
$$ V(\KK,\HH) = \Span\{\ad_j,\ad_k\} \oplus j\CC \otimes \sxk. $$
The subspace $\fm(\KK,\CC)\subset\fg(\KK,\HH)$ is $\der\hxk\oplus i\RR\otimes\sxk.$
We now define
$$ F^\fg_2(D) = D,\quad F^\fg_2(i L_X) = i\otimes X $$
for $D\in\der\hxk,\ X\in\sxk,$  and
$$ F^V_2(z\otimes X) = \frac{X_0}{2} \ad_{zj} + zj \otimes X_1 $$
for $z\in\CC$ and $\hxk\ni X = X_0 1 + X_1$ where $X_0\in\RR,\ X_1\in\sxk.$

Now, we check that, for $D\in\der\hxk,\ X\in\sxk,$ 
$z\in\CC$  and $\hxk\ni Y= Y_0 1+Y_1$ where $Y_0 \in\RR,\ Y_1\in\sxk,$ the
intertwining property holds:
\begin{eqnarray*}
[ F^\fg_2(D), F^V_2(z\otimes Y) ]
&=& [D, \frac{Y_0}{2}\ \ad_{zj} + zj\otimes Y_1 ] \\
&=& zj\otimes D(Y_1) \\
&=& F^V_2( z\otimes D(Y_1) ) \\
&=& F^V_2( D(z\otimes Y) ). \\ {}
[ F^\fg_2(iL_X), F^V_2(z\otimes Y) ] 
&=& [i\otimes X, \frac{Y_0}{2}\ \ad_{zj} + zj\otimes Y_1 ] \\
&=& izj \otimes Y_0 X + izj\otimes (X\times Y_1) + \frac{2}{12} h(X,Y_1) \ad_{izj} \\
&=& \frac{1}{6} h(X,Y_1) \ad_{izj} + izj\otimes (Y_0 X + X\times Y_1) \\
&=& F^V_2( iz\otimes \left[\frac{1}{3}h(X,Y_1) + (Y_0 X + X\times Y_1)\right]  ) \\
&=& F^V_2( iL_X(z\otimes Y) ).
\end{eqnarray*}

It remains to notice that the 
the generator of the remaining subspace $\RR\subset\fg(\KK,\HH),$
namely $\frac{1}{2}\ad_i,$ acts as multiplication by $\sqrt{-1}$ in $\CC\otimes
\hxk:$
\begin{eqnarray*}
[\frac{1}{2}\ad_i, F^V_2(z\otimes Y) ] &=& [\frac{1}{2}\ad_i, \frac{Y_0}{2}\ \ad_{zj} + zj\otimes Y_1 ] \\
&=& \frac{Y_0}{2}\ \ad_{izj} + izj\otimes Y_1 \\
&=& F^V_2(iz \otimes Y).
\end{eqnarray*}
Thus point 2 of the Proposition follows.

In the third family, the algebra is $$\cg_3(\KK) = \cH_0(\der\hxk\oplus i\hxk) \oplus \cH_1(\hxk\oplus i\hxk)$$ and 
the representation space is $$\cV_3(\KK) = \CC\otimes\fxk \simeq (\RR\oplus\hxk)
\otimes \CC^2.$$

On the other hand, we have \begin{eqnarray*}
 \fg(\KK,\OO) &=& \der\hxk \oplus \deg\OO \oplus (\Im\HH \otimes \sxk) \\
 V(\KK,\OO) &=& \dev\OO \oplus (l\HH\otimes\sxk),\end{eqnarray*}
where the derivations are identified with certain quaternionic spaces:
\begin{eqnarray*}
\deg\OO &=& (\cE\times\cE')(\Im\HH\oplus\Im\HH) \\
\dev\OO &=& (\cB'\times\cB'')(\HH\oplus\HH). \end{eqnarray*}

To relate the latter spaces to the former ones, we note that
the space $\CC^2$ becomes an irreducible left $\HH-$module with the 
action assigning to $q\in\HH$ an operator $\tau_q : \CC^2\to\CC^2$ given by
$$
\tau_i = \left(\begin{matrix}0&i\\i&0\end{matrix}\right),\quad
\tau_j = \left(\begin{matrix}0&-1\\1&0\end{matrix}\right),\quad
\tau_k = \left(\begin{matrix}i&0\\0&-i\end{matrix}\right),\quad \tau_1 = \id.
$$ As its dimension is four, there exists an intertwiner
$$ \psi : \CC^2\to\HH $$
$$ \psi(\tau_q \xi) = L_q \psi(\xi) $$
for $q\in\HH,\xi\in\CC^2,$ where $L_q$ is the left multiplication by $q$ in $\HH.$ 

As the reader may expect, $\psi$ will relate the $\CC^2$ factor in $\CC\otimes\fxk$ to the quaternionic
factors in $V(\KK,\OO).$ One also needs to note that the $\hxk$ spaces on the FTS side
must be decomposed into $\RR\oplus\sxk$ as the Tits construction features only the traceless
subspaces. Clearly, the traces are then to be identified with certain elements of $\der\OO.$

It is convenient to identify the two representation
spaces with an intermediate one, namely $ (\RR\oplus\RR\oplus\sxk)\otimes \HH,$ so that we have the isomorphisms
$$ \mu : (\RR\oplus\hxk)\otimes\CC^2 \to (\RR\oplus\RR\oplus\sxk) \otimes \HH $$
$$  (x,X) \otimes \xi \mapsto (x,X_0,X_1) \otimes \psi(\xi) $$
for $X=X_0 1 + X_1$ where $x,X_0\in\RR,\ X_1\in\sxk,$
and
$$ \nu : (\RR\oplus\RR\oplus\sxk) \otimes \HH \to V(\KK,\OO) $$
$$  (y',y'',Y) \otimes q \mapsto \cB'(y'q) + \cB''(y''q) - 12\varphi(q) \otimes Y $$
for $y,y'\in\RR,\ Y\in\sxk$ and $q\in\HH.$
We define the intertwiner to be their composition,
$$ F_3^V = \nu \circ \mu. $$ 
Then, for $A=A_0 1 + A_1 $ where $A_0 \in \RR,\ A_1\in\sxk,$ and
$B=B_0 1 + B_1,\ C=C_01 + C_1$ in the same manner, and $D\in\der\hxk,$
we define
$$ F_3^\fg : \cg_3(\KK) \to \der\hxk \oplus \deg\OO \oplus
(\Im\HH\otimes\sxk) $$
\begin{eqnarray*}
F_3^\fg(\cH_0(D,iC) + \cH_1(A,iB)) &=& D  - B_0 \cE(i) + A_0 \cE(j) + C_0 \cE(k) \\
& &\quad + i \otimes 2B_1 - j \otimes 2A_1 +k\otimes C_1. \end{eqnarray*}
Clearly, the image of $F_3^\fg$ is
$$ \fg(\KK,\OO) \supset \fm(\KK,\HH) = \der\hxk \oplus \cE(\Im\HH) \oplus (\Im\HH\otimes\sxk), $$
so that $\fg(\KK,\OO) = \fm(\KK,\HH) \oplus \cE'(\Im\HH). $

To check the intertwining property, we must find the action of both $\cg_3(\KK)$ and
$F_3^\fg(\cg_3(\KK))=\fm(\KK,\HH)\subset\fg(\KK,\OO)$ 
on $(\RR\oplus\RR\oplus\sxk)\otimes \HH$ (via the identifications $\mu,\nu$)
and show that they agree. 

As the action on the FTS side contains the maps $L^\bullet$ and $L$ on $\hxk,$
we have to decompose them with respect to $\hxk=\RR\oplus\sxk$
and express in terms of $L^\times,$ the multiplication map corresponding to the product $\times : \sxk
\times\sxk\to\sxk.$ 
Recalling $\langle X\bullet Y,Z\rangle  = 3N(X,Y,Z)$ and using the defining indentity for $N$ (Lemma \ref{lem-tsn}), we find that
$$ X\bullet Y = X\circ Y + \frac{1}{2}\tr X\ \tr Y - \frac{1}{2}[ \langle X,Y \rangle + (\tr Y)X + (\tr X)Y ] $$
for $X,Y\in\hxk.$
Now, regarding the expressions on the r.h.s as elements of $\RR \oplus \sxk,$ we have:
\begin{eqnarray*}
X\circ Y &=& \left( X_0 Y_0 + \frac{1}{3}\langle X_1,Y_1 \rangle ,\quad  X_1 \times Y_1 + X_0 Y_1 + Y_0 X_1 \right) \\
X\bullet Y &=& \left( X_0 Y_0 - \frac{1}{6}\langle X_1,Y_1 \rangle ,\quad  X_1 \times Y_1 - \frac{1}{2} X_0 Y_1
- \frac{1}{2} Y_0 X_1 \right)
\end{eqnarray*}
and
$$ \langle X,Y \rangle  = 3 X_0 Y_0 + \langle X_1, Y_1 \rangle $$
for $X=X_0 1 + X_1$ and $Y=Y_0 1 + Y_1$ where $X_0,Y_0\in\RR$ and $X_1,Y_1\in\sxk.$

Using these formulas and the definitions of $\cH_0$ and $\cH_1,$ we find that the element
$$ \cH_0(D,iC) + \cH_1(A,iB) \in \cg_3(\KK)$$
corresponds via $\mu$ to the following 
operator on $(\RR\oplus\RR\oplus\sxk)\otimes\HH:$
\begin{eqnarray}
\label{big-op}
\mu\circ [ \cH_0(D,iC) + \cH_1(A,iB) ] \circ \mu^{-1} & =  \nonumber \\
\left[ -B_0
\left(\begin{matrix}  0 & -3 & \\ -1 & -2 & \\ & & 1 \end{matrix}\right) 
+
\left(\begin{matrix}  0&0 & h(B_1) \\ 0&0& -\frac{1}{3}h(B_1) \\ B_1 & -B_1 & 2 L^\times_{B_1} \end{matrix}\right)
\right] &\otimes&
L_i 
\nonumber\\ + \left[
 A_0
\left(\begin{matrix} 0& 3 & \\ 1&-2 & \\ & & 1  \end{matrix}\right) 
+
\left(\begin{matrix}  0&0 & h(A_1) \\ 0&0& \frac{1}{3}h(A_1) \\ A_1 & A_1 & -2 L^\times_{A_1} \end{matrix}\right) 
\right] &\otimes&
L_j 
 \\ 
+\left[
C_0
\left(\begin{matrix} -3 & & \\ & 1 & \\ & & 1 \end{matrix}\right) 
+ \left(\begin{matrix} 0& &   \\ & 0 &\frac{1}{3}h(C_1) \\ &R_1 & L^\times_{C_1} \end{matrix}\right) 
\right]
&\otimes&
L_k 
\nonumber \\ +
\left(\begin{matrix} 0 & & \\ &0& \\ & & D \end{matrix}\right) &\otimes&
1. 
\nonumber
\end{eqnarray}

We now have to find the action of the image of $\cH_0(D,iC)+\cH_1(A,iB)$ under $F_3^\fg,$ namely
$$  D  - B_0 \cE(i) + A_0 \cE(j) + C_0 \cE(k) 
+ i \otimes 2B_1 - j \otimes 2A_1 +k\otimes C_1, $$
an element of $\fg(\KK,\OO),$
while $V(\KK,\OO)$ is identified with $(\RR\oplus\RR\oplus\sxk)\otimes\HH$ via $\nu.$

Acting with the $\der\hxk\oplus\deg\OO$ part is easy, as we have already found all
the necessary formulas while discussing the octonionic derivations
in the present section. We still need to compute the following commutator
in $\fm(\KK,\OO)$:
\begin{eqnarray*}
[ r \otimes A_1,\ \cB'(q') + \cB''(q'') + \varphi(q)\otimes X_1 ] 
&=& -\frac{1}{12} h(A_1,X_1) \cB_r(rq)  \\
&+& \varphi(rq) \otimes (A_1\times X_1) \\ &-& [\cB'(q')(r) + \cB''(q'')(r)] \otimes A_1
\end{eqnarray*}
for $q,q',q''\in\HH,$ $A_1,X_1 \in \sxk$ and $r$ a unit imaginary quaternion.
Having all that done, we arrive at precisely the same operator as in (\ref{big-op}):
\begin{eqnarray*}
\nu^{-1}\circ \ad_{[
 D  - B_0 \cE(i) + A_0 \cE(j) + C_0 \cE(k) 
+ i \otimes 2B_1 - j \otimes 2A_1 +k\otimes C_1
]} \circ \nu & = \\
\left[ -B_0
\left(\begin{matrix}  0 & -3 & \\ -1 & -2 & \\ & & 1 \end{matrix}\right) 
+
\left(\begin{matrix}  0&0 & h(B_1) \\ 0&0& -\frac{1}{3}h(B_1) \\ B_1 & -B_1 & 2 L^\times_{B_1} \end{matrix}\right)
\right] &\otimes&
L_i 
\\ + \left[
 A_0
\left(\begin{matrix} 0& 3 & \\ 1&-2 & \\ & & 1  \end{matrix}\right) 
+
\left(\begin{matrix}  0&0 & h(A_1) \\ 0&0& \frac{1}{3}h(A_1) \\ A_1 & A_1 & -2 L^\times_{A_1} \end{matrix}\right) 
\right] &\otimes&
L_j 
 \\ 
+\left[
C_0
\left(\begin{matrix} -3 & & \\ & 1 & \\ & & 1 \end{matrix}\right) 
+ \left(\begin{matrix} 0& &   \\ & 0 &\frac{1}{3}h(C_1) \\ &R_1 & L^\times_{C_1} \end{matrix}\right) 
\right]
&\otimes&
L_k 
\\ +
\left(\begin{matrix} 0 & & \\ &0& \\ & & D \end{matrix}\right) &\otimes&
1. 
\end{eqnarray*}
so that
$$\ad_{F_3^\fg[\cH_0(D,iC)+\cH_1(A,iB)]} \circ (\nu\circ\mu)  =
(\nu\circ\mu)\circ [ \cH_0(D,iC) + \cH_1(A,iB) ] 
$$
as claimed.

Finally, we need to check the action of the remaining $\cE'(\Im\HH) \simeq\fsp(1)$
on $\CC\otimes\fxk.$ Using once again the formulas for octonionic derivations, we find that
$\cE'(q)$ acts on $(\RR\oplus\RR\oplus\sxk) \otimes \HH$ via
$$ \id \otimes R_{\bar q}$$
where $R$ is the right multiplication map on the quaternions.

Thus the corresponding algebra $\fsp(1)$ acting on $(\RR\oplus\hxk)\otimes\CC^2$ is
generated by
three complex structures, which we choose to be
\begin{eqnarray*}
I &=& \id \otimes (\psi^{-1}\circ R_i \circ\psi) \\
J &=& \id \otimes (\psi^{-1}\circ R_{-j} \circ\psi) \\
K &=& \id \otimes (\psi^{-1}\circ R_k \circ\psi). 
\end{eqnarray*}
These can be brought to the form demanded by the Proposition once the intertwiner $\psi:\CC^2\to\HH$ is
set to:
$$ \psi(1,1) = 1,\quad \psi(i,i) = i,\quad \psi(-1,1) = j,\quad \psi(i,-i) = k. $$
Then we indeed have, for $x\in\RR,\ X\in\hxk$ and $z_1,z_2\in\CC:$
\begin{eqnarray*}
I((x,X)\otimes (z_1,z_2)) &=& (x,X)\otimes (i z_1, i z_2) \\
J((x,X)\otimes (z_1,z_2)) &=& (x,X)\otimes (\bar z_2, -\bar z_1) \\
K((x,X)\otimes (z_1,z_2)) &=& (x,X)\otimes (i\bar z_2, -i\bar z_1).
\end{eqnarray*}
Compared with the formulas for the inner product and symplectic form, this completes the proof.
\end{proof}

\section{Constructing symmetric invariants}
Now that we have found that the isotropy representations of our symmetric spaces are in fact the natural representations of certain automorphism or isotropy groups (up to complexification), extended by the action of a natural complex or quaternionic structure, we can get the symmetric invariants almost for free.


While the situation is clear in case of the first family, the spaces 
$$ \cV_2(\KK) = \CC\otimes\hxk,\quad \cV_3(\KK) = \CC\otimes\fxk $$
of other
two families can be viewed either over $\RR$ or over $\CC$. Since we will ultimately be dealing with
real geometry, the next lemma and proposition are referring to the \emph{real} vector spaces $\cV_{2,3}(\KK)$ (although the proofs will in turn utilize the complex point of view).
These are naturally equipped with a positive-definite scalar product, denoted in both cases by $g$ and defined as follows:
$$ g : \Sym^2 \cV_2(\KK) \to \RR $$
$$ g(z\otimes X,z\otimes X) = \langle z,z \rangle \langle X,X \rangle $$
for $z\in\CC,\ X\in\hxk$ and
$$ g : \Sym^2 \cV_3(\KK) \to \RR $$
$$ g(z\otimes X,z\otimes X) = \langle z,z \rangle \langle X,X \rangle $$
for $z\in\CC,\ X\in\fxk.$ This is the scalar product the word `orthogonal' will be referring to.
The spaces $\cV_3(\KK)$ are moreover left $\HH-$modules, with the action of a quaternion $q\in\HH$
being
$ L_q \in \End\cV_3(\KK),$ such that
$$ L_i = I,\ L_j = J,\ L_k = K,\ L_1 = \id, $$
where $I,J,K$ are the maps defined in point 3 of Proposition \ref{pro-isoreps}.

Observe that the problem is to turn the invariants of $\cG_{2,3}(\KK)$ into
maps invariant under the action of the complex or quaternionic structure. This is fairly simple
in the second family. The same task for the third family invariants is fulfilled with help of the following lemma, which produces an $\HH-$equivariant tensor on $\CC\otimes\fxk$ from a symmetric tensor on $\fxk:$
\begin{lem} \label{lem-tl} Given a tensor $t \in \Sym^p \fxk^*,$ define the map
$$ t_L : \Sym^p \cV_3(\KK) \to \Sym^p\HH $$
to be the dual, with respect to the natural scalar products, of the map
$$ t_L^* : \Sym^p \HH \otimes \Sym^p \cV_3(\KK) \to \RR $$
$$ t_L^*(q^{\otimes p}, Z^{\otimes p}) = \tilde t((L_q Z)^{\otimes p}), $$
where $\tilde t \in \Sym^p \cV_3(\KK)^*$ is defined by
$$ \tilde t(Z^{\otimes p}) = \Re(z^p) \cdot t(X^{\otimes p}) $$
for $Z=z\otimes X \in \cV_3(\KK),$ where $z\in\CC$ and $X\in\fxk.$

Then $t_L$ is equivariant in the following sense:
$$ t_L \circ L_q = R_{\bar q} \circ t_L\quad \forall q\in\HH. $$
\end{lem}
\begin{proof}
Let $x,w\in\HH$ and $Z=z\otimes X\in\cV_3(\KK).$ We have
\begin{eqnarray*}
\langle t_L((L_x Z)^{\otimes p}), w^{\otimes p} \rangle
&=& t^*_L (w^{\otimes p}, (L_x Z)^{\otimes p}) \\
&=& \tilde t((L_{wx} Z)^{\otimes p})  \\
&=& t^*_L ((wx)^{\otimes p}, Z^{\otimes p}) \\
&=& \langle t_L(Z^{\otimes p}), (wx)^{\otimes p} \rangle,
\end{eqnarray*}
where $\langle\cdot,\cdot\rangle$ denotes the natural scalar product on $\Sym^p \HH$
(induced by the one on $\HH$).
Nondegeneracy of the latter implies
$$ t_L((L_x Z)^{\otimes p}) = R^*_x t_L(Z^{\otimes p}), $$
where $R^*_x = R_{\bar x}$ is dual to $R_x$ with respect to the scalar product.
\end{proof}

Finally, the invariants defining $G(\KK,\KK')$ are given by the following
\begin{pro} \label{pro-invs}
Let us identify $V(\KK,\KK')$ with $\cV_n(\KK),$ where $n=1,2,3$ for, respectively, $\KK'=\CC,\HH,\OO,$ as in Corollary \ref{cor-reps}; 
consider all of them as \emph{real} vector spaces.
Let $c_1,c_2,c_3$ be some nonzero constants (introduced to simplify
further formulas). Then:
\begin{enumerate}
\item $G(\KK,\CC)$ is the isotropy group of the tensor $$\Upsilon:\Sym^3 \cV_1(\KK) \to \RR $$
$$ \Upsilon(X,X,X) = c_1 \langle X\times X,X \rangle $$
for $X\in\cV_1(\KK).$
\item $G(\KK,\HH)$ is a connected component of
the orthogonal isotropy group of the tensor $$\Xi : \Sym^6 \cV_2(\KK) \to \RR $$
$$ \Xi(Z,Z,Z,Z,Z,Z) = c_2^2 |z|^6 (\det X)^2 $$
for $Z = z\otimes X\in\cV_2(\KK),$ where $z\in\CC$ and $X\in\hxk.$
\item $G(\KK,\OO)$ is a connected component of
the orthogonal isotropy group of the tensor $$\mho : \Sym^8 \cV_3(\KK) \to \RR$$
$$ \mho(Z,Z,Z,Z,Z,Z,Z,Z) = c_3^2 \| \cQ_L(Z,Z,Z,Z) \|^2$$
for $Z\in\cV_3(\KK),$ where $\|\cdot\|$ is the natural norm on $\Sym^4 \HH$ (induced by the scalar product) and $\cQ_L$ is defined as in Lemma \ref{lem-tl}.
\end{enumerate}
\end{pro}

The rest of this section is occupied by the proof.
We first need to set up some notation.

\begin{note}\label{not-cpx}
Consider on $\cV_2(\KK)=\CC\otimes\hxk$ and $\cV_3(\KK)=\CC\otimes\fxk$  a complex structure given by multiplication by $\sqrt{-1}$ in the $\CC$ factor. 
(corresponding on $\cV_3(\KK)$ to $I$ defined in point 3 of Proposition \ref{pro-isoreps}).
In the following, it will be \emph{the} distinguished complex structure on these
spaces, the terms \emph{complex-linear} and \emph{antilinear} refer to.

Let us from now on assume $n=2,3$ and $\KK$ to be fixed. 
The (complex) space of
complex-valued oneforms on $\cV_n(\KK)$ decomposes into the complex-linear and antilinear part, which we denote using, respectively, $\V^*$ and $\bar \V^*:$
\begin{equation} \CC\otimes\cV^*_n(\KK) = \V^* \oplus \bar \V^*. \label{vbarv} \end{equation}
One extends the decomposition to complex-valued alternating and symmetric forms, so that
\begin{eqnarray*}
\Lam^p_\CC \cV^*_n(\KK) &=& \bigoplus_{r+s=p} \Lam^{r,s} \\
\Sym^p_\CC \cV^*_n(\KK) &=& \bigoplus_{r+s=p} \Sym^{r,s}
\end{eqnarray*}
where the elements of $\Lam^{r,s}$ (resp. $\Sym^{r,s}$)
are alternating (resp. symmetric) forms expressed as sums of expressions
linear in $r$  and antilinear in $s$ arguments.
We shall make use of the natural isomorphisms
\begin{eqnarray}
\Lam^{r,s} &\simeq& \Lam^{r,0} \otimes \Lam^{0,s} \nonumber \\ \label{sym-split}
\Sym^{r,s} &\simeq& \Sym^{r,0} \otimes \Sym^{0,s}
\end{eqnarray}
defined as orthogonal projections in $\otimes^{r+s} [\CC\otimes\cV^*_n(\KK)].$

The duals to $\V^*$ and $\bar \V^*$ are the genuinely complex spaces $\V$ and $\bar \V,$ and there is a complex-linear isomorphism
$ \cV_n(\KK) \simeq \V$ and an anti-isomorphism  $\bar\cdot : \V\to\bar \V,$ i.e. complex conjugation. One can view $\V$ (resp. $\bar\V$) as the space of linear (resp. antilinear)
maps from $\CC$ to $\cV_n(\KK)$.

The hermitian inner product on $\cV_n(\KK)$ is represented by a tensor
$ h \in \bar \V^* \otimes \V^*,$ inverse to $h^{-1} \in \V \otimes \bar \V:$
$$ h(Z,Z') = \bar z z' \cdot \langle X,X' \rangle\quad\textrm{for}\ Z=z\otimes X,\ Z'=z'\otimes X' $$
and defines isomorphisms $\V\simeq\bar \V^*$ and $\V^*\simeq\bar \V.$
We now introduce the abstract index notation with lower latin indices $a,b,\dots$ indexing copies of $\CC\otimes \cV_n^*(\KK)$ and unbarred and barred greek indices, resp. $\alpha,\beta,\dots$ and $\bar\alpha,\bar\beta,\dots$, indexing copies of resp. $\V^*$ and $\bar \V^*.$ Upper indices of the same kind naturally index dual spaces. Greek indices corresponding to latin ones denote subspaces.
In this manner, the real scalar product $g$ defined previously, is
$$ g_{ab}\in\Sym^{1,1}_{ab},\quad g_{ab} = h_{\bar\alpha\beta} + h_{\bar\beta\alpha}. $$
Via the scalar product, we have
$ \fso(\cV_n(\KK)) \simeq \Lam^2 \cV^*_n(\KK), $
so that
$$ \fso_\CC(\cV_n(\KK)) \simeq \Lam^{2,0} \oplus \Lam^{0,2} \oplus \Lam^{1,1}. $$
In particular, $\Re\Lam^{1,1} \simeq \fu(\cV_n(\KK)).$ The hermitian product $h_{\bar\alpha\beta}$ and
the inverse
$h^{\alpha\bar\beta}$ are used to lower and raise indices. The generator of $\fu(1)$ corresponding
to the complex structure is 
$$ \theta_{ab} = ih_{\bar\alpha\beta} - ih_{\bar\beta\alpha} \in \fu(1)\subset \Re\Lam^{1,1}.$$

On $\cV_3(\KK)$ there is moreover a distinguished two-form, denoted by slight abuse of notation by
$$  \omega_{\alpha\beta} + \bar\omega_{\bar\alpha\bar\beta} \in \Re (\Lam^{2,0} \oplus \Lam^{0,2}), $$
where $ \omega_{\alpha\beta} \in \Lam^{2,0}$ is a linear extension 
of the symplectic form $\omega$ on $\fxk:$
$$ \omega(Z,Z') = zz' \omega(X,X') $$
for $Z=z\otimes X$, $Z'=z'\otimes X'$, $z,z'\in\CC,$ and $X,X'\in\fxk.$ The subgroup of $\sU(\cV_3(\KK))$
preserving $\omega$ is $\sSp(\cV_3(\KK),\omega),$ with the Lie algebra $\fsp(\cV_3(\KK),\omega) \subset \Re\Lam^{1,1}$ consisting of maps $E_{ab} = b_{\bar\alpha\beta} + \bar b_{\alpha\bar\beta}$
such that $ b^{\mu}{}_{[\alpha} \omega_{\beta]\mu} = 0$
and $b_{\bar\alpha\beta} = - \bar b_{\beta\bar\alpha}.$

In general, $\omega$ defines an isomorphism
$$ \Lam^{1,1} \to \V^*\otimes \V^* = \Sym^{2,0} \oplus \Lam^{0,2} $$
\begin{equation}\label{omid}
\Lam^{1,1}_{ab} \ni 
b_{\bar\alpha\beta} - b_{\bar\beta\alpha} \mapsto 
b^\mu{}_\gamma \omega_{\delta\mu} 
\in \V^*_\gamma \otimes \V^*_\delta, \end{equation}
such that $$\CC\otimes\fsp(\cV_3(\KK),\omega) \simeq \Sym^{2,0}.$$

The complex structures of Proposition \ref{pro-isoreps} are $I =\theta \in \fu(1)$ and $J,K \in
\Re(\Lam^{2,0}\oplus\Lam^{0,2}):$
$$
I_{ab} = \theta_{ab},\quad
J_{ab} = \omega_{\alpha\beta} + \bar\omega_{\bar\alpha\bar\beta},\quad
K_{ab} = -i\omega_{\alpha\beta} + i\bar\omega_{\bar\alpha\bar\beta}, $$
and $\omega$ itself satisfies $\omega_{\alpha\mu} \bar\omega^\mu{}_{\bar\beta} = - h_{\bar\beta\alpha}.$ These three generate \emph{the} $\fsp(1).$
\end{note}

We can now start proving the Proposition.
\begin{proof}[Proof of points 1 \& 2 of Proposition \ref{pro-invs}] $ $
\begin{enumerate}
\item
In the first family, the group $G(\KK,\CC)$ is simply $\Aut\hxk$ acting in the natural way on $\cV_1(\KK) = \sxk.$ Now, the automorphisms of $\hxk = \RR \oplus \sxk$ automatically preserve the trace, the latter decomposition and the scalar product on both $\hxk$ and $\sxk.$ They moreover act faithfully
on the latter space. It thus follows that they
are the isotropy group of the cubic form  $X \mapsto \langle X, X\circ X\rangle$ on $\hxk,$ or
-- equivalently -- of the cubic form $X \mapsto \langle X, X\times X\rangle$ on $\sxk.$ But the latter
is, up to the constant $c_1\neq 0,$ the tensor $\Upsilon.$ 
\item
In the second family, we consider a rescaled linear extension of the determinant,
$$ \Lambda_{\alpha\beta\gamma}\in\Sym^{3,0}_{abc},\quad \Lambda(Z,Z,Z) = c_2 z^3 \det X $$
and its complex conjugate, the antilinear cubic
$$ \bar\Lambda_{\bar\alpha\bar\beta\bar\gamma}\in\Sym^{0,3}_{abc},\quad \bar\Lambda(Z,Z,Z) = c_2 \bar z^3 \det X $$
for $Z=z\otimes X \in \cV_2(\KK) = \CC\otimes\hxk.$ Our tensor $\Xi$ is then
$$ \Xi = \Lam \cdot \bar\Lam \in \Re\Sym^{3,3}, $$
projecting under the isomorphism (\ref{sym-split}) onto a multiple of
$$ \Lam_{\alpha\beta\gamma} \bar\Lam_{\bar\delta\bar\epsilon\bar\phi} \in \Sym^{3,0}_{abc} \otimes
\Sym^{0,3}_{def} $$
by a combinatorial factor.
Let now $E\in\fso(\cV_2(\KK))$ be given by 
$$ E = A + \bar A + B,\quad A\in\Lam^{2,0},\ B\in\Re\Lam^{1,1}. $$
Then 
\begin{eqnarray*}(A+\bar A)(\Xi) &\in& \Sym^{4,2}\oplus\Sym^{2,4}, \\
B(\Xi) &\in& \Sym^{3,3} \end{eqnarray*}
and as such must both vanish independently if $E$ is to preserve $\Xi.$

The projection of $(A+\bar A)(\Xi)$ onto $\Sym^{4,0}_{abcd} \otimes \Sym^{0,2}_{ef}$ is
proportional to
$$ \Lambda_{(\alpha\beta\gamma} A^{\bar\mu}{}_{\delta)} \bar\Lambda_{\bar\mu\bar\epsilon\bar\phi}, $$ vanishing iff $A=0,$ since $$\Lambda(X,\cdot,\cdot) = 0 \iff \LF_X = 0 \iff X=0.$$
The projection of $B(\Xi)$ onto $\Sym^{3,0} \otimes \Sym^{0,3}$
is proportional to $$B(\Lambda) \otimes \bar\Lambda + \Lambda \otimes \overline{B(\Lambda)}, $$
vanishing iff $B(\Lambda) = \lambda\ \Lambda$ with $\lambda \in i\RR.$ Thus, if $E$ preserves $\Xi,$
then
$$B^0_{ab} = B_{ab} - i\lambda (h_{\beta\bar\alpha} - h_{\alpha\bar\beta}) = B_{ab} - \lambda\theta_{ab} $$
must preserve $\Lam.$ As the latter is the linear extension of the determinant, it follows that
$ B^0 \in \CC\otimes\str_0(\hxk),$ and finally $$B = B^0 + \lambda\theta \in \cg_2(\KK) \oplus \fu(1).$$
Thus $E$ preserves $\Xi$ iff $ A=0$ and $B\in\fg(\KK,\HH),$ which proves point 2 of the Proposition.
\end{enumerate}\end{proof}

Proving the last point is considerably more involved.
We consider a rescaled linear extension of the quartic on $\fxk,$
$$ q_{\alpha\beta\gamma\delta} \in \Sym^{4,0}_{abcd},\quad q(Z,Z,Z,Z) = c_3 z^4 \cQ(X,X,X,X) $$
and its complex conjugate, the antilinear quartic
$$ \bar q_{\bar\alpha\bar\beta\bar\gamma\bar\delta}\in\Sym^{0,4}_{abcd},\quad \bar q (Z,Z,Z,Z) = c_3 \bar z^4 \cQ(X,X,X,X) $$
for $Z=z\otimes X \in \cV_3(\KK) = \CC\otimes\hxk.$ 

As the tensor $\mho$ has been defined in a rather obscure form, we need some more convenient formula.
Note first, that, being a symmetric rank eight tensor, invariant with respect to $I,$
it must be in $\Sym^{4,4}$ (invariance follows from Lemma \ref{lem-tl}). We now have the following
lemma, to be proved in the next section:
\begin{lem} \label{lem44}
The projection of $\mho\in\Sym^{4,4}$ onto $\Sym^{4,0}\otimes \Sym^{0,4}$
is 
$$ \mho|_{\Sym^{4,0}\otimes\Sym^{0,4}} = \frac{256}{70}\ (1\otimes J^*) P_{44} (q\otimes q), $$
where $P_{44} : \Sym^{4,0} \otimes \Sym^{4,0} \to \Sym^{4,0} \otimes \Sym^{4,0}$ is the projection
$$
(P_{44})^{\alpha\beta\gamma\delta\ \epsilon\phi\kappa\lambda}_{\mu\nu\rho\sigma\ \xi\eta\zeta\tau}
=
\delta^{^{(\alpha}}_{[\mu} \delta^{_{(\epsilon}}_{\xi]}
\delta^{^\beta}_{[\nu} \delta^{_\phi}_{\eta]}
\delta^{^\gamma}_{[\rho} \delta^{_\kappa}_{\zeta]}
\delta^{^{\delta)}}_{[\sigma} \delta^{_{\lambda)}}_{\tau]}
$$
on the subspace corresponding to the Young diagram
with two rows of four boxes each, and $J^*$ denotes the map taking 
$t_{\alpha\beta\gamma\delta} \in \Sym^{4,0}_{abcd}$
to $
\bar\omega^{\mu}{}_{\bar\epsilon}
\bar\omega^{\nu}{}_{\bar\phi}
\bar\omega^{\rho}{}_{\bar\kappa}
\bar\omega^{\sigma}{}_{\bar\lambda} t_{\mu\nu\rho\sigma} \in \Sym^{0,4}_{efkl}.$
\end{lem}

From now on, we define 
$$\kappa = \dim\KK \quad\textrm{and}\quad 
 c_3 = 24 \sqrt\frac{1}{\kappa+3}  $$
to claim the following:
\begin{lem} \label{lemq}$ $ \begin{enumerate}
\item Introducing $\chi = \frac{\kappa+2}{\sqrt{\kappa+3}}$ and $N = \dim \V = 6\kappa+8,$
the tensor $q$ satisfies the following identities:
\begin{eqnarray*}
q_{\alpha\beta\mu\nu} \bar\omega^{\mu\nu} &=& 0 \\
q_{\mu\nu\rho\sigma} \bar\omega^{\mu\alpha}\bar\omega^{\nu\beta}\bar\omega^{\rho\gamma}\bar\omega^{\sigma\delta} &=& \bar q^{\alpha\beta\gamma\delta} \\
q_{\alpha\mu\nu\rho} \bar q^{\beta\mu\nu\rho} &=& \frac{N+1}{2}\ \delta_\alpha^\beta \\
q_{\alpha\beta\mu\nu} \bar q^{\gamma\delta\mu\nu} &=& 
\frac{1}{2} [\delta_\alpha^\gamma \delta_\beta^\delta + \delta_\alpha^\delta \delta_\beta^\gamma]+ \chi\ q_{\alpha\beta\mu\nu} \bar\omega^{\mu\gamma}\bar\omega^{\nu\delta}.
\end{eqnarray*}
\item
Let us introduce a map
$$ \cD_q : \Sym^{2,0} \to \Sym^{2,0} $$
$$ \Sym^{2,0}_{ab} \ni b_{\alpha\beta} \mapsto \cD_q(b)_{\alpha\beta} = b_{\mu\nu} \bar q^{\mu\nu\rho\sigma} \omega_{\rho\alpha} \omega_{\sigma\beta} \in \Sym^{2,0}_{ab}. $$
Then, with respect to the decomposition
$$ \Sym^{2,0} \simeq \CC\otimes\fsp(\cV_3(\KK),\omega) = \CC\otimes[\cg_3(\KK) \oplus \cg_3(\KK)^\bot], $$
the map $\cD_q$ is given by:
\begin{eqnarray*}
\cD_q|_{\cg_3(\KK)} &=& -\sqrt{\kappa+3} \\
\cD_q|_{\cg_3(\KK)^\bot} &=& \sqrt\frac{1}{\kappa+3}. 
\end{eqnarray*}
It also satisfies:
$$\cD_q^2 = 1 - \chi\cD_q. $$
\end{enumerate}\end{lem}
This one is also proved in the next section.
We are at last ready to check the stabilizer of $\mho.$ We shall deal with it in two long lemmas,
splitting $\fso(\cV_3(\KK))$ into $\fu(\V)$ and its complement.

Note that the term $P_{44}(q\otimes q)$ in our expression for the $\Sym^{4,0}\otimes\Sym^{0,4}$
component of $\mho$ is itself an element of $\Sym^{4,0}\otimes\Sym^{4,0},$ i.e. it is completely linear.
It is then convenient to consider only the $\End \V$ component of elements of $\fu(\cV_3(\KK)),$
so that a map $F^a{}_b = f^\alpha{}_\beta + \bar f^{\bar\alpha}{}_{\bar\beta} \in \fu(\cV_3(\KK))^a{}_b$
is represented by $f^\alpha{}_\beta \in \fu(\V)^\alpha{}_\beta\subset(\End \V)^\alpha{}_\beta.$
(the isomorphism $\cV_3(\KK)\simeq \V$ is implicit).
The algebra $\fu(\V)$ decomposes into the symplectic subalgebra and its complement, such that
$$\fsp_\CC(\V,\omega) \oplus \fsp^\bot_\CC(\V,\omega) \simeq \Sym^{2,0} \oplus \Lam^{2,0}, $$
where the isomorphism is given by $\omega_{\alpha\beta}$ as in (\ref{omid}). In the following we will omit indicating the complexifications of algebras explicitly 
-- it is clear that the final stabilizer is an intersection of a complexified one with the real algebra $\fso(\cV_3(\KK)).$

\begin{lem} The stabilizer algebra of $\mho$ in $\fsu(\V)$ is $\cg_3(\KK).$ \end{lem}
\begin{proof}
Let $E = A + B + C$ with $A\in\cg_3(\KK),\ B\in\cg_3^\bot(\KK)\subset\fsp(\V)$ and $C\in\fsp^\bot(\V)\subset\fsu(\V).$
Clearly, $[A,J] = [B,J] = 0$ and $CJ + JC = 0.$ Now, according to the previous lemma, we have
$$ E(\mho) = 0 \iff E( (1\otimes J^*) P_{44} (q\otimes q) ) = 0. $$
Using the properties of $A,B$ and $C,$ we have
$$ 
(1\otimes J^*) \circ [(A+B+C)\otimes 1 + 1 \otimes (A+B-C)]  \circ P_{44} (q\otimes q) = 0, $$
which under the action of $1\otimes J^{-1}$ yields
$$ [(A+B+C)\otimes 1 + 1 \otimes (A+B-C)] P_{44} (q\otimes q) = 0. $$
However terms involving $A$ shall obviously vanish, it is instructive to keep them for the sake of verification.
We shall contract the latter expression with two copies of $\bar q:$
$$ \left\lbrace[(A+B+C)\otimes 1 + 1 \otimes (A+B-C)] 
P_{44} (q\otimes q) \right\rbrace_{\alpha\beta\gamma\delta\mu\nu\rho\sigma} \bar q^{\xi\beta\gamma\delta} \bar q^{\mu\nu\rho\sigma} = 0. $$

In order to simplify bookkeeping, we introduce the following maps:
$$ \Phi,\Psi : \fsu(\V) \to \fsu(\V) $$
\begin{eqnarray*}
\Phi(F)^\xi_\alpha &=& \left\lbrace(F \otimes 1) 
P_{44} (q\otimes q) \right\rbrace_{\alpha\beta\gamma\delta\mu\nu\rho\sigma} \bar q^{\xi\beta\gamma\delta} \bar q^{\mu\nu\rho\sigma} \\
\Psi(F)^\xi_\alpha &=& \left\lbrace(1 \otimes F) 
P_{44} (q\otimes q) \right\rbrace_{\alpha\beta\gamma\delta\mu\nu\rho\sigma} 
\bar q^{\xi\beta\gamma\delta} \bar q^{\mu\nu\rho\sigma}
,
\end{eqnarray*}
so that our condition becomes $\Phi(A+B+C) + \Psi(A+B-C) = 0.$

Expanding the projection $P_{44},$ and utilising the symmetry of $q\otimes q,$ we have (for some irrelevant combinatorial constant $c'\neq 0$):
\begin{eqnarray*}
c'\ \Phi(F)^\xi_\alpha
&=& 2\  
q_{\eta(\alpha\beta\gamma} F^\eta{}_{\delta)} q_{\mu\nu\rho\sigma} 
\bar q^{\xi\beta\gamma\delta} \bar q^{\mu\nu\rho\sigma} 
\\
&-& 8 \left[
\frac{3}{4} q_{\eta(\mu(\alpha\beta} F^\eta{}_{\gamma} q_{\delta)\nu\rho\sigma)} +
\frac{1}{4} q_{(\mu(\alpha\beta\gamma} F^\eta{}_{\delta)} q_{\nu\rho\sigma)\eta}  
\right]
\bar q^{\xi\beta\gamma\delta} \bar q^{\mu\nu\rho\sigma} 
\\
&+& 6\ 
q_{\eta(\mu\nu(\alpha} F^\eta{}_{\beta} q_{\gamma\delta)\rho\sigma)}
\bar q^{\xi\beta\gamma\delta} \bar q^{\mu\nu\rho\sigma}
\end{eqnarray*}
and, similarly,
\begin{eqnarray*}
c'\ \Psi(F)^\xi_\alpha
&=& 2\ 
q_{\eta(\mu\nu\rho} F^\eta{}_{\sigma)} q_{\alpha\beta\gamma\delta} 
\bar q^{\xi\beta\gamma\delta} \bar q^{\mu\nu\rho\sigma} 
\\
&-& 8 \left[
\frac{3}{4} q_{\eta(\alpha(\mu\nu} F^\eta{}_{\rho} q_{\sigma)\beta\gamma\delta)} +
\frac{1}{4} q_{(\alpha(\mu\nu\rho} F^\eta{}_{\sigma)} q_{\beta\gamma\delta)\eta}  
\right]
\bar q^{\xi\beta\gamma\delta} \bar q^{\mu\nu\rho\sigma} 
\\
&+& 6\ 
q_{\eta(\alpha\beta(\mu} F^\eta{}_{\nu} q_{\rho\sigma)\gamma\delta)}
\bar q^{\xi\beta\gamma\delta} \bar q^{\mu\nu\rho\sigma}
\end{eqnarray*}
We will show that both $\Phi(F)$ and $\Psi(F)$ are linear combinations of $F$ and $\cD_q(F),$ where
the domain of $\cD_q$ is trivially extended onto entire $\fsu(\V)$ (so that $\fsp^\bot(\V)$ is its kernel). To this end we first get rid of the symmetrizers:
\begin{eqnarray*}
c'\ \Phi(F)^\xi_\alpha
&=& 2 \left[
\frac{3}{4} q_{\eta\alpha\beta\gamma} F^\eta{}_{\delta} q_{\mu\nu\rho\sigma} +
\frac{1}{4} q_{\eta\delta\beta\gamma} F^\eta{}_{\alpha} q_{\mu\nu\rho\sigma} 
\right]
\bar q^{\xi\beta\gamma\delta} \bar q^{\mu\nu\rho\sigma} 
\\
&-& 8 \left[
\frac{3}{4} \left\lbrace
\frac{1}{2} q_{\eta\mu\alpha\beta} F^\eta{}_{\gamma} q_{\delta\nu\rho\sigma} +
\frac{1}{4} q_{\eta\mu\gamma\beta} F^\eta{}_{\alpha} q_{\delta\nu\rho\sigma} + 
\frac{1}{4} q_{\eta\mu\delta\beta} F^\eta{}_{\gamma} q_{\alpha\nu\rho\sigma} 
\right\rbrace\right. \\ & & \quad \left. +
\frac{1}{4} \left\lbrace 
\frac{3}{4} q_{\mu\alpha\beta\gamma} F^\eta{}_{\delta} q_{\nu\rho\sigma\eta} + 
\frac{1}{4} q_{\mu\delta\beta\gamma} F^\eta{}_{\alpha} q_{\nu\rho\sigma\eta}  
\right\rbrace
\right]
\bar q^{\xi\beta\gamma\delta} \bar q^{\mu\nu\rho\sigma} 
\\
&+& 6\left[
\frac{1}{4} q_{\eta\mu\nu\alpha} F^\eta{}_{\beta} q_{\gamma\delta\rho\sigma} +
\frac{1}{4} q_{\eta\mu\nu\beta} F^\eta{}_{\alpha} q_{\gamma\delta\rho\sigma} 
\right. \\ & & \quad \left.
+ \frac{1}{2} q_{\eta\mu\nu\gamma} F^\eta{}_{\beta} q_{\alpha\delta\rho\sigma} 
\right]
\bar q^{\xi\beta\gamma\delta} \bar q^{\mu\nu\rho\sigma}, \end{eqnarray*}
and then perform contractions:
\begin{eqnarray*}
c'\ \Phi(F)^\xi_\alpha &=&
\frac{N(N+1)}{2} \left[3\varphi(F)^\xi_\alpha + \frac{N+1}{2}F^\xi{}_\alpha\right]
\\ &-& 3 \frac{N+1}{2}\varphi(F)^\xi_\alpha - \frac{3}{2} \left(\frac{N+1}{2}\right)^2 F^\xi{}_\alpha
- \frac{3}{2} \frac{N+1}{2}\varphi(F)^\xi_\alpha 
\\ & & \quad 
- \frac{3}{2}\frac{N+1}{2}\varphi(F)^\xi_\alpha
- \frac{1}{2} \left(\frac{N+1}{2}\right)^2 F^\xi{}_\alpha
\\ &+& \frac{3}{2}[(1+\chi^2)\varphi(F)^\xi_\alpha  + \chi \cD_q(F)^\xi_\alpha]
+ \frac{3}{2}(1+\chi^2)\frac{N+1}{2}F^\xi_\alpha \\
& & \quad + 3 \varphi(F)^\xi_\alpha + 3 \chi  
\left[ \left(-\frac{1}{2} + \chi^2 \right) \cD_q(F)^\xi_\alpha + \frac{\chi}{2}(F^\xi{}_\alpha + F^\mu{}_\mu \delta^\xi_\alpha) \right]
,
\end{eqnarray*}
where in the last lines we used the following identities:
$$
q_{\alpha\beta\mu\nu} \bar q^{\mu\nu\rho\sigma} q_{\rho\sigma\gamma\delta} =
(1+\chi^2) q_{\alpha\beta\gamma\delta} 
+ \frac{\chi}{2} ( \omega_{\alpha\gamma}\omega_{\beta\delta} + \omega_{\alpha\delta}\omega_{\beta\gamma} )
$$
$$ q_{\alpha\mu\kappa\lambda} \bar\omega^{\kappa\rho} \bar\omega^{\lambda\sigma} q_{\rho\sigma\beta\nu} \bar q^{\gamma\delta\mu\nu} = 
\left(-\frac{1}{2}+\chi^2\right) q_{\alpha\beta\kappa\lambda}\bar\omega^{\kappa\gamma}\bar\omega^{\lambda\delta}
+\frac{\chi}{2} ( \delta^\gamma_\alpha \delta^\delta_\beta +\delta^\delta_\alpha \delta^\gamma_\beta )
$$
and we have introduced
$$ \varphi(F)^\xi_\alpha = F^\mu{}_\nu q_{\alpha\mu\rho\sigma} \bar q^{\xi\nu\rho\sigma}
= \frac{1}{2} F^\xi{}_\alpha + \frac{1}{2} F^\mu{}_\mu \delta^\xi_\alpha 
+ \chi q_{\alpha\mu\rho\sigma} \bar\omega^{\rho\xi} \bar\omega^{\sigma\nu} F^\mu{}_\nu
, $$
so that $ \varphi(F) = \frac{1}{2} F + \chi \cD_q(F), $ since we assume $F$ to be traceless (note
also that the last term is indeed zero for $F\in\fsp^\bot(\V)$).
We perform the same operations for $\Psi:$ expand symmetrizers,
\begin{eqnarray*}
c'\ \Psi(F)^\xi_\alpha
&=& 2\ 
q_{\eta\mu\nu\rho} F^\eta{}_{\sigma} q_{\alpha\beta\gamma\delta} 
\bar q^{\xi\beta\gamma\delta} \bar q^{\mu\nu\rho\sigma} 
\\
&-& 8 \left[
\frac{3}{4}\left\lbrace
\frac{1}{4} q_{\eta\alpha\mu\nu} F^\eta{}_{\rho} q_{\sigma\beta\gamma\delta} +
\frac{3}{4} q_{\eta\delta\mu\nu} F^\eta{}_{\rho} q_{\sigma\beta\gamma\alpha}
\right\rbrace
\right. \\ & & + \quad\left.
\frac{1}{4} \left\lbrace
\frac{1}{4} q_{\alpha\mu\nu\rho} F^\eta{}_{\sigma} q_{\beta\gamma\delta\eta} +  
\frac{3}{4} q_{\delta\mu\nu\rho} F^\eta{}_{\sigma} q_{\beta\gamma\alpha\eta}  
\right\rbrace
\right]
\bar q^{\xi\beta\gamma\delta} \bar q^{\mu\nu\rho\sigma} 
\\
&+& 6 \left[ 
\frac{1}{2} q_{\eta\alpha\beta\mu} F^\eta{}_{\nu} q_{\rho\sigma\gamma\delta} +
\frac{1}{2} q_{\eta\delta\beta\mu} F^\eta{}_{\nu} q_{\rho\sigma\gamma\alpha}
\right]
\bar q^{\xi\beta\gamma\delta} \bar q^{\mu\nu\rho\sigma},
\end{eqnarray*}
and contract $q:$
\begin{eqnarray*}
c'\ \Psi(F)^\xi_\alpha &=& 2 \left(\frac{N+1}{2}\right)^2 F^\mu{}_\mu \delta^\xi_\alpha \\
&-& \frac{3}{2} \frac{N+1}{2}\varphi(F)^\xi_\alpha - \frac{9}{2} \varphi^2(F)^\xi_\alpha \\
& & \quad - \frac{1}{2} \left(\frac{N+1}{2}\right)^2 F^\xi_\alpha - \frac{3}{2} \frac{N+1}{2} \varphi(F)^\xi_\alpha \\
&+& 3 [(1+\chi^2) \varphi(F)^\xi_\alpha - \frac{\chi}{2} \cD_q(F)^\xi_\alpha ] \\
& & \quad
+ 3 \left[ \frac{1}{2} \frac{N+1}{2} F^\xi_\alpha + \frac{1}{2} \varphi(F)^\xi_\alpha \right]
\\ & & \quad + 3 \chi \left[ \left(\frac{1}{2}-\chi^2\right) \cD_q(F)^\xi_\alpha 
- \frac{\chi}{2}( \sigma(F)^\xi_\alpha - F^\mu{}_\mu \delta^\xi_\alpha)
\right],
\end{eqnarray*}
where in the last line we used one more identity:
$$ 
q_{\gamma\alpha\kappa\lambda} \bar\omega^{\kappa\mu}\bar\omega^{\lambda\nu} q_{\mu\eta\beta\delta}
\bar q^{\beta\delta\xi\gamma} 
=
\left( \frac{1}{2} - \chi^2 \right) 
q_{\eta\alpha\kappa\lambda} \bar\omega^{\kappa\xi}\bar\omega^{\lambda\nu}
+\frac{\chi}{2} [ \delta^\nu_\eta \delta^\xi_\alpha
- \bar\omega^{\xi\nu}\omega_{\eta\alpha} 
],
$$
and we have introduced a map $$ \sigma(F) = J^{-1} F J, $$
so that $\sigma|_{\fsp(\V)} = \id$ and $\sigma|_{\fsp^\bot(\V)} = -\id.$ The square of $\varphi,$
appearing in the second line, reads
$$ \varphi^2(F) = \left(\frac{1}{4} + \chi^2\frac{1+\sigma}{2}\right) F + \chi (1 - \chi^2) \cD_q(F). $$

We thus have:
\begin{eqnarray*}
c'\ \Phi(F) &=& (1-N+N^2+3\chi^2)\left\lbrace
\frac{N+4}{4} F + \frac{3\chi}{2} \cD_q(F) \right\rbrace \\
c'\ \Psi(F) &=& \frac{8-2N-N^2-3\chi^2(\sigma+2)}{8} F + \frac{3\chi}{2}(3\chi^2-N-1)  \cD_q(F)
\end{eqnarray*}
Using $\cD_q(A) = \lambda_+,$ $\cD_q(B)=\lambda_-,$ $\cD_q(C)=0$ and $\sigma(A)=A,$
 $\sigma(B)=B,$ $\sigma(C)=-C,$ and substituting $N,\lambda_\pm,\chi$ in
 $$ \Phi(A+B+C) + \Psi(A+B-C) = 0, $$
 we obtain $$ f_1(\kappa) B + f_2(\kappa) C = 0, $$ where $f_{1,2}(\kappa)$ are nonzero for
 $\kappa = 1,2,4,8$ (note that terms involving $A$ vanish as expected). Thus $$E(\mho)=0 \implies E \in \cg_3(\KK). $$
 The converse is obviously true form the definition of $\mho.$

\end{proof}

\begin{lem} The subspace of $\fu^\bot(\cV_3(\KK))\subset\fso(\cV_3(\KK))$ stabilizing $\mho$ is spanned by $J$ and $K.$
\end{lem}
\begin{proof}
We shall repeat the procedure applied in the previous proof.
Let $$E^\alpha{}_{\bar\beta}+\bar E^{\bar\alpha}{}_\beta \in\fu^\bot(\cV_3(\KK))^a{}_b.$$ We have $E(\mho) \in \Sym^{5,3} \oplus \Sym^{3,5}.$
Let us demand vanishing of the part in $\Sym^{3,5}:$
$$ P_{44}(q\otimes q)_{\xi\alpha\beta\gamma\epsilon\phi\kappa\lambda}
E^\xi{}_{(\bar\eta} 
\bar\omega^\epsilon{}_{\bar\mu}
\bar\omega^\phi{}_{\bar\nu}
\bar\omega^\kappa{}_{\bar\rho}
\bar\omega^\lambda{}_{\bar\sigma)}
= 0.$$
Acting with $J^{-1}$ on five barred indices we obtain:
$$ P_{44}(q\otimes q)_{\xi\alpha\beta\gamma(\epsilon\phi\kappa\lambda}
F^\xi{}_{\zeta)} = 0,$$
where $F^\xi{}_\zeta = - E^\xi{}_{\bar\eta}\omega^{\bar\eta}{}_\zeta.$
Contracting the latter with $ 
\bar q^{\alpha\beta\gamma\delta} \bar q^{\epsilon\phi\kappa\lambda} 
$ yields:
\begin{eqnarray*}
0 &=& 2\ q_{\xi\alpha\beta\gamma} q_{(\epsilon\phi\kappa\lambda} F^\xi{}_{\zeta)}
\bar q^{\alpha\beta\gamma\delta} \bar q^{\epsilon\phi\kappa\lambda} 
\\
&-& 8
\left[
\frac{3}{4} 
q_{\gamma(\phi\kappa\lambda} 
F^\xi{}_{\zeta}   
q_{\epsilon)\xi\alpha\beta} 
+
\frac{1}{4} 
q_{\xi(\phi\kappa\lambda} 
F^\xi{}_{\zeta}
q_{\epsilon)\alpha\beta\gamma} 
\right]
\bar q^{\alpha\beta\gamma\delta} \bar q^{\epsilon\phi\kappa\lambda} 
\\
&+& 6\ 
q_{\gamma\alpha(\kappa\lambda} F^\xi{}_{\zeta}
q_{\epsilon\phi)\beta\xi} 
\bar q^{\alpha\beta\gamma\delta} \bar q^{\epsilon\phi\kappa\lambda}. 
\end{eqnarray*}
Expanding symmetrizers:
\begin{eqnarray*}
0 &=& 2
\left[
\frac{1}{5} q_{\xi\alpha\beta\gamma} q_{\epsilon\phi\kappa\lambda} F^\xi{}_{\zeta} +
\frac{4}{5} q_{\xi\alpha\beta\gamma} q_{\zeta\phi\kappa\lambda} F^\xi{}_{\epsilon}
\right]
\bar q^{\alpha\beta\gamma\delta} \bar q^{\epsilon\phi\kappa\lambda} 
\\
&-& 8
\left[
\frac{3}{4} \left\lbrace
\frac{1}{5} q_{\gamma\phi\kappa\lambda} F^\xi{}_{\zeta} q_{\epsilon\xi\alpha\beta} +
\frac{1}{5} q_{\gamma\phi\kappa\lambda} F^\xi{}_{\epsilon} q_{\zeta\xi\alpha\beta} +
\frac{3}{5} q_{\gamma\zeta\kappa\lambda} F^\xi{}_{\phi} q_{\epsilon\xi\alpha\beta} 
\right\rbrace
\right. \\ & & \ \left. +
\frac{1}{4} \left\lbrace
\frac{1}{5} q_{\xi\phi\kappa\lambda} F^\xi{}_{\zeta} q_{\epsilon\alpha\beta\gamma} +
\frac{1}{5} q_{\xi\phi\kappa\lambda} F^\xi{}_{\epsilon} q_{\zeta\alpha\beta\gamma} +
\frac{3}{5} q_{\xi\zeta\kappa\lambda} F^\xi{}_{\phi} q_{\epsilon\alpha\beta\gamma}
\right\rbrace\right]
\bar q^{\alpha\beta\gamma\delta} \bar q^{\epsilon\phi\kappa\lambda} 
\\
&+& 6 \left[
\frac{1}{5} q_{\gamma\alpha\kappa\lambda} F^\xi{}_{\zeta} q_{\epsilon\phi\beta\xi} +
\frac{2}{5} q_{\gamma\alpha\zeta\lambda} F^\xi{}_{\kappa} q_{\epsilon\phi\beta\xi} +
\frac{2}{5} q_{\gamma\alpha\kappa\lambda} F^\xi{}_{\phi} q_{\epsilon\zeta\beta\xi} 
\right]
\bar q^{\alpha\beta\gamma\delta} \bar q^{\epsilon\phi\kappa\lambda},
\end{eqnarray*}
and performing contractions gives:
\begin{eqnarray*}
0&=& \frac{2}{5} \frac{N(N+1)}{2} F^\delta{}_\zeta + \frac{8}{5} \left(\frac{N+1}{2}\right)^2 F^\delta{}_\zeta \\
&-& \frac{6}{5} \left(\frac{N+1}{2}\right)^2 F^\delta{}_\zeta - \frac{6}{5} \frac{N+1}{2}\varphi(F)^\delta_\zeta 
\\ & & \quad 
- \frac{18}{5} \left[\frac{1}{2}\frac{N+1}{2} F^\delta{}_\zeta + \frac{1}{2}\varphi(F)^\delta_\zeta  \right] \\ & & \quad\quad
- \frac{18\chi}{5} \left[\left(\frac{1}{2}-\chi^2\right) \cD_q(F)^\delta_\zeta - \frac{\chi}{2}(
\sigma(F)^\delta_\zeta - F^\mu{}_\mu \delta^\delta_\zeta
)\right]
\\ & & \quad 
-\frac{2}{5} \left(\frac{N+1}{2}\right)^2 F^\delta{}_\zeta 
-\frac{2}{5} \frac{N+1}{2} F^\mu{}_\mu \delta^\delta_\zeta- \frac{6}{5} \frac{N+1}{2} \varphi(F)^\delta_\zeta \\
&+& \frac{6}{5} (1+\chi^2)\frac{N+1}{2} F^\delta{}_\zeta + \frac{12}{5} \varphi^2(F)^\delta_\zeta \\
& & \quad + \frac{12}{5} \varphi(F)^\delta_\zeta + \frac{12\chi}{5}\left[\left(-\frac{1}{2}+\chi^2\right)\cD_q(F)^\delta_\zeta + \frac{\chi}{2}(F^\delta{}_\zeta + F^\mu{}_\mu \delta^\delta_\zeta )\right],
\end{eqnarray*}
where we used the identities and maps introduced in the previous proof, the latter understood as mapping
$\End(\V)$ to itself, namely:
\begin{eqnarray*}
\cD_q(F)^\mu_\nu &=& F^\xi{}_\eta q_{\xi\nu\kappa\lambda} \bar\omega^{\kappa\eta} \bar\omega^{\lambda\nu} \\
\varphi(F)^\mu_\nu &=& \frac{1}{2}F^\mu{}_\nu + \frac{1}{2} F^\alpha{}_\alpha \delta^\mu_nu
+ \chi \cD_q(F)^\mu_\nu \\
\sigma(F)^\mu_\nu &=& F^\xi{}_\eta \bar\omega^{\mu\eta}\omega_{\nu\xi}.
\end{eqnarray*}
Collecting all terms gives:
\begin{eqnarray*}
0 &=& [2N^2+(6\chi^2-7)N+6\chi(3\chi+1)(1+\sigma)]F \\
&+& 12\chi(N+1-3\chi^2) \cD_q(F) \\
&-& (8N+5+6\chi^2) (\tr F) \id. \end{eqnarray*}
Recalling the eigenvalues of the (commuting) operators $\cD_q$ and $\sigma,$ one finds that
the latter equation is satisfied only for
$$ F^\mu{}_\nu = \lambda \delta^\mu_\nu,\quad \lambda\in\CC. $$
Hence $E$ is a linear combination of $J$ and $K.$

\end{proof}

At last, we arrive at the final step:
\begin{proof}[Proof of point 3 of Proposition \ref{pro-invs}]
Let us
consider $E \in \fso(\cV_3(\KK)).$ Then $E=A+B$ with $A\in\fu(\cV_3(\KK))$ and $B\in\fu^\bot(\cV_3(\KK)),$ and
the components map $\mho$ into independent subspaces:
$$ A(\mho) \in \Sym^{4,4} \quad\textrm{and}\quad B(\mho) \in \Sym^{5,3}\oplus\Sym^{3,5}. $$
Thus $E(\mho) = 0 \iff A(\mho) = 0\ \&\ B(\mho) = 0.$ Let now $A=A_0 + a I$ with $A_0 \in \fsu(\cV_3(\KK))$
and $a\in\RR.$ Clearly, $I(\mho)=0,$ so that $A(\mho)=0 \iff A_0(\mho)=0.$ We can now apply
the lemmas to $A_0$ and $B$ to find that
$$ A_0 \in \cg_3(\KK) \quad \textrm{and}\quad B \in \fsp(1) \cap \fu^\bot(\cV_3(\KK)), $$
and finally $E \in \fg(\KK,\OO)$ as claimed.
\end{proof}

\section{Important identities and proofs of Lemmas 18 \& 19}

The invariants of $\cG_n(\KK)$ are essentially derived from certain product maps: 
the traceless product $\times,$ Freudenthal product $\bullet$
and $\tau, $ the triple product on a FTS. As usually, finite dimension of the spaces the
maps operate on leads to some characteristic equations; these in turn are translated to certain identities the tensors $\Upsilon,$ $\Lambda$ and $q$ satisfy under contraction (the constants $c_1,\ c_2$ and
$c_3$ of
Proposition \ref{pro-invs} have been introduced in order to simplify these identities).

\subsection{First family}

\begin{lem}
Let $\LC_X : Y \mapsto X\times Y$ denote the (left) multiplication map for $X,Y\in\sxk.$ Then:
\begin{eqnarray}
\tr\ \LC_X &=& 0 \label{idlc} \\
\tr\ \LC_X \LC_X &=& \frac{3\dim\KK+4}{12} \langle X,X \rangle \label{idlc2} \\
 \tr\ \LC_X \LC_X \LC_X &=& -\frac{\dim\KK}{8} \langle X, X\times X \rangle \label{idlc3} \\
\langle X\times X, X\times X \rangle &=& \frac{1}{6} \langle X,X \rangle^2 \label{idlc4} 
\end{eqnarray}
for $X \in \sxk.$
\end{lem}
\begin{proof}
Identities (\ref{idlc}) and (\ref{idlc2}) follow from irreducibility of $\sxk$ as a representation
of $\Aut(\sxk),$ where  in the second case we employ Schur's lemma and check the formula for 
some simple $X$ to find the proportionality constant.

To prove (\ref{idlc4}) we recall Lemma \ref{lem-tsn}, substituting $T(X),S(X,X)$ and $N(X,X,X):$ 
$$ X^3 - \frac{1}{2} (\tr X^2) X - \frac{1}{3} \tr X^3 = 0, $$
where we used $\tr X=0.$
Taking the scalar product of this expression with $X$ and using symmetry of $\LC_X$ yields
$$ \langle X^2,X^2 \rangle = \frac{1}{2} \langle X,X \rangle^2. $$
Now,
\begin{eqnarray*}
\langle X\times X,X\times X\rangle &=& \langle X^2 - \frac{1}{3} \tr X^2,
X^2 - \frac{1}{3}\tr X^2 \rangle \\ &=& \langle X^2,X^2 \rangle - \frac{1}{3} \langle X,X \rangle^2
\\
&=& \frac{1}{6} \langle X,X\rangle^2.
\end{eqnarray*}

Using $i,j,k,\dots$ to index $\cV_1(\KK)\simeq\cV_1^*(\KK),$ with the identification given by $\langle\cdot,\cdot\rangle,$ we introduce the structure constants of $\times:$
$$ (X\times Y)_i = f_{ijk} X^j Y^k. $$
With $ g_{ij} X^i Y^j = \langle X,Y \rangle, $ identities  (\ref{idlc}, \ref{idlc2}, \ref{idlc4}) read:
$$ f_{imm} = 0,\quad f_{imn} f_{jmn} = \frac{3\dim\KK+4}{12}g_{ij},\quad f_{m(ij} f_{kl)m} = \frac{1}{6} g_{(ij}g_{kl)}.$$
Contraction of the latter one with $c_{kln}$ yields:
$$ \frac{3\dim\KK+4}{24} f_{ijn} +  f_{ipq} f_{jqr} f_{nrp} = 
\frac{1}{6}  f_{ijn} 
$$
$$ f_{ipq} f_{jqr} f_{nrp} = - \frac{\dim\KK}{8} f_{ijn}, $$
proving (\ref{idlc3}).
\end{proof}

\begin{cor} \label{cor-idy}
Let us set $c_1 = \sqrt\frac{12}{3\kappa+4},$ where $\kappa = \det\KK.$
Using $i,j,k,\dots$ to index $\cV_1(\KK) \simeq \cV_1^*(\KK),$ with the identification
given by $\langle\cdot,\cdot\rangle,$ we have:
\begin{eqnarray*}
\Upsilon_i{}^{mm} &=& 0 \\
\Upsilon_i{}^{mn} \Upsilon_j{}^{mn} &=& g_{ij} \\
\Upsilon_{i}{}^{rp} \Upsilon_j{}^{pq} \Upsilon_{k}{}^{qr} &=& \frac{-3\kappa}{6\kappa+8} \Upsilon_{ijk} \\
\Upsilon_{(ij}{}^m \Upsilon_{kl)}{}^m &=& \frac{10}{3\kappa+4} g_{(ij} g_{kl)},
\end{eqnarray*}
where $g(X,Y) = \langle X,Y \rangle. $
\end{cor}

Notably, the last identity is, up to a proportionality constant, the one used by Nurowski in \cite{bobienski-2005}
to define his cubic invariant (without referring to all the Jordan machinery).

\begin{lem} \label{lemdy} Let us introduce a map
$$ \cD_\Upsilon : \Lambda^2 \cV_1(\KK) \to \Lambda^2 \cV_1(\KK) $$
$$ (\Lambda^2 \cV_1(\KK))_{ij} \ni E_{ij} \mapsto E_{pq} \Upsilon_{iqm}\Upsilon_{jpm} \in (\Lambda^2 \cV_1(\KK))_{ij}. $$
Then, with respect to the decomposition
$$ \Lambda^2 \cV_1(\KK) \simeq \fso(\cV_1(\KK)) = \cg_1(\KK) \oplus \cg_1(\KK)^\bot, $$
the map $\cD_\Upsilon$ is given by:
\begin{eqnarray*}
\cD_\Upsilon|_{\cg_1(\KK)} &=&-\frac{1}{2} \\
\cD_\Upsilon|_{\cg_1(\KK)^\bot} &=& \frac{3}{3\kappa+4} .
\end{eqnarray*}
\end{lem}

\begin{proof} 
Let us first consider the familiar map into derivations (Lemma \ref{lem-dmapj}):
\begin{equation}
\label{eq-foo}
\Lambda^2 \cV_1(\KK) \ni X \wedge Y \mapsto \cD_{X,Y} = [L_X,L_Y] \in \cg_1(\KK) \subset
\Lambda^2 \cV_1(\KK). \end{equation}
We have $$ (X^p Y^q - Y^p X^q) \Upsilon_{iqm} \Upsilon_{jpm} = -c_1^2 [\LC_X,\LC_Y]_{ij} $$
and
$$ [\LC_X,\LC_Y] = [L_X,L_Y] - \frac{1}{3} X \wedge Y. $$
Since the map (\ref{eq-foo}) is symmetric and its image is $\cg_1(\KK)=\der\sxk,$ it follows
that it vanishes on $\cg_1(\KK)^\bot.$ It then follows from equivariance, that its action on $\cg_1(\KK)$ is a multiple of identity. We thus have:
$$ \cD_\Upsilon|_{\cg_1(\KK)^\bot} = \frac{c_1^2}{3},\quad \cD_\Upsilon|_{\cg_1(\KK)}
= \alpha $$
for some constant $\alpha.$ Then, demanding 
$$\tr\cD_\Upsilon = 
\Upsilon_i{}^{qm} \Upsilon^p{}_{jm}
\ \delta^{[i}_p \delta^{j]}_q = \frac{N}{2} $$
and substituting the dimensions of $\der\sxk,$ one finds $\alpha = -c_1^2 (\frac{3\kappa+4}{24}),$
and the lemma follows.
\end{proof} 

\subsection{Second family}
It is convenient to consider the following general result, which shall prove useful in both second and third family. The idea is to use some abstract \emph{real} spaces $W,\tilde W\simeq\hxk$ instead of $\V,\bar\V$ (recall Remark \ref{not-cpx}):

\begin{lem}\label{lem-id2}
Let us introduce two spaces $W,\tilde W \simeq \hxk$ and maps
\begin{eqnarray*}
\lF_X &:& W \to \tilde W\quad \textrm{for}\ X \in W\\
\tlF_{\tilde X} &:& \tilde W \to W\quad \textrm{for}\ \tilde X \in \tilde W
\end{eqnarray*}
such that $\lF_X$ and $\tlF_{\tilde X}$ become
respectively $\LF_X$ and $\LF_{\tilde X}$ under the isomorphisms $W,\tilde W\simeq\hxk.$ Taking the traces over $W\oplus\tilde W,$ we have:
\begin{eqnarray}
\tr\ \lF_Y &=& 0 \label{idlf1} \\
\tr\ \tlF_{\tilde X} \lF_Y &=& \frac{\dim\KK+2}{4} \langle \tilde X, Y \rangle \label{idlf2} \\
\label{idlf4}
\tr\ \lF_Y \tlF_{\tilde Z} \lF_T &=& 0 \label{idlf3} \\
\tr\ \tlF_{\tilde X} \lF_Y \tlF_{\tilde Z} \lF_T &=& \frac{\dim\KK+2}{32}[
\langle\tilde X, Y\rangle\langle\tilde Z,T\rangle + \langle\tilde X,T \rangle\langle\tilde Z,Y\rangle] \\ &- & \frac{\dim\KK}{8} \langle \lF_Y T, \tlF_{\tilde X}\tilde Z\rangle  \nonumber
\end{eqnarray}
for $Y,T\in W$ and $\tilde X,\tilde Z \in \tilde W.$
\end{lem}
\begin{proof}
Identities (\ref{idlf1}) and (\ref{idlf3}) are obvious, as we trace maps which swap the (sub)spaces $W$ and
$\tilde W.$ Formula (\ref{idlf2}) follows, up to the constant, from Schur's lemma:
indeed, let us complexify $W,\tilde W$ and identify them complex-linearily with $\V,\bar \V$ respectively. Then (\ref{idlf2}) defines a sesquilinear form on $\cV_2(\KK)$ -- as the latter is an irreducible
representation of $\cG_2(\KK)$, the form must be a multiple of the hermitian inner product. 
The factor can be found by checking for some simple $\tilde X$ and $Y.$

To prove (\ref{idlf4}), we introduce some extra notation: we
shall use $i,j,\dots$ to index $W$ and $\tilde i,\tilde j,\dots$ to index $\tilde W,$ and the same for their duals.
We introduce a tensor
$ \gamma(\tilde X,Y) = \langle\tilde X,Y\rangle,$
so that $\gamma_{\tilde i j} \in \tilde W^*_i\otimes W^*_j,$ with
its inverse satisfying
$$ \gamma_{\tilde i j} \gamma^{j\tilde i} = \dim\hxk = 3\dim\KK + 3. $$
The determinant defines tensors $N_{ijk}$ and $\tilde N_{\tilde i\tilde j\tilde k}$ such that
$$ 
\det X = N_{ijk} X^i X^j X^k,\quad
\det \tilde X = \tilde N_{\tilde i\tilde j \tilde k} \tilde X^{\tilde i} \tilde X^{\tilde j} \tilde X^{
\tilde k}. $$
We will use $\gamma$ and its inverse to lower and raise indices, so that the maps $\lF$ and $\tlF$
are simply given by
$$
(\lF_X Y)^{\tilde i} = 3 N^{\tilde i}{}_{jk} X^j Y^k,\quad
(\tlF_{\tilde X} \tilde Y)^{i} = 3 \tilde N^{i}{}_{\tilde j\tilde k}
\tilde X^{\tilde j} \tilde Y^{\tilde k}.
$$
Equation (\ref{idlf2}) is expressed as: 
$$ 9 \tilde N^m{}_{\tilde i\tilde n} N^{\tilde n}{}_{jm} = \frac{\dim\KK+2}{4} \gamma_{\tilde i j}. $$

We now recall that $X^{\sharp\sharp} = (\det X)X$ for $X \in \hxk.$ Utilising the isomorphisms
$W,\tilde W \simeq \hxk,$ we can write this identity as follows:
\begin{eqnarray*}
9 [
N_{imp} N_{knq} \tilde N_{\tilde j}{}^{pq} +
N_{mkp} N_{inq} \tilde N_{\tilde j}{}^{pq} \\ \quad +
N_{ikp} N_{mnq} \tilde N_{\tilde j}{}^{pq} 
] &=& \frac{1}{4} [
N_{imk} \gamma_{\tilde jn} +
N_{ink} \gamma_{\tilde jm} \\ & & \quad +
N_{imn} \gamma_{\tilde jk} +
N_{nmk} \gamma_{\tilde ji}
].
\end{eqnarray*}
Contraction with $\tilde N_{\tilde l}{}^{mn}$ yields:
\begin{eqnarray*} 18\ 
N_{ip}{}^{\tilde m} \tilde N_{\tilde j\tilde q}{}^p N_{kn}{}^{\tilde q} \tilde N_{\tilde l\tilde m}{}^n +
\frac{\dim\KK+2}{4} N_{ikp} \tilde N_{\tilde j\tilde l}{}^p
&=&
\frac{1}{2} N_{ikm} \tilde N_{\tilde j\tilde l}{}^m \\ &+& \frac{1}{36} \frac{\dim\KK+2}{4} 
[ \gamma_{\tilde li}\gamma_{\tilde jk} + \gamma_{\tilde ji} \gamma_{\tilde lk} ],
\end{eqnarray*}
where we used (\ref{idlf2}). Thus
$$ 
81 N_{ip}{}^{\tilde m} \tilde N_{\tilde j\tilde q}{}^p N_{kn}{}^{\tilde q} \tilde N_{\tilde l\tilde m}{}^n = - \frac{\dim\KK}{8} 9 N_{ikp} \tilde N_{\tilde jl}{}^p + \frac{\dim\KK}{32}[ \gamma_{\tilde li}\gamma_{\tilde jk} + \gamma_{\tilde ji} \gamma_{\tilde lk} ], $$
which is equivalent to (\ref{idlf4}).
\end{proof}

For $\CC\otimes W = \V$ and $\CC\otimes \tilde W = \bar \V,$ where $\V,\bar \V$ are the duals of the (complex) spaces
of linear and antilinear complex forms on $\cV_2(\KK),$ as in (\ref{vbarv}), we readily have
the following
\begin{cor}\label{cor-id2}
Let us set $c_2 = \sqrt\frac{4}{\kappa+2},$ where $\kappa = \det\KK.$ Recalling the conventions
introduced in Remark \ref{not-cpx}, we have:
\begin{eqnarray*} \
\Lambda_{\alpha\mu\nu} \bar\Lambda_{\bar\beta}{}^{\mu\nu} &=& h_{\bar\beta\alpha} \\ 
\Lambda_{\alpha\mu\nu} \bar\Lambda_{\bar\beta}{}^{\nu\rho} 
\Lambda_{\gamma\rho\sigma} \bar\Lambda_{\bar\delta}{}^{\sigma\mu} &=&
\frac{1}{2\kappa+4} (h_{\bar\beta\alpha}h_{\bar\delta\gamma}+h_{\bar\delta\alpha} h_{\bar\beta\gamma}) \\
&-& \frac{\kappa}{2\kappa+4}\Lambda_{\alpha\gamma\mu} \bar\Lambda_{\bar\beta\bar\delta}{}^\mu.
\end{eqnarray*}
\end{cor}

We shall further make use of the following:
\begin{lem} \label{lemdl}
Let us introduce a map
$$ \cD_\Lambda : \Lam^{1,1} \to \Lam^{1,1} $$
$$ \Lam^{1,1}_{ab} \ni b_{\bar\alpha\beta} - b_{\bar\beta\alpha} \mapsto 
b_{\bar\mu\nu}\Lambda^{\bar\mu\bar\rho}{}_{\beta}\bar\Lambda^{\nu}{}_{\bar\rho\bar\alpha} -
b_{\bar\mu\nu}\Lambda^{\bar\mu\bar\rho}{}_{\alpha}\bar\Lambda^{\nu}{}_{\bar\rho\bar\beta}
\in \Lam^{1,1}_{ab}. $$
Then, with respect to the decomposition
$$ \Re\Lam^{1,1} \simeq \fu(\cV_2(\KK)) = \theta\RR \oplus \cg_2(\KK)\oplus\cg_2(\KK)^\bot, $$
where $\cg_2(\KK)^\bot$ is the orthogonal complement of $\cg_2(\KK)$ in $\fsu(\cV_2(\KK)),$ the map
$\cD_\Lambda$ is given by:
\begin{eqnarray*}
\cD_\Lambda|_{\theta\CC} &=& 1 \\
\cD_\Lambda|_{\cg_2(\KK)} &=& -\frac{1}{2} \\
\cD_\Lambda|_{\cg_2(\KK)^\bot} &=& \frac{1}{\kappa+2}.
\end{eqnarray*}
It also satisfies:
$$ {\cD_\Lambda}|_{\fsu(\cV_2(\KK))}^2 = \frac{1}{2\kappa+4}[1 - \kappa\cD_\Lambda|_{\fsu(\cV_2(\KK))}].$$ 
\end{lem}

\begin{proof}
It is easy to check that $\cD_\Lambda(\theta) = \theta.$ Now,
let $$B_{ab}=b_{\bar\alpha\beta}-b_{\bar\beta\alpha} \in \fsu(\cV_2(\KK))_{ab}.$$ Then
$ \cD_\Lambda(B)_{ab} = 
b_{\bar\mu\nu}\Lambda^{\bar\mu\bar\rho}{}_{\beta}\bar\Lambda^{\nu}{}_{\bar\rho\bar\alpha} -
b_{\bar\mu\nu}\Lambda^{\bar\mu\bar\rho}{}_{\alpha}\bar\Lambda^{\nu}{}_{\bar\rho\bar\beta}$
(still in $\fsu(\cV_2(\KK))_{ab},$ due to hermiticity of $\cD_\Lambda$) and
\begin{eqnarray*} \cD_\Lambda^2(B)_{ab} &=& 
b_{\bar\mu\nu}\Lambda^{\bar\mu\bar\rho}{}_{\lambda}\bar\Lambda^{\nu}{}_{\bar\rho\bar\kappa} 
\Lambda^{\bar\lambda\bar\sigma}{}_{\beta}\bar\Lambda^{\kappa}{}_{\bar\sigma\bar\alpha} 
\\ & & \quad -\ 
b_{\bar\mu\nu}\Lambda^{\bar\mu\bar\rho}{}_{\lambda}\bar\Lambda^{\nu}{}_{\bar\rho\bar\kappa} 
\Lambda^{\bar\lambda\bar\sigma}{}_{\alpha}\bar\Lambda^{\kappa}{}_{\bar\sigma\bar\beta} \\
&=&
\frac{1}{2\kappa+4}[ b_{\bar\alpha\beta} + b^\mu{}_\mu h_{\bar\alpha\beta} - \kappa
b_{\bar\mu\nu} \Lambda^{\bar\mu\bar\rho}{}_{\beta}\bar\Lambda^{\nu}{}_{\bar\rho\bar\alpha}]
\\ & & \quad -\ 
\frac{1}{2\kappa+4}[ b_{\bar\beta\alpha} + b^\mu{}_\mu h_{\bar\beta\alpha} - \kappa
b_{\bar\mu\nu} \Lambda^{\bar\mu\bar\rho}{}_{\alpha}\bar\Lambda^{\nu}{}_{\bar\rho\bar\beta}],
\end{eqnarray*}
where we used Corollary \ref{cor-id2} and $b^\mu{}_\mu = 0.$ Thus, when restricted to
$\fsu(\cV_2(\KK)),$
$$ \cD_\Lambda^2 = \frac{1}{2\kappa+4}[1 - \kappa\cD_\Lambda].$$ 
Then, since $\cD_\Lambda$ is hermitian, it splits the space $\fsu(\cV_2(\KK))$ orthogonally into
$\fsu^+(\cV_2(\KK))\oplus\fsu^-(\cV_2(\KK))$ with
$$ \cD_{\Lambda}|_{\fsu^\pm(\cV_2(\KK))} = \lambda^\pm \id,\quad 
\lambda^\pm = \frac{-\kappa\pm (\kappa+4)}{4\kappa+8}.$$
It follows that
$$ \lambda^+ \dim \fsu^+(\cV_2(\KK)) + \lambda^- \dim\fsu^-(\cV_2(\KK)) = \tr \cD_\Lambda = \Lambda_{\mu\nu\rho}
\bar\Lambda^{\mu\nu\rho} = 3\kappa+3. $$
Using $\dim\fsu(\cV_2(\KK)) = N^2-1$ we can compute the dimensions of $\fsu^\pm(\cV_2(\KK))$ and identify
$\fsu^-(\cV_2(\KK)) $ as $ \cg_2(\KK).$
\end{proof}

\subsection{Third family}
\begin{lem}
Let $\LT_{X,Y} : Z \mapsto \tau(X,Y,Z)$ be the left multiplication map for $X,Y,Z\in\fxk.$
Then:
\begin{eqnarray} \label{idt}
\tr\ \LT_{X,Y} &=& 0 \\ \label{idtt}
\tr\ \LT_{X,Y} \LT_{Z,T} &=& \frac{\dim\KK+3}{2^7 3^2} [\omega(X,Z)\omega(Y,T)+\omega(X,T)\omega(Y,Z)]
\\ &+& \frac{\dim\KK+2}{24} \cQ(X,Y,Z,T). \nonumber
\end{eqnarray}
Moreover, let $J_0 : \fxk\to\fxk$ be the map given by $\langle J_0 X,Y\rangle = \omega(X,Y). $
Then \begin{eqnarray} \label{idqj} \cQ(J_0 X,J_0 X,J_0 X,J_0 X) = \cQ(X,X,X,X). \end{eqnarray}
\end{lem}

\begin{proof} 
The equation (\ref{idt}) follows from symmetry of $\cQ$ and antisymmetry of $\omega,$ 
and (\ref{idqj}) is easy to check directly. On the other hand, proving (\ref{idtt}) will
require considerably more effort, and some extra notation.

To use the results of Lemma (\ref{lem-id2}), we introduce spaces
$$ W,\tilde W \simeq \hxk,\ L,\tilde L \simeq \RR $$
such that
$$\fxk = (\RR\oplus \hxk) \otimes \RR^2 = (L \oplus W) \otimes (1,0) \oplus (\tilde L \oplus \tilde W) \otimes (0,1). $$

In the same manner as in the proof of Lemma (\ref{lem-id2}), the scalar product on $\hxk$ is extended to a tensor
$$ \gamma \in \tilde W^* \otimes W^*,\quad \gamma(\tilde X,Y) = \langle\tilde X,Y\rangle, $$
and we introduce for later convenience the normalised forms on the real lines $L,\tilde L,$ denoted
$$\xi\in L^*,\tilde\xi\in\tilde L^*,\quad \xi(x)=x,\ \tilde\xi(\tilde x)=\tilde x. $$
We shall now use $i,j,\dots$ to index $L\oplus W$ and $\tilde i,\tilde j,\dots$ to index $\tilde L\oplus \tilde W,$ and the same for their duals.
Indices are lowered and raised with help of the tensor
$$ \gamma_{\tilde ij} + \tilde\xi_{\tilde i}\xi_{\tilde j} $$
and its inverse, $\gamma^{i\tilde j} + \tilde\xi^i \xi^{\tilde j},$
where 
$$ \gamma_{\tilde i j} \gamma^{j\tilde i} = \dim\hxk = 3\dim\KK + 3,\quad \xi_i\tilde\xi^i = 1 = \tilde\xi_{\tilde j}\xi^{\tilde j}$$
and
$$ \gamma_{\tilde i j} \tilde\xi^j = 0,\quad \gamma_{\tilde i j} \xi^{\tilde i} = 0,
\quad \gamma^{i\tilde j} \xi_i = 0,\quad \gamma^{i\tilde j}\tilde\xi_{\tilde j} = 0. $$

The determinant defines tensors $N_{ijk}$ and $\tilde N_{\tilde i\tilde j\tilde k}$ such that
$$ 
\det X = N_{ijk} X^i X^j X^k,\quad
\det \tilde X = \tilde N_{\tilde i\tilde j \tilde k} \tilde X^{\tilde i} \tilde X^{\tilde j} \tilde X^{
\tilde k} $$
for $X\in W$ and $\tilde X \in \tilde W,$ vanishing on $L$ and $\tilde L,$ so that
$$ N_{ijk} \tilde\xi^k = 0,\quad \tilde N_{\tilde i\tilde j\tilde k} \xi^{\tilde k} = 0. $$

The maps $\lF$ and $\tlF$ 
are given by
$$
(\lF_X Y)^{\tilde i} = 3 N^{\tilde i}{}_{jk} X^j Y^k,\quad
(\tlF_{\tilde X} \tilde Y)^{i} = 3 \tilde N^{i}{}_{\tilde j\tilde k}
\tilde X^{\tilde j} \tilde Y^{\tilde k}
$$ for $X,Y \in W$ and $\tilde X,\tilde Y$ in $\tilde W,$
and identities of Lemma \ref{lem-id2} are expressed as: 
$$ 9 \tilde N^m{}_{\tilde i\tilde n} N^{\tilde n}{}_{jm} = \frac{\dim\KK+2}{4} \gamma_{\tilde i j}, $$
\begin{eqnarray*}
81
\tilde N^p{}_{\tilde i\tilde q} N^{\tilde n}{}_{jp} 
\tilde N^m{}_{\tilde k\tilde n} N^{\tilde q}{}_{lm} 
&=& \frac{\dim\KK+2}{32} [\gamma_{\tilde i j} \gamma_{\tilde k l} + \gamma_{\tilde i l} \gamma_{\tilde k j}]
\\ &-& \frac{9 \dim\KK}{8} 
\tilde N^m{}_{\tilde i\tilde k} N_{mjl}.
\end{eqnarray*}

Using corresponding uppercase letters to index full $\fxk,$ we have
$$ \fxk^I = (L\oplus W)^i \oplus (\tilde L\oplus \tilde W)^{\tilde i},$$
etc., and the scalar product and symplectic form on this space is represented by
$$ k^{PQ} = 
\gamma^{p\tilde q} + \tilde\xi^p\xi^{\tilde q} +
\gamma^{q\tilde p} + \tilde\xi^q\xi^{\tilde p} $$
$$ \omega^{PQ} = 
\gamma^{p\tilde q} + \tilde\xi^p\xi^{\tilde q} -
\gamma^{q\tilde p} - \tilde\xi^q\xi^{\tilde p}. $$
The quartic defining the triple product is $Q_{IJKL},$
and, recalling that $\tau$ is given in terms of
$\cQ_{IJKM} \omega^{ML},$
the expression we wish to compute is
$$ \cQ^2_{IJKL} = \cQ_{IJPQ} \cQ_{KLRS} \omega^{PR} \omega^{QS}. $$
Projecting onto subspaces of $\fxk,$ the components of the \emph{symmetric} tensor $\cQ$ are given by
\begin{eqnarray*}
\cQ_{IJPQ} &=& \cQ_{ijpq} + \cQ_{\tilde i\tilde j\tilde p\tilde q} +
\cQ_{ij\tilde p\tilde q} +\cQ_{\tilde i\tilde j pq} \\ &+&
\cQ_{\tilde i j \tilde p q} + \cQ_{\tilde i j p \tilde q} +
\cQ_{i\tilde j \tilde p q} + \cQ_{i\tilde j p \tilde q}
\end{eqnarray*}
with
\begin{eqnarray*}
\cQ_{ijkl} &=& \frac{1}{4} [\xi_{i} N_{jkl}+ \xi_{j} N_{kli} +\xi_{k} N_{lij} +\xi_{l} N_{ijk} ]
\\
6 \cQ_{ij\tilde k\tilde l} &=& 9 N_{mij} \tilde N^m{}_{\tilde k\tilde l} 
-\frac{1}{4} \xi_i \xi_j \tilde\xi_{\tilde k} \tilde\xi_{\tilde l}
\\
&+&\frac{1}{8} 
[-\gamma_{\tilde k i} \gamma_{\tilde l j} - \gamma_{\tilde l i}\gamma_{\tilde k j} +
\gamma_{\tilde k i} \tilde\xi_{\tilde l} \xi_j +
\gamma_{\tilde l i} \tilde\xi_{\tilde k} \xi_j +
\gamma_{\tilde k j} \tilde\xi_{\tilde l} \xi_i +
\gamma_{\tilde l j} \tilde\xi_{\tilde k} \xi_i
].
\end{eqnarray*}
Recalling that we contract $\cQ_{IJPQ}$ with
\begin{eqnarray*} \cQ_{KLRS} \omega^{PR} \omega^{QS} &=& 
\cQ_{kl}{}^{pq} + \cQ_{\tilde k\tilde l}{}^{\tilde p\tilde q}
+ \cQ_{kl}{}^{\tilde p\tilde q}
+ \cQ_{\tilde k \tilde l}{}^{pq} \\ &-&
\cQ_{\tilde kl}{}^{\tilde pq}
-\cQ_{\tilde kl}{}^{p\tilde q}
-\cQ_{k\tilde l}{}^{\tilde pq}
-\cQ_{k\tilde l}{}^{p\tilde q}, \end{eqnarray*}
we find that the following expressions need to be computed:
\begin{eqnarray*}
24 \cQ_{ijpq} \cQ_{kl}{}^{pq} &=& 
[2\xi_{(i} N_{j)pq} + 2 \xi_{p} N_{qij} ] 
[9 N_{mkl} \tilde N^{mpq} - \frac{1}{4} \xi_k\xi_l\tilde\xi^p\tilde \xi^q
\\ & & \quad
+\frac{1}{8} (-\delta_k^p\delta_l^q-\delta_l^p\delta_k^q
+\delta_k^p \xi_l\tilde\xi^q
+\delta_k^q \xi_l\tilde\xi^p
+\delta_l^p \xi_k\tilde\xi^q
+\delta_l^q \xi_k\tilde\xi^p
)] \\
&=& \left(2\frac{\dim\KK+2}{4} -\frac{1}{2} \right) \xi_{(i}N_{j)kl} + \frac{1}{2} N_{ij(k}\xi_{l)}
\end{eqnarray*}

\begin{eqnarray*}
16 \cQ_{ijpq} \cQ^{pq}{}_{\tilde k\tilde l} &=& 
[2\xi_{(i} N_{j)pq} 
+\xi_{p} N_{qij} 
+\xi_{q} N_{pij} 
]
[
\tilde\xi^p \tilde N^q{}_{\tilde k\tilde l} + 
\tilde\xi^q \tilde N^p{}_{\tilde k\tilde l} + 
2\tilde \xi_{(\tilde k} \tilde N^{pq}{}_{\tilde l)} 
 ] \\
&=& 2 N_{mij} \tilde N^m{}_{\tilde k\tilde l} +
\frac{1}{9}\frac{\dim\KK+2}{4}[\gamma_{\tilde k i} \tilde\xi_{\tilde l} \xi_j +
\gamma_{\tilde l i} \tilde\xi_{\tilde k} \xi_j +
\gamma_{\tilde k j} \tilde\xi_{\tilde l} \xi_i +
\gamma_{\tilde l j} \tilde\xi_{\tilde k} \xi_i]
\end{eqnarray*}

\begin{eqnarray*}
36 \cQ_{ij\tilde p\tilde q} \cQ^{\tilde p\tilde q}{}_{\tilde k\tilde l} &=&
[ 9 N_{mij} \tilde N^m{}_{\tilde p\tilde q} - \frac{1}{4} \gamma_{\tilde p (i} \gamma_{j)\tilde q } 
- \frac{1}{4}\xi_i\xi_j\tilde\xi_{\tilde p}\tilde\xi_{\tilde q} 
+\frac{1}{4}( 
\gamma_{\tilde p (i}\xi_{j)}\tilde \xi_{\tilde q} +
\gamma_{\tilde q (i}\xi_{j)}\tilde \xi_{\tilde p} 
)
] \\ &\times&
[ 9 N_{n}{}^{\tilde p\tilde q} \tilde N^n{}_{\tilde k\tilde l} 
- \frac{1}{4} \delta_{(\tilde k}^{\tilde p} \delta_{\tilde l)}^{\tilde q} 
- \frac{1}{4}\xi^{\tilde p}\xi^{\tilde q}\tilde\xi_{\tilde k}\tilde\xi_{\tilde l} 
+\frac{1}{4}(
\delta_{(\tilde k}^{\tilde p}\tilde \xi_{\tilde l)} \xi^{\tilde q} +
\delta_{(\tilde k}^{\tilde q}\tilde \xi_{\tilde l)} \xi^{\tilde p}
)]\\
&=&
9\left(\frac{\dim\KK+2}{4}-\frac{1}{4}-\frac{1}{4}\right) N^{\tilde m}{}_{ij} \tilde N_{\tilde m\tilde k \tilde l} + \frac{1}{32} (\gamma_{\tilde k i}\gamma_{\tilde l j} + \gamma_{\tilde li}\gamma_{\tilde kj})
\\
&+& \frac{1}{16} \xi_i\xi_j\tilde\xi_{\tilde k}\tilde\xi_{\tilde l}
+ \frac{1}{16} (
\gamma_{\tilde k i} \tilde\xi_{\tilde l} \xi_j +
\gamma_{\tilde l i} \tilde\xi_{\tilde k} \xi_j +
\gamma_{\tilde k j} \tilde\xi_{\tilde l} \xi_i +
\gamma_{\tilde l j} \tilde\xi_{\tilde k} \xi_i )
\end{eqnarray*}

\begin{eqnarray*}
36 \cQ_{ip\tilde j\tilde q} \cQ_k{}^{\tilde q}{}_{\tilde l}{}^p &=&
[ 9 N_{mip} \tilde N^m{}_{\tilde j\tilde q} - \frac{1}{8} (
\gamma_{\tilde j i}\gamma_{\tilde q p}+
\gamma_{\tilde j p}\gamma_{\tilde q i})
-\frac{1}{4} \xi_i\xi_p\tilde\xi_{\tilde j}\tilde\xi_{\tilde q} 
\\ & & \quad + \frac{1}{8}(
\gamma_{\tilde j i} \xi_p\tilde\xi_{\tilde q} +
\gamma_{\tilde j p} \xi_i\tilde\xi_{\tilde q} +
\gamma_{\tilde q i} \xi_p\tilde\xi_{\tilde j} +
\gamma_{\tilde q p} \xi_i\tilde\xi_{\tilde j}) ]
\\ &\times&
[ 9 N_{nk}{}^{\tilde q} \tilde N^{np}{}_{\tilde l} - \frac{1}{8} (
\gamma_{\tilde l k} \gamma^{p \tilde q} +
\delta_k^p \delta_{\tilde l}^{\tilde q}
)
- \frac{1}{4} \xi_k \xi^{\tilde q} \tilde\xi_{\tilde l}\tilde\xi^p
\\ & & \quad + \frac{1}{8}(
\gamma_{\tilde lk} \xi^{\tilde q}\tilde\xi^p +
\gamma^{p\tilde q} \xi_{k}\tilde\xi_{\tilde l} +
\delta_k^p \xi^{\tilde q}\tilde\xi_{\tilde l}+
\delta_{\tilde l}^{\tilde q} \xi_k\tilde\xi^p
)
]
\\ &=&
\frac{\dim\KK+2}{32}[\gamma_{\tilde ji}\gamma_{\tilde lk}+\gamma_{\tilde kj}\gamma_{\tilde li}]
- \frac{9\dim\KK}{8} \tilde N^m{}_{\tilde j\tilde l} N_{mik} \\
&-& \frac{2}{8} \left(\frac{\dim\KK+2}{4} \gamma_{\tilde lk}\gamma_{\tilde ji}
+ 9 N_{mik} \tilde N^m{}_{\tilde j\tilde l}\right) \\
&+& \frac{1}{64}[(3\dim\KK+3+2)\gamma_{\tilde ji}\gamma_{\tilde lk} + \gamma_{\tilde jk}\gamma_{\tilde li}] + \frac{1}{16} \xi_i\xi_k\tilde\xi_{\tilde j}\tilde\xi_{\tilde q} \\
&+& \frac{1}{64} [ 
2\gamma_{\tilde lk}  \xi_i\tilde\xi_{\tilde j}+ 
2\gamma_{\tilde ji} \xi_k \tilde\xi_{\tilde l} +
\gamma_{\tilde jk} \xi_i\tilde\xi_{\tilde l} + 
\gamma_{\tilde li} \xi_k\tilde\xi_{\tilde j} 
] \\
&+& \frac{1}{64}[
\gamma_{\tilde ji} \gamma_{\tilde lk} + 
(3\dim\KK+3) \xi_i\xi_k\tilde\xi_{\tilde j}\tilde\xi_{\tilde l} ].
\end{eqnarray*}
Using these we evaluate the components of $\cQ^2_{IJKL}:$
\begin{eqnarray*} \cQ^2_{ijkl} &=& \cQ_{ijpq} \cQ_{kl}{}^{pq} + \cQ_{ij\tilde p\tilde q} \cQ_{kl}{}^{\tilde p\tilde q} 
\\ &=&
\frac{\dim\KK+2}{48} [ \xi_{(i} N_{j)kl} + N_{ij(k} \xi_{l)} ]
\\ &=& \frac{\dim\KK+2}{24} \cQ_{ijkl}.
\end{eqnarray*}
\begin{eqnarray*} \cQ^2_{ij\tilde k\tilde l} &=& \cQ_{ijpq} \cQ_{\tilde k\tilde l}{}^{pq} + \cQ_{ij\tilde p\tilde q} \cQ_{\tilde k\tilde l}{}^{\tilde p \tilde q} \\
&=& \frac{\dim\KK+2}{16} N_{mij} \tilde N^m{}_{\tilde k\tilde l} + 
\frac{\dim\KK+3}{3^2 2^6} [  
\gamma_{\tilde ki}\xi_j\tilde\xi_{\tilde l} +
\gamma_{\tilde kj}\xi_i\tilde\xi_{\tilde l} +
\gamma_{\tilde li}\xi_j\tilde\xi_{\tilde k} +
\gamma_{\tilde lj}\xi_i\tilde\xi_{\tilde k} ] \\
&+& 
\frac{1}{3^2 2^7} (\gamma_{\tilde ki}\gamma_{\tilde lj} + \gamma_{\tilde li}\gamma_{\tilde kj})  +
\frac{1}{3^2 2^6} \xi_i\xi_j\tilde\xi_{\tilde k}\tilde\xi_{\tilde l} \\
&=& \frac{\dim\KK+2}{24} \cQ_{ij\tilde k\tilde l} 
+ \frac{\dim\KK+3}{3^2 2^7} [
2 \xi_i\xi_j\tilde\xi_{\tilde k}\tilde\xi_{\tilde l} + \gamma_{\tilde ki}\gamma_{\tilde lj}
+ \gamma_{\tilde li} \gamma_{\tilde kj} \\ & & \quad +
\gamma_{\tilde ki}\xi_j\tilde\xi_{\tilde l} +
\gamma_{\tilde kj}\xi_i\tilde\xi_{\tilde l} +
\gamma_{\tilde li}\xi_j\tilde\xi_{\tilde k} +
\gamma_{\tilde lj}\xi_i\tilde\xi_{\tilde k} ] 
\end{eqnarray*}
\begin{eqnarray*} \cQ^2_{i\tilde j k\tilde l} &=& - 2 \cQ_{i\tilde j\tilde pq} \cQ_{k\tilde l}{}^{\tilde p q} \\
&=& \frac{\dim\KK+2}{16} N_{mik} \tilde N^m{}_{\tilde j\tilde l} + \dots \\
&=& \frac{\dim\KK+2}{24} \cQ_{i\tilde j k\tilde l} 
+ \frac{\dim\KK+3}{3^2 2^7} [
2 \xi_i\xi_j\tilde\xi_{\tilde k}\tilde\xi_{\tilde l} + \gamma_{\tilde ki}\gamma_{\tilde lj}
+ \gamma_{\tilde li} \gamma_{\tilde kj} \\ & & \quad +
\gamma_{\tilde ki}\xi_j\tilde\xi_{\tilde l} +
\gamma_{\tilde kj}\xi_i\tilde\xi_{\tilde l} +
\gamma_{\tilde li}\xi_j\tilde\xi_{\tilde k} +
\gamma_{\tilde lj}\xi_i\tilde\xi_{\tilde k} ].
\end{eqnarray*}
Finally, it follows that
$$ \cQ^2_{IJKL} = \frac{\dim\KK+2}{24} \cQ_{IJKL} + \frac{\dim\KK+3}{3^2 2^7} (
\omega_{IJ}\omega_{KL} + \omega_{IL}\omega_{KJ}
),$$
which is equivalent to (\ref{idtt}).
\end{proof}

Passing to the complexification $\cV_3(\KK) = \CC\otimes\fxk,$ we readily have the following
\begin{cor}\label{corq}
Let us set $c_3=24\sqrt\frac{1}{\kappa+3},$ where $\kappa=\dim\KK.$ Recalling the conventions introduced
in Remark \ref{not-cpx}, we have
\begin{eqnarray}
\label{qw}
q_{\alpha\beta\mu\nu} \bar\omega^{\mu\nu} &=& 0 \\ 
\label{qq}
q_{\alpha\beta\mu\nu} \bar q^{\gamma\delta\mu\nu} &=& 
\frac{1}{2} [\delta_\alpha^\gamma \delta_\beta^\delta + \delta_\alpha^\delta \delta_\beta^\gamma]+ \chi\ q_{\alpha\beta\mu\nu} \bar\omega^{\mu\gamma}\bar\omega^{\nu\delta}.
\end{eqnarray}
Moreover, $J^* q = \bar q,$ i.e. $$
q_{\mu\nu\rho\sigma} \bar\omega^{\mu\alpha}\bar\omega^{\nu\beta}\bar\omega^{\rho\gamma}\bar\omega^{\sigma\delta} = \bar q^{\alpha\beta\gamma\delta}. $$
\end{cor}

These lead us directly to the missing proof:
\begin{proof}[Proof of Lemma \ref{lemq}] $ $
\begin{enumerate}
\item The only formula absent in Corollary \ref{corq} is $ q_{\alpha\mu\nu\rho} \bar q^{\beta\mu\nu\rho}
= \frac{N+1}{2}\delta^\beta_\alpha$ with $N=6\kappa+8$ being the complex dimension of $\cV_3(\KK).$ But
this follows by contraction from (\ref{qq}).
\item 
Let $b \in \Sym^{2,0}.$
Then $ \cD_q(b)_{\alpha\beta} = b_{\mu\nu} \bar q^{\mu\nu\rho\sigma} \omega_{\rho\alpha} \omega_{\sigma\beta} $
and
\begin{eqnarray*}
\cD_q^2(b)^\alpha{}_\beta
&= &
b^\mu{}_\nu q_{\phi\mu\rho\sigma} \bar\omega^{\rho\nu} \bar\omega^{\sigma\epsilon} q_{\beta\epsilon\xi\eta} \bar\omega^{\xi\phi} \bar\omega^{\eta\alpha} 
\\ &=&
- b^\mu{}_\nu q_{\mu\rho\sigma\phi}  \bar q^{\sigma\phi\eta\alpha}  \omega_{\eta\beta} \bar\omega^{\rho\nu}
\\ &=&
b^\alpha{}_\beta - \chi\  b^\mu{}_\nu 
q_{\mu\rho\beta\phi}  
\bar\omega^{\phi\alpha}
\bar\omega^{\rho\nu},
\end{eqnarray*}
so that $\cD_q^2 = 1 - \chi\cD_q.$
Then, since $\cD_q$ is hermitian, it
splits the space $\fsp(\cV_3(\KK),\omega)$ orthogonally into $\fsp^+(\cV_3(\KK),\omega) \oplus \fsp^-(\cV_3(\KK),\omega)$
with
$$ \cD_q|_{\fsp^\pm(\cV_3(\KK),\omega)} = \lambda^\pm\ \id,\quad \lambda^\pm = -\frac{\chi \pm \sqrt{\chi^2+4}}{2}. $$

Now, using (\ref{qw}), one easily checks that $\tr\cD_q = 0.$ It thus follows that
$$ \lambda^+\dim\fsp^+(\cV_3(\KK),\omega) +\lambda^- \dim\fsp^-(\cV_3(\KK),\omega)  = 0. $$
Using $ \dim\fsp(\cV_3(\KK),\omega) = \frac{N(N+1)}{2} $
we can compute the dimensions of $\fsp^\pm(\cV_3(\KK),\omega)$ and identify
$\fsp^+(\cV_3(\KK),\omega)$ as  $ \cg_3(\KK).$ 
Finally, it turns out that one can further simplify $\lambda^\pm$ and $\chi,$ so that
$$ \lambda_+ = -\sqrt{\kappa+3},\quad \lambda_- = -\frac{1}{\lambda_+}. $$
\end{enumerate}
\end{proof}

\subsection{Explicit formula for $\mho$}
We now wish to prove Lemma \ref{lem44}, which gave us an explicit expression for the projection
of $\mho$ onto $\Sym^{4,0}\otimes\Sym^{0,4}.$ We first state a simpler result, which
essentially expands the
result of applying the construction of Lemma \ref{lem-tl} to a one-form:
\begin{lem}
Let $F:\cV_3(\KK)\to\RR$ be a real linear map $F_a = f_\alpha + \bar f_{\bar\alpha}.$
Let $\Phi:\Sym^2\cV_3(\KK)\to\RR$ be the quadratic map given by
$$ \Phi(X,X) = 
F(X)^2 
+ F(I(X))^2 
+ F(J(X))^2 
+ F(K(X))^2. $$
Then $\Phi\in\Sym^{1,1}$ and 
$$ \frac{1}{2} \Phi = 
f\otimes\bar f + J^* \bar f \otimes J^* f
+ \bar f\otimes f + J^* f \otimes J^* \bar f
. $$ 
\end{lem}
\begin{proof} Using the expressions for $I,J,K,$ we have
\begin{eqnarray*}
\Phi
&=& f\otimes f + \bar f \otimes \bar f + f \otimes \bar f + \bar f \otimes f \\
&-& f\otimes f - \bar f \otimes \bar f + f \otimes \bar f + \bar f \otimes f \\
&+& J^* f\otimes J^* f + J^* f \otimes \bar J^* f + J^* f \otimes J^* \bar f + J^* \bar f \otimes J^* f \\
&-& J^* f\otimes J^* f - J^* f \otimes \bar J^* f + J^* f \otimes J^* \bar f + J^* \bar f \otimes J^* f,
\end{eqnarray*} reducing to the former expression. \end{proof}

With indices present, the formula reads 
$$  \Phi_{ab} = 2 f_\mu \bar f_{\bar\nu}
[ \delta^\mu_\alpha \delta^{\bar\nu}_{\bar\epsilon} 
- \omega^{\bar\nu}{}_\alpha \bar\omega^\mu{}_{\bar\epsilon} 
+ \delta^{\bar\nu}_{\bar\alpha} \delta^\mu_\epsilon 
- \bar\omega^\mu{}_{\bar\alpha} \omega^{\bar\nu}{}_\epsilon ]. $$
It is now easy to extend this result to multilinear maps. This leads us to the following
\begin{proof}[Proof of Lemma \ref{lem44}]
Let us introduce on $\cV_3(\KK)$ a representation of an orthonormal basis in $\HH:$ 
$L_1 = 1, L_2 = I, L_3 = J, L_4 = K.$ Recalling
the definition $\mho(Z,\dots,Z) = c_3^2\| \cQ_L(Z,Z,Z,Z) \|^2,$ we expand it as
$$ \mho(Z) = c_3^2 \sum_{ABCD} \tilde\cQ(L_A Z,L_B Z, L_C Z,L_D Z)^2, $$
with $A,B,C,D$ ranging from 1 to 4, where $\tilde\cQ$ is defined as in Lemma \ref{lem-tl}:
$$ \tilde\cQ(Z,Z,Z,Z) = |z|^4 \cQ(X,X,X,X) $$
for $Z = z\otimes X \in \CC\otimes\fxk.$

Perfofming the sums over $A,B,C,D$ separately, we can apply the former lemma to obtain
\begin{eqnarray*}
\mho_{abcdefkl} 
= & \frac{1}{70} q_{\mu\nu\rho\sigma} \bar q_{\bar\xi\bar\eta\bar\zeta\bar\tau} \lbrace &
[2 \delta^\mu_\alpha \delta^{\bar\xi}_{\bar\epsilon} + 2\omega^{\bar\xi}{}_\alpha \bar\omega^\mu{}_{\bar\epsilon}] 
[2 \delta^\nu_\beta \delta^{\bar\eta}_{\bar\phi} + 2\omega^{\bar\eta}{}_\beta \bar\omega^\nu{}_{\bar\phi}] \\
& &
[ 2\delta^\rho_\gamma \delta^{\bar\zeta}_{\bar\kappa} + 2\omega^{\bar\zeta}{}_\gamma \bar\omega^\rho{}_{\bar\kappa}] 
[ 2\delta^\sigma_\delta \delta^{\bar\tau}_{\bar\lambda} + 2\omega^{\bar\tau}{}_\delta \bar\omega^\sigma{}_{\bar\lambda}] 
\quad+ sym. \rbrace
\end{eqnarray*}
where symmetrization in $a\dots l$ is indicated, producing ${8\choose 4}=70$
terms on the right hand side. Projecting onto $\Sym^{4,0}\otimes
\Sym^{0,4}$ we obtain
\begin{eqnarray*}
\frac{70}{16} \mho_{\alpha\beta\gamma\delta\bar\epsilon\bar\phi\bar\kappa\bar\lambda} 
&=& q_{\alpha\beta\gamma\delta} \bar q_{\bar\epsilon\bar\phi\bar\kappa\bar\lambda} \\
&+& 4\ q_{\mu(\alpha\beta\gamma} \omega^{\bar\xi}{}_{\delta)} \bar\omega^\mu{}_{(\bar\epsilon}\bar q_{\bar\phi\bar\kappa\bar\lambda)\bar\xi}  \\
&+& 6\ q_{\mu\nu(\alpha\beta} 
\omega^{\bar\xi}{}_{\gamma} \omega^{\bar\eta}{}_{\delta)} 
\bar\omega^\mu{}_{(\bar\epsilon} \bar\omega^\nu{}_{\bar\phi}
\bar q_{\bar\kappa\bar\lambda)\bar\xi\bar\eta}\\
&+& 4\ q_{\mu\nu\rho(\alpha} 
\omega^{\bar\xi}{}_{\beta} \omega^{\bar\eta}{}_{\gamma} \omega^{\bar\zeta}{}_{\delta)} 
\bar\omega^\mu{}_{(\bar\epsilon} \bar\omega^\nu{}_{\bar\phi} \bar\omega^\rho{}_{\bar\kappa}
\bar q_{\bar\lambda)\bar\xi\bar\eta\bar\zeta}\\
&+&  q_{\mu\nu\rho\sigma} 
\omega^{\bar\xi}{}_{(\alpha} \omega^{\bar\eta}{}_{\beta} \omega^{\bar\zeta}{}_{\gamma} \omega^{\bar\tau}{}_{\delta)} 
\bar\omega^\mu{}_{(\bar\epsilon} \bar\omega^\nu{}_{\bar\phi} \bar\omega^\rho{}_{\bar\kappa} \bar\omega^\sigma{}_{\bar\lambda)}
\bar q_{\bar\xi\bar\eta\bar\zeta\bar\tau}.
\end{eqnarray*}
Now, using $\bar q = J^* q $ and $\bar\omega^\alpha{}_{\bar\mu}\omega^{\bar\mu}{}_\beta 
= -\delta^\alpha_\beta,$ we find that
\begin{eqnarray*} 
\frac{70}{16}
\omega^{\bar\epsilon}{}_\mu
\omega^{\bar\phi}{}_\nu
\omega^{\bar\kappa}{}_\rho
\omega^{\bar\lambda}{}_\sigma
\mho_{\alpha\beta\gamma\delta\bar\epsilon\bar\phi\bar\kappa\bar\lambda} 
&\propto& q_{\alpha\beta\gamma\delta} q_{\mu\nu\rho\sigma} \\
&-& 4\ q_{(\mu(\alpha\beta\gamma} q_{\delta)\nu\rho\sigma)} \\
&+& 6\ q_{(\mu\nu(\alpha\beta} q_{\gamma\delta)\rho\sigma)} \\
&-& 4\ q_{(\mu\nu\rho(\alpha} q_{\beta\gamma\delta)\sigma)} \\
&+& q_{\mu\nu\rho\sigma} q_{\alpha\beta\gamma\delta}, 
\end{eqnarray*}
where $(\alpha\beta\gamma\delta)$ and $(\mu\nu\rho\sigma)$ are symmetrized
separately.
It is now easy to see that the expression on the right hand side is $16 P_{44} (q\otimes q).$
Inverting the omegas on the left, we obtain the lemma.\end{proof}

\chapter{Geometric part}

\section{Summary of the algebraic results}
We wish to keep this part of our work possibly independent of the Jordan-algebra-related theory developed in the previous chapter. We will now recall the results we need, in a form which claims the existence of certain tensors subject to a number of identities.

Still, the symbols $\KK$ and $\KK'$ denote, respectively, one of $\RR,\CC,\HH,\OO$ and one of $\CC,\HH,\OO.$ Pairs $(\KK,\KK')$ enumerate compact Riemannian symmetric spaces collected in the following table:
\begin{center}
\begin{tabular}{cc|ccc}
 & $\KK'$ & $\CC$ & $\HH$ & $\OO$ \\
 $\KK$ & & & & \\ \hline
 $\RR$ & &
 $\frac{\sSU(3)}{\sSO(3)}$ &
 $\frac{\sSp(3)}{\sU(3)}$ &
 $\frac{\sF_4}{\sSp(3)\sSp(1)}$ \\
 $\CC$ & &
 $\frac{\sSU(3)\times\sSU(3)}{\sSU(3)}$ &
 $\frac{\sSU(6)}{S(\sU(3)\times\sU(3))}$ &
 $\frac{\sE_6}{\sSU(6)\sSp(1)}$ \\
 $\HH$ & &
 $\frac{\sSU(6)}{\sSp(3)}$ &
 $\frac{\sSO(12)}{\sU(6)}$ &
 $\frac{\sE_7}{\sSO(12)\sSp(1)}$ \\
 $\OO$ & &
 $\frac{\sE_6}{\sF_{4}}$ &
 $\frac{\sE_7}{\sE_6\times\sU(1)}$ &
 $\frac{\sE_8}{\sE_7\sSp(1)}$
\end{tabular}
\end{center}
Let $M_S(\KK,\KK')$ denote the corresponding symmetric space and $G(\KK,\KK')$ its underlying
isotropy group, with a Lie algebra $\fg(\KK,\KK').$ The isotropy representation is denoted $V(\KK,\KK').$

Abstract index notation, including the conventions introduced in Remark \ref{not-cpx}, is assumed.
Moreover, for given $\KK$ and $\KK',$ the following symbols are defined:
$$ \kappa = \dim\KK,\quad 
N = \left\lbrace\begin{matrix} 
\dim M_S(\KK,\CC) &=& 3\kappa+2 & \quad \textrm{for}\ \KK'=\CC \\
\frac{1}{2}\dim M_S(\KK,\HH) &=& 3\kappa+3 & \quad \textrm{for}\ \KK'=\HH \\
\frac{1}{2}\dim M_S(\KK,\OO) &=& 6\kappa+8 & \quad \textrm{for}\ \KK'=\OO
\end{matrix}\right. $$
and $ \chi = \frac{\kappa+2}{\sqrt{\kappa+3}} $ for $\KK'=\OO.$

The following statements are reformulations or simple corollaries of: Corollary \ref{cor-reps}, Proposition \ref{pro-invs},
Lemmas \ref{lem44}, \ref{lemq}, \ref{lemdy}, \ref{lemdl}, and Corollaries \ref{cor-idy}, 
\ref{cor-id2}.

\subsection{First family}
Set $\KK'=\CC$ and choose $\KK.$  Let $\cV_1(\KK)$ be a real vector space of dimension $N$
equipped with a positive definite scalar product $g$ identifying $\cV_1(\KK)\simeq\cV_1(\KK)^*.$

There exists a tensor $\Upsilon \in \Sym^3 \cV_1(\KK)$ reducing the group $\sGL(N)$ to
$G(\KK,\CC)$ acting on $\cV_1(\KK)$ in the isotropy representation of $M_S(\KK,\CC),$
with $g$ as the preserved scalar product.

This tensor moreover satisfies the following identities:
\begin{eqnarray*}
\Upsilon_i{}^{mm} &=& 0 \\
\Upsilon_i{}^{mn} \Upsilon_j{}^{mn} &=& g_{ij} \\
\Upsilon_{i}{}^{rp} \Upsilon_j{}^{pq} \Upsilon_{k}{}^{qr} &=& \frac{-3\kappa}{6\kappa+8} \Upsilon_{ijk} \\
\Upsilon_{(ij}{}^m \Upsilon_{kl)}{}^m &=& \frac{10}{3\kappa+4} g_{(ij} g_{kl)}.
\end{eqnarray*}
The latter shows that, up to rescaling, $\Upsilon$ is the same tensor as the one
used by Nurowski in \cite{nurowski-2006}.

An operator $$\cD_\Upsilon : \Lam^2 \cV_1(\KK) \to \Lam^2 \cV_1(\KK)$$
$$ \cD_\Upsilon(E)_{ij} = E_{pq} \Upsilon_{iqm} \Upsilon_{jpm} $$ acts as
$$ \cD_\Upsilon = - \frac{1}{2} \pr_\fg + \frac{3}{3\kappa+4} \pr_\bot $$
where $\pr_\fg$ and $\pr_\bot$ are projections corresponding to the orthogonal decomposition
$$ \Lam^2 \cV_1(\KK) = \fg(\KK,\CC) \oplus \bot, $$
with $$\fg(\KK,\CC) \subset \fso(\cV_1(\KK),g)\simeq\Lam^2 \cV_1(\KK)$$ being the isotropy
algebra of $\Upsilon.$

\subsection{Second family}
Set $\KK'=\HH$ and choose $\KK.$ Let $\cV_2(\KK)$ be a complex vector space of (complex) dimension
$N$ equipped with a hermitian inner product $h,$ giving rise to a positive definite real scalar product $g$
on (the realification of) $\cV_2(\KK).$

Let $\Sym^{p,q}$ and $\Lam^{p,q}$ denote the spaces
of $p$-linear $q$-antilinear, respectively symmetric and antisymmetric, complex-valued
forms on $\cV_2(\KK).$

There exists a real tensor $\Xi \in \Re\Sym^{3,3}$ reducing the group $\sO(\cV_2(\KK),g)$ to a subgroup
whose connected component is $G(\KK,\HH),$ acting on $\cV_2(\KK)$ in the isotropy representation of $M_S(\KK,\HH),$ with $g$ as the preserved scalar product.

Recall that $h$ gives rise to an identification
$$ \fu(\cV_2(\KK),h) \simeq \Re\Lam^{1,1} = \theta\RR \oplus \fsu(\cV_2(\KK),h), $$
where $\theta_{ab} = ih_{\bar\alpha\beta} - ih_{\bar\beta\alpha}.$ In particular,
the orthogonal isotropy algebra of
$\Xi$ is a subalgebra in $\fu(\cV_2(\KK),h).$

There moreover exists a tensor $\Lambda \in \Sym^{3,0}$ such that $ \Xi = \Lambda \cdot \bar\Lambda,$
satisfying the following identities:
\begin{eqnarray*} \
\Lambda_{\alpha\mu\nu} \bar\Lambda_{\bar\beta}{}^{\mu\nu} &=& h_{\bar\beta\alpha} \\ 
\Lambda_{\alpha\mu\nu} \bar\Lambda_{\bar\beta}{}^{\nu\rho} 
\Lambda_{\gamma\rho\sigma} \bar\Lambda_{\bar\delta}{}^{\sigma\mu} &=&
\frac{1}{2\kappa+4} (h_{\bar\beta\alpha}h_{\bar\delta\gamma}+h_{\bar\delta\alpha} h_{\bar\beta\gamma}) \\
&-& \frac{\kappa}{2\kappa+4}\Lambda_{\alpha\gamma\mu} \bar\Lambda_{\bar\beta\bar\delta}{}^\mu.
\end{eqnarray*}
The algebra $\fg(\KK,\HH)$ is a direct sum 
$$ \fg(\KK,\HH) = \cg_2(\KK) \oplus \fu(1) $$
of the orthogonal stabilizer algebra
of $\Lambda$ and a $\fu(1)$ spanned by multiplication by $i.$ The latter is the centre of $\fg(\KK,\HH).$

There is a $G(\KK,\HH)$-invariant operator
$$ \cD_\Lambda : \Lam^{1,1} \to \Lam^{1,1} $$
$$ \cD_\Lambda(B)_{ab} = 
b_{\bar\mu\nu} \Lambda^{\bar\mu\bar\rho}{}_\beta \bar\Lambda^\nu{}_{\bar\rho\bar\alpha} 
- b_{\bar\mu\nu} \Lambda^{\bar\mu\bar\rho}{}_\alpha \bar\Lambda^\nu{}_{\bar\rho\bar\beta} 
$$
for $B_{ab} = b_{\bar\alpha\beta} - b_{\bar\beta\alpha},$ acting as
$$ \cD_\Lambda = \pr_0 - \frac{1}{2} \pr_\cg + \frac{1}{\kappa+2} \pr_\bot, $$
where $\pr_0,$ $\pr_\cg$ and $\pr_\bot$ are projections corresponding to the orthogonal decomposition
into $G(\KK,\HH)$-invariant spaces:
$$ \Lam^{1,1} = \CC\otimes\fu(1)\ \oplus\ \CC\otimes\cg_2(\KK)\ \oplus\ \bot. $$

\subsection{Third family}
Set $\KK'=\OO$ and choose $\KK.$ Let $\cV_3(\KK)$ be a real vector space of dimension $2N$ equipped
with a positive definite scalar product $g$ and a quaternion-hermitian structure, i.e. three hermitian
structures $I,J,K$ subject to
$$ I^2 = J^2 = K^2 = IJK = -\id. $$

There exists a tensor $\mho \in \Sym^8 \cV_3(\KK)$ reducing the group $\sO(\cV_3(\KK),g)$ to
a subgroup whose connected component is $G(\KK,\OO),$ acting on $\cV_3(\KK)$ in the isotropy representation
of $M_S(\KK,\OO),$ with $g$ as the preserved scalar product.

Consider now $\cV_3(\KK)$ as a complex vector space of dimension $N,$ with one of the three hermitian
structures, for concreteness $I,$ as \emph{the} complex structure. 
Let $\Sym^{p,q}$ and $\Lam^{p,q}$ denote the spaces
of $p$-linear $q$-antilinear, respectively symmetric and antisymmetric, complex-valued
forms. By polarisation, the scalar product $g$ gives rise to a hermitian inner product $h\in\Sym^{1,1}.$
Moreover, by the identification 
$$\fso(\cV_3(\KK,g)) \simeq \Lam^2 \cV_3(\KK,g) = \Re (\Lam^{2,0} \oplus \Lam^{1,1} \oplus \Lam^{0,2}),$$
one has a symplectic form $\omega_{\alpha\beta} \in \Lam^{2,0}_{ab}$ such that
$$J \in \Re(\Lam^{2,0} \oplus \Lam^{0,2}),\quad J^a{}_b = \omega^{\bar\alpha}{}_\beta 
+ \bar\omega^\alpha{}_{\bar\beta}. $$

In this context, $\mho\in\Re\Sym^{4,4}.$ There moreover exists a tensor $q\in\Sym^{4,0}$ such that the
projection of $\mho\in\Sym^{4,4}$ onto $\Sym^{4,0} \otimes \Sym^{0,4}$ is
$$
(\mho|_{\Sym^{4,0}\otimes\Sym^{0,4}})_{\alpha\beta\gamma\delta\bar\epsilon\bar\phi\bar\kappa\bar\lambda}
=  \frac{256}{70}\ 
\delta^{[\mu}_{(\alpha}
\bar\omega^{\xi]}{}_{(\bar\epsilon}
\delta^{[\nu}_\beta
\bar\omega^{\eta]}{}_{\bar\phi}
\delta^{[\rho}_\gamma
\bar\omega^{\zeta]}{}_{\bar\kappa}
\delta^{[\sigma}_{\delta)}
\bar\omega^{\tau]}{}_{\bar\lambda)}
\ q_{\mu\nu\rho\sigma} q_{\xi\eta\zeta\tau},$$
where barred and unbarred indices are symmetrized separately.
Following identities are satisfied:
\begin{eqnarray}
q_{\alpha\beta\mu\nu} \bar\omega^{\mu\nu} &=& 0 \\ 
q_{\alpha\mu\nu\rho} \bar q^{\beta\mu\nu\rho} &=& \frac{N+1}{2} \delta^\alpha_\beta \\ 
q_{\alpha\beta\mu\nu} \bar q^{\gamma\delta\mu\nu} &=& 
\frac{1}{2} [\delta_\alpha^\gamma \delta_\beta^\delta + \delta_\alpha^\delta \delta_\beta^\gamma]+ \chi\ q_{\alpha\beta\mu\nu} \bar\omega^{\mu\gamma}\bar\omega^{\nu\delta} \\
q_{\mu\nu\rho\sigma} \bar\omega^{\mu\alpha}\bar\omega^{\nu\beta}\bar\omega^{\rho\gamma}\bar\omega^{\sigma\delta} &=& \bar q^{\alpha\beta\gamma\delta}. 
\end{eqnarray}

Let $\sSp(\cV_3(\KK),\omega)$ denote the symplectic group preserving the quaternionic structure,
and $\sSp(1)$ \emph{the} group generated by $I,J,K.$
Let $\fsp(\cV_3(\KK),\omega)$ and $\fsp(1)$ be the corresponding  Lie algebras. 
From the complex viewpoint, with $I$ the distinguished
complex structure, one has
$$ \fsp(\cV_3(\KK),\omega) \subset \fu(\cV_3(\KK),h) \simeq \Re\Lam^{1,1}. $$
Moreover, $\omega$ gives rise to an identification
$$ \Lam^{1,1} \supset \CC\otimes\fsp(\cV_3(\KK),\omega) \simeq \Sym^{2,0} $$
via 
$$ \Lam^{1,1}_{ab} \ni b_{\bar\alpha\beta} - b_{\bar\beta\alpha} \mapsto b^\mu{}_\alpha\omega_{\beta\mu}
\in \Lam^{1,0}_a\otimes \Lam^{1,0}_b. $$
The orthogonal isotropy algebra of $\mho$ can be decomposed as $$\fg(\KK,\OO) = \cg_3(\KK) \oplus 
\fsp(1)$$ where $\cg_3(\KK)$ is the orthogonal stabilizer of $q.$

An operator
$$ \cD_q : \Sym^{2,0} \to \Sym^{2,0} $$
$$ \cD_q(b)_{\alpha\beta} = b_{\mu\nu} \bar q^{\mu\nu\rho\sigma} \omega_{\rho\alpha} \omega_{\sigma\beta}
$$
acts as
$$ \cD_q = -\sqrt{\kappa+3}\ \pr_\fg + \sqrt{\frac{1}{\kappa+3}}\ \pr_\bot, $$
where $\pr_\fg$ and $\pr_\bot$ are projections corresponding to the decomposition
$$ \Sym^{2,0} \simeq \CC \otimes \cg_3(\KK) \oplus \bot. $$

\section{$G-$structures and intrinsic torsion}

$G$-structures provide a covariant description of additional structure introduced on a manifold
(be it Riemannian, complex, CR etc.).
While they are defined as reductions of the frame bundle, or simply principal bundles equipped with
a soldering form, related objects are pulled back to the mainfold and its tangent bundle by means
of local sections, i.e. adapted frames.

One can also introduce somewhat weaker structure directly on the tangent bundle, specifying
endomorphisms of the latter corresponding to the Lie algebra $\fg$ of the original structure group $G.$ Under certain technical assumption, the
latter can be then reconstructed up to connected components, so that in general 
one ends up, locally, with several associated $G-$structures (although these may fail to be globally defined on a manifold which is not simply connected).

Finally, we shall discuss structures defined in terms of invariant tensors. However,
the groups $G(\KK,\KK')$ we are interested in are in general defined only as identity components of the isotropy groups. Selecting a $G(\KK,\KK')-$structure requires thus an extra choice of a frame in a single point of each connected component of the manifold, and may fail due to global problems even if the tensor is globally defined.

\subsection{$G-$structures and the intrinsic torsion} \label{ss-gstr}
Recall, that a $G-$structure $Q$ on an $m-$dimensional Riemannian
manifold $(M,g),$ for $G$ being a subgroup
of $\sO(m),$ is a reduction $$ i: Q \hookrightarrow P$$ of the orthonormal frame bundle $\pi : P \to M$
to a subbundle $Q$
with structure group $G,$ acting on $Q$ by restriction of the action of $\sO(m)$ on $P.$ Its
local sections will be called \emph{adapted frames}.

A connection on the bundle $P$ 
is said to be compatible with  $Q$ iff 
the associated horizontal distribution ${\rm Hor} \subset TP$ is
tangent to $Q:$ $${\rm Hor}_{q} \subset T_q Q\quad \textrm{for each}\ q\in Q.$$
In terms of a connection form $\omega \in \Omega^1(P)\otimes\fso(m),$ the compatibility
condition reads simply\footnote{
Indeed, consider a vector $X\in T_q Q\subset T_q P.$ Then $X - \widehat{\omega(X)}_q$
is horizontal, where $\hat A \in \mathcal{X}(P)$ denotes the vertical vector field
associated to $A\in\fso(m)$ by the structure group action. Demanding ${\rm Hor}_q
\subset T_q Q$ implies $\widehat{\omega(X)}_q \in T_q Q,$ so that $\omega(X)\in\fg.$
}
$$ i^* \omega \in \Omega^1(Q) \otimes \fg, $$
where $\fg\subset\fso(m)$ is the Lie algebra of $G.$
If it is satisfied, $i^*\omega$ becomes
a connection on $Q$ (since $G$ acts on $Q$ by restriction of its action on $P$).

\begin{defn} A $G-$structure is said to be \emph{integrable} iff it admits a torsion-free compatible
connection.
\end{defn}

Following Agricola \cite{agricola-2006}, we consider the Levi-Civita connection $ \omega^{LC} 
\in \Omega_1(P)\otimes\fso(m)$ and decompose its restriction to $Q$ orthogonally:
$$\fso(m) = \fg \oplus \ft,\quad [\fg,\ft]\subset\ft $$
\begin{equation} i^*\omega^{LC} = \omega^\fg - \alpha^\ft \label{omega-dec}\end{equation}
so that $\omega^\fg \in \Omega^1(Q)\otimes\fg$ and $\alpha^\ft \in \Omega^1(Q)\otimes\ft.$
\begin{lem}[cf. \cite{agricola-2006}] $ $
\begin{enumerate}
\item $\omega^\fg$ is a connection on $Q.$
\item $\alpha^\ft$ is a horizontal one-form of type $\Ad$ on $Q.$
\end{enumerate}
\end{lem}
The usual name for the $\alpha^\ft$ is the \emph{intrinsic torsion}
of $Q$ (in fact, this notion can be defined also for a general $G-$structure, not necessarily
Riemannian, although it becomes less explicit \cite{joyce-x}).

The form $\omega^\fg$ can be extended by covariance to a connection on entire $P,$ with
values in the orthogonal orbit of $\fg:$  $$ \tilde\omega^\fg \in \Omega_1(P) \otimes \Ad_{\sO(m)} \fg.$$
Let $\tilde D^\fg$ be the associated covariant derivative:
$$ \tilde D^\fg : \Omega^p_{hor}(P)\otimes H  \to \Omega^{p+1}_{hor}(P)\otimes H $$
$$ \tilde D^\fg \phi = d\phi + \tilde\omega^\fg \wedge \phi $$
for any $\sO(m)$-module $H.$ 
Recall that a point $p$ in the frame bundle $P$ is an orthogonal isomorphism $p : \RR^m \to T_{\pi(p)}M,$ where $\RR^m$ is equipped with some fixed positive definite scalar product.
Introducing the soldering form 
$$\theta \in \Omega^1(P)\otimes\RR^m,\quad \theta_p = p^{-1} \circ T_p\pi\quad\textrm{for each}\ p\in P, $$
we have the torsion of $\omega^\fg:$
$$ \Theta^\fg = \tilde D^\fg \theta = d\theta + \tilde\omega^\fg\wedge\theta
\in \Omega^2(P)\otimes\RR^m. $$
Recalling that the Levi-Civita connection is torsion-free and using (\ref{omega-dec}), one finds
\begin{eqnarray*}
i^* \Theta^\fg &=& i^* d\theta + \omega^\fg\wedge i^*\theta \\ &=& i^* (d\theta + \omega^{LC} \wedge\theta
) + \alpha^\ft \wedge i^*\theta \\ &=& \alpha^\ft\wedge i^*\theta.
\end{eqnarray*}
It then follows that $$\Theta^\fg(X,Y) = \alpha^\ft(X) (\theta(Y)) - \alpha^\ft(Y) (\theta(X)) $$ 
for $X,Y\in T_q Q.$ Conversely, it turns out that $\Theta^\fg$ determines the
horizontal form $\alpha^\ft$
(it needn't be true in the non-Riemannian case, cf. \cite{agricola-2006}):
\begin{eqnarray*} 2 \langle\ \alpha^\ft(X)(\theta(Y)),\ \theta(Z)\ \rangle 
&=& \langle\ \Theta^\fg(X,Y),\ \theta (Z)\ \rangle \\
&-& \langle\ \Theta^\fg(Y,Z),\ \theta (X)\ \rangle \\
&+& \langle\ \Theta^\fg(Z,X),\ \theta (Y)\ \rangle \end{eqnarray*}
for $X,Y,Z\in T_q Q,$ where $\langle \cdot,\cdot\rangle$ denotes the scalar product
on $\RR^m.$ 

Note that a $Q-$compatible connection is necessarily metric. Thus a torsion-free $Q-$compatible
connection, provided it exists, necessarily coincides with the Levi-Civita one.
This leads us to the following obvious
\begin{pro}[cf. \cite{agricola-2006,joyce-x}] Let $Q$ be a $G-$structure on an $m-$dimensional Riemannian manifold $(M,g),$ 
with $G\subset\sO(m)$ The following are equivalent:
\begin{enumerate}
\item{$Q$ is integrable.}
\item{The intrinsic torsion of $Q$ is trivial.}
\item{The Levi-Civita connection on $M$ is compatible with $Q.$}
\end{enumerate}
Each of these implies reduction of the Riemannian holonomy group ${\rm Hol}(g)$ to
a subgroup of $G.$
\end{pro}
Particular examples one should have in mind are the almost-hermitian and almost-quaternion-hermitian
structures, with structure groups $\sU(\frac{m}{2})$ and $\sSp(\frac{m}{4})\sSp(1)$ respectively
(their integrable cases being featured in the celebrated theorem of Berger). The geometries modelled
on the second and third family of symmetric spaces described so far fall into these categories.

\subsection{Tangent bundle view}

Choosing a local section $e:M\supset U \to Q,$ i.e. an adapted frame, one can pull $\omega^\fg$ and
$\alpha^\ft$ back to $M,$ so that $\omega^\fg$ gives rise to a connection
in the tangent bundle, and $\alpha^\ft$ becomes a well-defined tensorial one-form.

It is however instructive to consider the structure which appears
on the tangent bundle independently of any section. In what follows, we use $g$ to identify
$TM\simeq T^*M.$
Observe first that, having
fixed a point $x\in M,$ the 
image of $\fg\subset\Lam^2 \RR^m$ under (the extension of) $ q : \RR^m \to T_x M,$ where
$q\in Q$ and $\pi(q) = x,$ does not depend on the choice of $q$ from
the fibre of $Q$ over $x.$ Indeed, for every two frames $q, q' \in \pi|_Q^{-1}(x)$
there exists $g\in G$ such that $q' = q \circ g$; on the other hand, $\Ad_g(\fg) = \fg,$
so that $q(\fg) = q'(\fg).$ We shall denote this subspace as $\fg_M(x) \subset \Lam^2 T_x M.$

It thus follows, that a $G-$structure equips $M$ with an orthogonal splitting
\begin{equation} \Lam^2 TM = \fg_M \oplus \ft_M \label{tsplit} \end{equation}
into subbundles
such that at each point $x\in M$ the subspace 
$$\fg_M(x) \subset \Lam^2 T_x M \simeq \fso(T_x M, g_x)$$
is a Lie subalgebra isomorphic to $\fg.$

A natural question is to what extent can one reconstruct the original $G-$structure from
the data given by (\ref{tsplit}). In what follows we shall assume that:
\begin{enumerate}
\item $G$ is connected
\item $\ft$ contains no
one-dimensional $G-$invariant subspace
\end{enumerate}
(these are true for $G(\KK,\KK'),$ as one can check in Section \ref{sec-dec}.
Point 2 is needed in particular for the following result to hold:
\begin{lem}\label{lem-gm}
The set of all orthonormal frames mapping $\fg$ to $\fg_M$ forms a principal
bundle with a structure group whose identity component is $G.$
\end{lem}
\begin{proof}
Consider an orthonormal frame $e_x : \RR^m \to T_x M$ at a fixed point $x\in M,$
such that $e_x(\fg) = \fg_M(x).$ Every other frame at $x$ is of the form $e_x' = e_x \circ a$
for some $a\in\sO(m),$ so that $e_x'$ maps $\fg$ to $\fg_M(x)$ iff $\Ad_a(\fg) = \fg.$ It
thus follows that all such frames at $x$ form a fibre of a principal bundle with
structure group
$$ \tilde G = \{ a \in \sO(m)\ |\ \Ad_a(\fg) = \fg \}. $$
The corresponding Lie algebra is
$$ \tilde\fg = \{ A \in \fso(m)\ |\ [A,\fg]\subset\fg \}.$$
Now, every $A\in\fso(m)$ is of the form $A_1 + A_2$ with $A_1\in\fg$ and $A_2\in\ft,$
so that $[A,\fg] = [A_1,\fg] + [A_2,\fg],$ where automatically $[A_1,\fg]\subset \fg$
and $[A_2,\fg] \subset \ft,$ since $[\fg,\ft]\subset\ft.$
Thus $[A,\fg]\subset\fg$ iff $[A_2,\fg] = 0.$ Then the assumption we have made on $\ft$
implies that every such $A_2$ is zero, so that finally $\tilde\fg=\fg,$ and $\tilde G$
has $G$ as its component of identity.
\end{proof}
One moreover obtains  a characterization of compatible connections:
\begin{lem}
A metric connection $\nabla$ in the tangent bundle, lifted to a connection on the frame bundle,
is compatible with the $G-$structure iff it preserves the splitting (\ref{tsplit}):
$$ f \in \Omega^0(M,\fg_M) \implies \nabla f \in \Omega^1(M,\fg_M). $$
\end{lem}
\begin{proof}
Choose (locally) an adapted frame, i.e. a (local) section $e : M \to Q.$ A compatibility
condition for a connection form $\omega$ on the frame bundle $P$ is $ i^* \omega \in \Omega^1(Q)\otimes \fg.$ Consider now the corresponding connection $\nabla^\omega$ in the frame bundle,
whose local connection form in the adapted frame $e$ is
$$ \Gamma^\omega\in\Omega^1(M,\Lam^2 TM),
\quad \Gamma^\omega(X) = e_x(e^*\omega(X))\ \textrm{for}\ X\in T_xM. $$
Then the compatibility condition is equivalent to $\Gamma^\omega \in \Omega^1(M,\fg_M).$

On the other hand, $\nabla^\omega$ preserves $\fg_M$ iff for each $X\in T_x M$
$$ [ \Gamma^\omega(X), \fg_M(x) ] \subset \fg_M(x). $$
From the assumption on $\ft,$ via an argument given in the proof of Lemma \ref{lem-gm},
it follows that $\nabla^\omega$ preserves $\fg_M$ iff $\Gamma^\omega\in\Omega^1(M,\fg_M).$
Hence the lemma.
\end{proof}
\begin{cor}
The $G-$structure is integrable iff 
the Levi-Civita derivative of each section of $\fg_M$
is $\fg_M-$valued.
\end{cor}

Recall now, that a difference of two connections in the tangent bundle is a well-defined
tensor\footnote{
By $A(X)(Y)$ we mean the element of $\End TM,$ obtained by evaluation of $A$ on $X$, acting on $Y$.
}:
$$ \nabla_X Y - \nabla'_X Y = A(X) (Y),\quad A \in \Omega^1(M,\End\ TM), $$
where in particular $A\in\Omega^1(M,\Lam^2 M)$ for metric connections.
We have the following
\begin{lem}
There exists a unique metric connection $\nabla^\fg$ preserving $\fg_M$ such that
the difference tensor of $\nabla^\fg$ and the Levi-Civita connection $\nabla^{LC}$
is $\ft_M-$valued, i.e.  $$
\exists A^\ft\in\Omega^1(M,\ft_M) 
\ \forall X,Y\in\mathcal{X}(M)\ :
\ \nabla_X^\fg Y - \nabla^{LC}_X Y = A^\ft(X)(Y). $$
\end{lem}
\begin{proof} Choose (locally) an adapted frame $e.$ Then the local connection
form of $\nabla^{LC}$ is uniquely decomposed with respect to
$$ \Lam^2 TM = \fg_M \oplus \ft_M $$
$$ \Gamma^{LC} = \Gamma^\fg - A^\ft, $$
where $\Gamma^\fg\in\Omega^1(M,\fg_M)$ and $A^\ft\in\Omega^1(M,\ft_M).$
One easily checks that $\Gamma^\fg$ transforms as a connection form corresponding
to some connection $\nabla^\fg$ in the tangent bundle, while $A^\ft$ is a well defined
tensor. It further follows that, irrespectively of the frame,
$$ \nabla^\fg_X Y - \Gamma^{LC}_X Y = A^\ft(X)(Y). $$
\end{proof}
It is not difficult to check, that $A^\ft\in \Omega^1(M,\ft_M)$ coincides
with a pull-back by an arbitrary section $e:M\to Q$
of the intrinsic torsion $\alpha^\ft$ introduced in the previous subsection:
$$ A^\ft(X) = e_x (e^* \alpha^\ft(X))\quad\textrm{for}\ X\in T_xM $$
As the reader may expect, $A^\ft$ is equivalent to the torsion $T^\fg$
of $\nabla^\fg:$
$$ T^\fg(X,Y) = A^\ft(X)(Y) - A^\ft(Y)(X), $$
$$ 2g(A^\ft(X)(Y),Z) = g(T^\fg(X,Y),Z) - g(T^\fg(Y,Z),X) + g(T^\fg(Z,X),Y). $$

\subsection{Invariant tensors} \label{ss-its}
Finally, we consider the following structure: let $\tilde G \subset \sO(m)$ be the orthogonal
stabilizer group of a tensor $\Y\in\otimes^p\RR^m,$ with $G$ its identity component and $\fg$ the Lie algebra. Then a $\tilde G-$structure on an $m-$dimensional manifold $(M,g)$ can be defined
by a tensor $$ \Y_M \in \otimes^p TM $$ such that at each point $x\in M$ there
exists an orthonormal frame $e_x : \RR^m \to T_x M$ such that $e_x(\Y) = \Y_M(x).$ Clearly, the set
of all such frames at $x$ forms a fibre of a principal bundle with structure group $\tilde G,$
a reduction $\tilde i : \tilde Q \to P$ of the orthonormal frame bundle.

In a typical situation, we will however be interested in having a $G-$structure $Q$, where $G$
is the identity component of possibly not connected $\tilde G.$
A reduction of $\tilde Q$ to $Q$ is always possible locally (i.e. on $\pi^{-1}(U)$
for $U$ a neighbourhood of a point on $M$), by choosing a single point $q\in\tilde Q,$
declaring it to be a member of $Q$ and identifying the latter by continuity. It may however
fail to yield a $G-$structure over entire $M$, even if $\Y_M,$ and thus $\tilde Q,$ are globally
defined. Such a situation is illustrated by the following:

\begin{eg} \label{eg-noG} A simple
flat geometry related locally to the symmetric space $\sSp(3) / \sU(3).$
Consider first the (real) space 
$M_0=\CC^6$ whose tangent spaces carry a natural complex structure and
a fixed complex-linear identification with $\cV_2(\RR)\simeq\CC^6,$ so that they become equipped
with a metric $g_{M_0}$ and a parallel tensor $\Xi_{M_0}$ being a real section
of $\Sym^{3,3} TM_0.$ The latter defines a trivial $\tilde G-$structure on $M_0,$ and
restricting to holomorphic frames gives a reduction to a trivial $G-$structure, $G=\sU(3)$.

Introduce now the natural complex coordinates $z^1,\dots,z^6 : \CC^6 \to \CC,$ such that $g_{M_0} 
= \sum_{i=1}^6 dz^i d\bar z^i,$ and consider a (real) manifold $M$ obtained by identifying
points subject to the equivalence relation
$$ (z^1,z^2,\dots,z^6) \sim (\bar z^1 + 2\pi, \bar z^2, \dots, \bar z^6). $$
Considering the projection $p:M_0 \to M$ one finds that $M$ inherits a metric and orientation
(complex conjugation on a space of even complex dimension has determinant one). The tensor
$\Xi_{M_0}$ is also uniquely pushed forward to $\Xi_M$ on $M,$ since $\bar\Xi = \Xi.$

We thus have a $\tilde G-$ structure on $M$, where $\tilde G$ is the full orthogonal 
isotropy group of $\Xi,$ with
$\sU(3)$ as its identity component. However, as we have just noted, $\tilde G$ also
contains an antiunitary component, in particular the complex conjugation map, $ C
: \CC^6 \to \CC^6.$
A local $\sU(3)-$structure can be defined by picking a frame at a single point, say
the one obtained by projecting the natural holomorphic frame at the origin of $M_0=\CC^6$ to
$$ e_{p(0)} : \CC^6 \to T_{p(0)} M. $$
However, parallel transporting $e_{p(0)}$ along a closed loop, given by the projection of
a curve joining $(0,\dots,0)$ with $(2\pi,\dots,0)$ in $M_0,$ one ends with $ e_{p(0)} \circ
C,$ which clearly belongs to a different connected component of the fibre than the original frame.
One thus fails to define a global $\sU(3)$ structure (which is already implied by the fact
that the complex structure of $M_0$ does not descend to $M$). Note that this problem is not resolved by fixing an orientation.
\end{eg}

Nevertheless, ignoring global questions and considering \emph{any}
of the local $G-$structures defined by
$\Y_M,$ we can recover most of the local information about the former from geometric data given
by the latter (i.e. $\Y_M$ and its Levi-Civita derivatives).
Note first, that in an adapted frame $e$ we have $e_x^{-1} \Y_M(x) = \Y,$ a constant $\otimes^p \RR^m-$valued function, so that 
$$ \nabla_X \Y_M = \Gamma(X)(\Y_M) $$
for a connection $\nabla$ with local connection form $\Gamma\in\Omega^1(M,\Lam^2 TM).$ This leads to the following
\begin{lem}
A connection $\nabla$ in the tangent bundle is compatible with a local $G-$structure defined
by $\Y_M$ iff $ \nabla\Y_M = 0.$
\end{lem}
\begin{proof}
Choose an adapted frame $e.$ Then for $X\in T_x M$ 
$$ \nabla_X \Y_M = \Gamma(X)(\Y_M) = 0 \iff \Gamma(X) \in \fg_M.$$
The latter is the compatibility condition for $\nabla.$
\end{proof}

\begin{cor} A local $G-$structure defined by $\Y_M$ is integrable iff $\nabla^{LC}\Y_M = 0.$
\end{cor}

While $\Y_M$ being parallel implies vanishing of the intrinsic torsion, it turns out
that entire $A^\ft$ can be reconstructed from $\nabla^{LC}\Y_M.$ Since $\ft$ is the complement
of the isotropy algebra of $\Y,$ we have the following obvious
\begin{lem}
The kernel of
$$ \ft \ni E \mapsto E(\Y) \in \otimes^p \RR^m $$
is trivial.
\end{lem}
\begin{cor}\label{cor-fi}
There exists a $G$-equivariant
map $\varphi : \otimes^p\RR^m \to \ft$ such that, independently of the choice of an adapted frame $e,$
$$ A^\ft(X) = -(e_x \circ \varphi \circ e_x^{-1}) \nabla^{LC}_X \Y_M\quad\textrm{for each}\ X\in T_x M. $$
\end{cor}
\begin{proof}
Set $\varphi$ to be the left inverse of $E\mapsto E(\Y),$ which exists due to the Lemma.
Then, in an adapted frame $e,$ we have
\begin{eqnarray*}
( e_x \circ \varphi \circ e_x^{-1} ) \nabla^{LC}_X \Y_M &=& 
(e_x \circ \varphi) [(e_x^{-1} \circ A^\ft(X) \circ e_x) (\Y)] \\ &=&
e_x ( e_x^{-1} \circ A^\ft(X) \circ e_x ) \\ &=& A^\ft(X).
\end{eqnarray*}
\end{proof}

\begin{note}\label{rem-it}
We can derive a more direct result for a symmetric tensor $$\Y\in\Sym^p(\RR^m),\quad p\ge 3,$$
assuming $\RR^m$ to be $G$-irreducible.  
Since $\Y$ is $G-$invariant, irreducibility implies
$$\Y_{i m_2 \dots m_p} \Y_{j m_2 \dots m_p} = \lambda_0\ g_{ij}$$
for some $\lambda_0\in\RR,$
where $g_{ij} X^i Y^j =\langle X,Y\rangle.$ 
Similarily, the map
$$ \cD_\Y : \Lam^2 \RR^m \to \Lam^2 \RR^m $$
$$ \cD_\Y(E)_{ij} = E_{kl} \Y_{kjm_3\dots m_p} \Y_{lim_3\dots m_p} $$
restricted to $\ft$ is given by 
$$ \cD_\Y|_\ft = \lambda_1 \pr_{\ft_1} + \dots + \lambda_r \pr_{\ft_r} $$
for some $\lambda_1,\dots,\lambda_r \in \RR,$ 
where $\pr_{\ft_1},\dots,\pr_{\ft_r}$ are projections
onto $G-$irreducible subspaces of $\ft:$
$$\ft = \ft_1 \oplus \dots \oplus \ft_r. $$
We finally introduce the map
$$ \Y^\ft : \ft \to \ft $$
$$ \Y^\ft(E)_{ij} = E(\Y)_{im_2\dots m_p} \Y_{jm_2\dots m_p} $$
and, performing the contractions, easily find that $\Y^\ft = \lambda_0\ \id + (p-1) \cD_\Y.$

If we thus calculate $\lambda_0,\lambda_1,\dots,\lambda_r$ and check that $\lambda_0 + (p-1) \lambda_k
\neq 0$ for $k=1,\dots,r,$ we can express the intrinsic torsion as in Corollary \ref{cor-fi}
with 
\begin{equation}
\varphi(t)_{ij} = \sum_{k=1}^r \frac{1}{\lambda_0 + (p-1)\lambda_k}\ (\pr_{\ft_k})^{kl}_{ij}
\ t_{km_2\dots m_p} \Y_{lm_2\dots m_p} \label{eqfi}\end{equation}
for $t\in\Sym^p\RR^m.$
As one can find in Section \ref{sec-dec},
the space $\ft$ in case of $G(\KK,\KK')-$structures decomposes into at most two subspaces. Thus the projections
$\pr_{\ft_k}$ with $k=1,2$ can be expressed
as combinations of $\cD_\Y$ and the
identity map (provided $\lambda_1\neq\lambda_2$). This will provide us with explicit expressions for $A^\ft$ in terms
of $\Y_M$ and $\nabla^{LC}\Y_M.$

\end{note}

\section{$\fg(\KK,\KK')$-geometries and their torsion}
\label{sec-kkgeo}

We will now finally define the geometries we are to investigate. Our choice is
to consider Riemannian manifolds equipped with the sole invariant tensor, disregarding global existence
of a $G(\KK,\KK')$-structure and its local choice. When referring to a local $G(\KK,\KK')$-structure
associated with the $\fg(\KK,\KK')$-geometry, we mean any of the possible ones (i.e. any connected
component of the full bundle of frames defined on some region of the manifold and preserving the special form of the tensor).\footnote{
In particular, the tensor gives global $G(\KK,\KK')$-structures on a simply connected manifold. This
can be seen noting, that every loop lifted to the \emph{full} bundle of adapted frames (i.e. the one whose structure
group is the full isotropy group of the tensor) must necessarily start and end in the same connected
component of a fibre.
}

Recall the conventions summarized at the beginning of this chapter. In particular, $\KK$
is any of $\RR,\CC,\HH,\OO,$ and $\KK'$ is set to $\CC,\HH,\OO$ in, respectively, first, second
and third family. We have defined $\kappa = \dim\KK.$ To avoid index clutter, we
abandon the convention of writing $\Y_M$ in favour of $\underline\Y,$ provided $M$
is clear from context.

\subsection{First family}
The following definition essentially
coincides with that of the geometries studied by Nurowski \cite{nurowski-2006}:
\begin{defn} A $\fg(\KK,\CC)$-geometry is a $(3\kappa+2)-$dimensional Riemannian manifold $M$ equipped with a tensor
$$ \UpsilonM \in \Omega^0(M, \Sym^3 TM ) $$
such that at each $x\in M$ there exists an orthonormal 
frame $e_x : \cV_1(\KK) \to T_x M$ such that $
\UpsilonM = e_x(\Upsilon).$
\end{defn}

Let thus $M$ be a $\fg(\KK,\CC)$-geometry. Since in the first family the groups $G(\KK,\CC)$
are the precise isotropy groups of $\Upsilon,$ it follows that the global principal bundle
of all frames mapping $\Upsilon$ to $\UpsilonM$ has $G(\KK,\CC)$ as its structure group:
\begin{pro}
A $\fg(\KK,\CC)$-geometry possesses a natural global $G(\KK,\CC)$-structure.
\end{pro}

In what follows, we shall use $i,j,\dots$ to index $TM \simeq T^*M,$ when referring to tensors and 
tensor-valued forms,
so that we have e.g.
\begin{eqnarray*}
{\UpsilonM}_{ijk} &\in& \Omega^0(M,(S^3 TM)_{ijk}) \\
(A^\ft)_{ij} &\in& \Omega^1(M,(\Lam^2 TM)_{ij}) \\
(\nabla^{LC}\Upsilon_M)_{ijk} &\in& \Omega^1(M,(\Sym^3 TM)_{ijk}).
\end{eqnarray*}

\begin{pro}\label{pro-tor1} The intrinsic torsion of the $G(\KK,\CC)$-structure associated with a $\fg(\KK,\CC)$-geometry $(M,\UpsilonM)$ is given by
$$ (A^\ft)_{ij} = \frac{3\kappa+4}{3\kappa+10} 
(\nabla^{LC} \UpsilonM)_{imn}
\UpsilonM_{jmn}.
$$
\end{pro}
\begin{proof} The proof is by application of the procedure
described in Remark \ref{rem-it}.
The decompositions of $\Lam^2 \cV_1(\KK)$ given in Section \ref{sec-dec}
show that $\ft = \fg(\KK,\CC)^\bot$ is always irreducible; one has
$$ \cD_\Upsilon|_\ft = \frac{3}{3\kappa+4}, $$
according to the summary presented in the first section of this chapter.
One thus uses the formula (\ref{eqfi}) with $\lambda_0 = 1,$ $\lambda_1 = \frac{3}{3\kappa+4}$
and $p=3,$
while
the projection is simply identity. Having the map $\varphi : \Sym^3\cV_1(\KK) \to \ft$ expressed
using $\Upsilon,$ the Proposition follows from  Corollary \ref{cor-fi}.
\end{proof}

\subsection{Second family}
\begin{defn} A $\fg(\KK,\HH)$-geometry is a $(6\kappa+6)-$dimensional Riemannian manifold $M$ equipped with a tensor
$$ \XiM \in \Omega^0(M, \Sym^6 TM ) $$
such that at each $x\in M$ there exists an orthonormal frame $e_x : \cV_2(\KK) \to T_x M$ 
such that $\XiM = e_x(\Xi),$ where $\cV_2(\KK)$ is considered as a real vector space.
\end{defn}

Let thus $M$ be a  $\fg(\KK,\HH)$-geometry. As it has been demonstrated in Example \ref{eg-noG},
there may be in general no globally defined $G(\KK,\HH)$-structure associated to $\XiM.$ However,
one has a global split
$$ \Lam^2 TM = \fg_M \oplus \ft_M $$
with $\fg_M(x) \simeq \fg(\KK,\HH)$ as a Lie subalgebra of $\Lam^2 T_x M \simeq \fso(\cV_2(\KK)).$
Recall that the algebra $\fg(\KK,\HH)$ has a one-dimensional centre spanned by the hermitian complex
structure of $\cV_2(\KK).$ The same is true of each $\fg_M(x),$ so that defining a bundle
$$ \fu(1)_M \subset \fg_M,\quad \fu(1)_M(x) = \mathcal{Z}(\fg_M(x)), $$
with $\mathcal{Z}$ defining the centre,
we have the following
\begin{pro} \label{pro-ah2}
A $\fg(\KK,\HH)$-geometry $(M,\XiM)$ is naturally equipped with a one-dimensional subbundle $
\fu(1)_M \subset \Lam^2 TM$
spanned locally by an almost hermitian structure.
\end{pro}
Each fibre $\fu(1)_M(x)$ contains exactly two complex structures on $T_x M$ (singled
out by normalization), and there is a priori no way to distinguish between these. A
particular choice of one complex structure at a single point of $M$
can be extended by continuity to a neighbourhood, 
but may fail to yield a globally defined almost hermitian structure on $M$ (cf. Example \ref{eg-noG}).


Before we proceed, an algebraic consideration is necessary:
\begin{lem} \label{lem-dxi} Reintroducing $N = 3\kappa+3,$ we have: \begin{enumerate}
\item
$ \Xi_{adefkl} \Xi^{bdefkl} = \frac{N}{40} \delta^a_b. $
\item
Let us introduce a map
$$ \cD_\Xi : \Lam^2 \cV_2(\KK) \to \Lam^2 \cV_2(\KK) $$
$$ \cD_\Xi(E)_{ab} = E_{cd} \Xi_a{}^{defkl} \Xi^c{}_{befkl}. $$
Then, with respect to the decomposition
$$ \Lam^2 \cV_2(\KK) = \fu(1) \oplus \cg_2(\KK) \oplus (\ft\cap\Lam^{1,1}) \oplus \Re(\Lam^{2,0}\oplus\Lam^{0,2}) $$
the map $ \cD_\Xi$ is given by
\begin{eqnarray*}
200\ \cD_\Xi|_{\fu(1)} &=& {2N+3} \\
200\ \cD_\Xi|_{\cg_2(\KK)} &=& -{N} \\
200\ \cD_\Xi|_{\ft\cap\Lam^{1,1}} &=& \frac{2N}{\kappa+2} \\
200\ \cD_\Xi|_{\Re(\Lam^{2,0}\oplus\Lam^{0,2})} &=& 3. 
\end{eqnarray*}
\end{enumerate}
\end{lem}
The proof is given in Section \ref{sec-xprf}.
We are now ready to express the intrinsic torsion in terms of $\XiM$ and its derivative. In the
following we use $a,b,\dots$ to index $TM\simeq T^*M:$
\begin{pro}\label{pro-tor2}
The intrinsic torsion of a local $G(\KK,\HH)$-structure associated with a $\fg(\KK,\HH)$-geometry
$(M,\XiM)$ is given by
\begin{eqnarray*}
(A^\ft)_{ab} &=& \frac{40}{3\kappa^2+15\kappa+12} \left[ \frac{\kappa^2+6\kappa+6}{\kappa+2}\ \delta_a^c \delta_b^d 
- \frac{200}{3}\ \XiM_a{}^{defkl} \XiM_b{}^{cefkl} \right] \\
&\times& (\nabla^{LC} \XiM)_{cmnrst} \XiM_d{}^{mnrst}.
\end{eqnarray*}
\end{pro}
\begin{proof} The proof is by application of the procedure
described in Remark \ref{rem-it}.
The decompositions of $\Lam^2 \cV_2(\KK)$ given in Section \ref{sec-dec}
show that $\ft = \fg(\KK,\CC)^\bot$ decomposes into two irreducible subspaces:
$$ \ft = (\ft\cap\Lam^{1,1}) \oplus \Re(\Lam^{2,0}\oplus\Lam^{0,2}). $$
The map $\cD_\Upsilon|_\ft$ and the contraction $\Xi_{adefkl} \Xi^{bdefkl}$ is
found in Lemma \ref{lem-dxi}, so that one uses the formula
(\ref{eqfi}) with $\lambda_0 = \frac{N}{40},$ 
$\lambda_1 = \frac{1}{100}\frac{N}{\kappa+2},$ $\lambda_2 = \frac{3}{200}$
and $p=6,$
while
the projections are (having checked that $\lambda_1\neq\lambda_2$ for $\kappa=1,2,4,8$)
\begin{eqnarray}
\pr_{\ft\cap\Lam^{1,1}} &=& -\frac{\kappa+2}{\kappa}\ \id + \frac{\kappa+2}{3\kappa}\ 200\ \cD_\Xi 
\nonumber\\
\pr_{\Re(\Lam^{2,0}\oplus\Lam^{0,2})} &=& \frac{2\kappa+2}{\kappa}\ \id -  \frac{\kappa+2}{3\kappa} \ 200\ \cD_\Xi. \label{eq-prau}
\end{eqnarray}
Having the map $\varphi : \Sym^6\cV_2(\KK) \to \ft$ expressed
using $\Xi,$ the Proposition follows from  Corollary \ref{cor-fi}.
\end{proof}

Recall that, by Proposition \ref{pro-ah2}, normalized local sections of $\fu(1)_M$ give
local almost hermitian structures on $M$. These are integrable iff $\fu(1)_M$ is parallel,
i.e. iff
$$ f \in \Omega^0(M, \fu(1)_M) \implies \nabla^{LC} f \in \Omega^1(M,\fu(1)_M). $$
Since $\nabla^\fg$ has this property, it follows that in this case the intrinsic torsion
must satisfy $$[A^\ft(X), \fu(1)_M(x)] \subset \fu(1)_M(x)$$ for each $X\in T_x M.$ Expressed
in an adapted frame $e_x : \cV_2(\KK) \to T_x M,$ the condition reads $$ A^\ft(X) \in e_x(\Re\Lam^{1,1})$$ 
i.e. the antiunitary component, obtained by projecting onto $\Re(\Lam^{2,0}\oplus\Lam^{0,2}),$
must vanish.
Applying the projection (\ref{eq-prau}) to the intrinsic torsion expressed in Proposition \ref{pro-tor2} the following formula is readily proved:
\begin{pro}\label{pro-kaehler}
A $\fg(\KK,\HH)$-geometry $(M,\XiM)$ possesses a natural local K\"ahler structure, given by a unit
local section of $\fu(1)_M,$ iff
$$
\left[
\XiM_a{}^{defkl} \XiM_b{}^{cefkl} -\frac{3}{100}\frac{\kappa+1}{\kappa+2}\delta_a^c \delta_b^d 
\right]
\ (\nabla^{LC} \XiM)_{cmnrst} \XiM_d{}^{mnrst} = 0. $$
\end{pro}

\subsection{Third family}
\begin{defn} A $\fg(\KK,\OO)$-geometry is a $(12\kappa+16)-$dimensional Riemannian manifold $M$ equipped with a tensor
$$ \mhoM \in \Omega^0(M, \Sym^8 TM ) $$
such that at each $x\in M$ there exists an orthonormal frame $e_x : \cV_3(\KK) \to T_x M$ 
such that $\mhoM = e_x(\mho).$
\end{defn}
Let thus $M$ be a $\fg(\KK,\OO)$-geometry. There exists a global split
$$ \Lam^2 TM = \fg_M \oplus \ft_M $$
such that $\fg_M(x) \simeq \fg(\KK,\OO)$ at each $x\in M.$ Recall that
$\fg(\KK,\OO)$ decomposes under its adjoint representation into two a direct sum
$$ \fg(\KK,\OO) = \cg_3(\KK) \oplus \fsp(1) $$
with $\fsp(1)$ generated by a quaternionic structure on $\cV_3(\KK).$ The same is true of
each $\fg_M(x)$, so that defining a bundle
$$ \fsp(1)_M \subset \fg_M,\quad [\fg_M,\fsp(1)_M]=\fsp(1)_M,\quad \dim\fsp(1)_M(x) = 3 $$
we have the following
\begin{pro} A $\fg(\KK,\OO)$-geometry $(M,\mhoM)$ is naturally equipped with a three-dimensional subbundle
$\fsp(1)_M \subset \Lam^2 TM$ spanned locally by three almost hermitian structures subject to
the algebra of imaginary quaternions.
\end{pro}
Note that the full orthogonal isotropy group of $\mho$ necessarily preserves the structure constants
of $\fsp(1) \subset \fg(\KK,\OO),$ and thus an orientation on this space, so that it acts on it
as $\sSO(3) \simeq \Aut\HH.$ This in turn implies that the images of $I,J,K\in\Lam^2\cV_3(\KK)$ under different adapted frames are related by automorphisms of the unique quaternionic structure. Hence the following\footnote{
There are many equivalent ways of introducing (almost) quaternion-hermitian structures on a manifold -- e.g. by means of a 4-form, or a 2-sphere bundle in $\End TM$. For reference, see \cite{joyce-x, swann-1991, ivanov-2004}.
}
\begin{cor} \label{cor-qh}A $\fg(\KK,\OO)$-geometry $(M,\mhoM)$ is naturally equipped with a unique global almost quaternion-hermitian structure defined by
$$ \underline\Omega\in\Omega^4(M),\quad \underline\Omega(x) = e_x ( I\wedge I + J\wedge J + K\wedge K) $$
in an arbitrary adapted frame $e_x : \cV_3(\KK)\to T_x M.$
\end{cor}

Once again some algebra is needed before we can express the intrinsic torsion:
\begin{lem}\label{lem-dmho}
Reintroducing $N=3\kappa+4,$ we have:\begin{enumerate}
\item $\mho_{am_2\dots m_8} \mho^{bm_2\dots m_8} = \frac{64}{35}\frac{N+1}{2}[25(N-1)+12\chi^2]\ \delta^a_b$
\item Let us introduce a map
$$ \cD_\mho : \Lam^2 \cV_3(\KK) \to \Lam^2 \cV_3(\KK) $$
$$ \cD_\mho(E)_{ab} = E_{cd} \mho_a{}^{dm_3\dots m_8} \mho^c{}_{bm_3\dots m_8}. $$
Then, with respect to the decomposition
$$ \Lam^2 \cV_3(\KK) = \fsp(1) \oplus \cg_3(\KK) \oplus [\ft\cap\fsp(\cV_3(\KK,\omega))] \oplus 
\bot, $$
the map $\cD_\mho$ is given by
\begin{eqnarray*}
\cD_\mho|_{\fsp(1)} &=& 
\frac{32}{245} \left(30N^2 - 9N - 21 - 636\chi^2\right) \\
\cD_\mho|_{\cg_3(\KK)} &=& 
\frac{32}{245} \left(51N^2 + 13N -83 -18\chi^2 - \sqrt{\kappa+3}\right) \\
\cD_\mho|_{\ft\cap\fsp(\cV_3(\KK),\omega)} &=&
\frac{32}{245} \left(51N^2+13N-83-18\chi^2+\frac{1}{\sqrt{\kappa+3}}\right) \\
\cD_\mho|_{\bot} &=&
\frac{32}{245} (15N^2+53N+47-386\chi^2).
\end{eqnarray*}
\end{enumerate}\end{lem}
The proof can be found in Section \ref{sec-xprf}.
We are now able to read the intrinsic torsion from $\mhoM$ and its derivative:
\begin{pro}\label{pro-tor3}
The intrinsic torsion of a local $G(\KK,\OO)$-structure associated with a $\fg(\KK,\OO)$-geometry
$(M,\mhoM)$ is given by
\begin{eqnarray*}
(A^\ft)_{ab} &=& \gamma^{-1} [ \alpha\ \mhoM_a{}^{dm_3\dots m_8} \mhoM^c{}_{bm_3\dots m_8} + \beta
\ \delta^c_a\delta^d_b
] \\
&\times& (\nabla^{LC} \mhoM)_{cn_2\dots n_8} \mhoM_d{}^{kn_2\dots n_8}
\end{eqnarray*}
where
\begin{eqnarray*}
\alpha &=& \frac{1}{\lambda_0 + 7 \lambda_1} - \frac{1}{\lambda_0 + 7\lambda_2} \\
\beta &=& -\frac{\lambda_2}{\lambda_0 + 7 \lambda_1} + \frac{\lambda_1}{\lambda_0 + 7\lambda_2} \\
\gamma &=& \frac{32}{245}\left(36N^2+40N-130+368\chi^2+\frac{1}{\sqrt{\kappa+3}}\right) \\
\lambda_0 &=& \frac{64}{35}\frac{N+1}{2}[25(N-1)+12\chi^2] \\
\lambda_1 &=& \frac{32}{245} \left(51N^2+13N-83-18\chi^2+\frac{1}{\sqrt{\kappa+3}}\right) \\
\lambda_2 &=& \frac{32}{245} (15N^2+53N+47-386\chi^2).
\end{eqnarray*}
\end{pro}
Before we give the proof, we must sadly admit that even after $N$, $\chi$ and $\kappa$ have
been substituted by numbers for given $\KK,$ the constants remain unreasonably complicated.
\begin{proof} The proof is by application of the procedure
described in Remark \ref{rem-it}.
The decompositions of $\Lam^2 \cV_3(\KK)$ given in Section \ref{sec-dec}
show that $\ft = \fg(\KK,\OO)^\bot$ decomposes into two irreducible subspaces:
$$ \ft = [\ft\cap\fsp(\cV_3,\omega)] \oplus [\fsp(\cV_3,\omega)+\fsp(1)]^\bot. $$
The map $\cD_\mho|_\ft$ and the contraction $\mho_{am_2\dots m_8} \mho^{bm_2\dotsm_8}$ is
found in Lemma \ref{lem-dmho}, so that one uses the formula
(\ref{eqfi}) with 
\begin{eqnarray*}
\lambda_0 &=& \frac{64}{35}\frac{N+1}{2}[25(N-1)+12\chi^2] \\
\lambda_1 &=& \frac{32}{245} \left(51N^2+13N-83-18\chi^2+\frac{1}{\sqrt{\kappa+3}}\right) \\
\lambda_2 &=& \frac{32}{245} (15N^2+53N+47-386\chi^2).
\end{eqnarray*}
and $p=8,$
while
the projections are (having checked that $\lambda_1\neq\lambda_2$ for $\kappa=1,2,4,8$)
\begin{eqnarray}
\pr_{\ft\cap\fsp(\cV_3(\KK),\omega)} &=& \frac{1}{\lambda_1-\lambda_2} (\cD_\mho - \lambda_2 \id)
\nonumber\\ \label{eq-pras}
\pr_{[\fsp(\cV_3(\KK),\omega)+\fsp(1)]^\bot} &=& \frac{1}{\lambda_2-\lambda_1} (\cD_\mho - \lambda_1 \id).
\end{eqnarray}
Having the map $\varphi : \Sym^8\cV_3(\KK) \to \ft$ expressed
using $\mho,$ the Proposition follows from  Corollary \ref{cor-fi}.
\end{proof}

The condition for integrability of the almost-quaternion-hermitian structure 
described in Corollary \ref{cor-qh} is that the Levi-Civita
parallel transport preserve a bundle of unit 2-spheres in
$\fsp(1)_M$:
$$ \mathcal{S} \subset \fsp(1)_M,\quad \mathcal{S}_x = \{ E \in \fsp(1)_M(x)\ | E^2 = -\id \}, $$
which under this condition becomes the natural twistor bundle of the quaternion-K\"ahler structure
(cf. \cite{joyce-x}).
It is equivalent to
$$ f \in \Omega^0(M, \fsp(1)_M) \implies \nabla^{LC} f \in \Omega^1(M,\fsp(1)_M). $$
Since $\nabla^\fg$ has this property, it follows that in this case the intrinsic torsion
must satisfy $$[A^\ft(X), \fsp(1)_M(x)] \subset \fsp(1)_M(x)$$ for each $X\in T_x M.$ Expressed
in an adapted frame $e_x : \cV_3(\KK) \to T_x M,$ the condition reads $$ A^\ft(X) \in e_x(\fsp(\cV_3(\KK),\omega))$$ 
i.e. the component in $\fsp(\cV_3(\KK),\omega)^\bot$
must vanish.
Applying the projection (\ref{eq-pras}) to the intrinsic torsion expressed in Proposition \ref{pro-tor3} the following formula is readily proved:
\begin{pro}\label{pro-qkaehler}
A $\fg(\KK,\OO)$-geometry $(M,\mhoM)$ possesses a natural quaternion-K\"ahler structure, given by
the 2-sphere bundle 
$$ \mathcal{S} \subset \fsp(1)_M,\quad \mathcal{S}_x = \{ E \in \fsp(1)_M(x)\ | E^2 = -\id \}, $$
iff
$$
[ \mhoM_a{}^{dm_3\dots m_8} \mhoM^c{}_{bm_3\dots m_8} - \lambda_1
\ \delta^c_a\delta^d_b
] (\nabla^{LC} \mhoM)_{cn_2\dots n_8} \mhoM_d{}^{n_2\dots n_8} = 0 $$
where
$$ \lambda_1 = \frac{32}{245} \left(51N^2+13N-83-18\chi^2+\frac{1}{\sqrt{\kappa+3}}\right). $$
\end{pro}

\section{$G$-structures with characteristic torsion}

In the previous section we have at last defined the geometries modelled after the Magic Square symmetric spaces, and expressed their intrinsic torsion in terms of geometric data, i.e. the defining tensorial invariant and its Levi-Civita derivative. It is of course natural to investigate at first the integrable case, that is, manifolds with parallel $\UpsilonM,$ $\XiM$ or $\mhoM.$ However, we immediately have
the following
\begin{pro}[Corollary of Berger's theorem]
Let $(M,g,\Y_M)$ be a $\fg(\KK,\KK')$-geometry, whose underlying local $G(\KK,\KK')$-structures
are integrable. Then $(M,g)$ is either a locally symmetric space or a product of Riemannian manifolds
of lower dimension,
with the product metric tensor.
\end{pro}
\begin{proof}[Proof (sketch)]
Integrability implies that
the connected holonomy group ${\rm Hol}_0(g),$ considered via some adapted frame as a subgroup 
of $\sGL(\cV_n(\KK)),$ must be contained in $G(\KK,\KK').$ Simple dimension count shows that none
of the irreducible Riemannian holonomy groups listed in Berger's theorem meet this requirement. Thus,
$(M,g)$ cannot be irreducible and not locally symmetric.
\end{proof}

A larger variety of geometries is available once one relaxes the integrability condition to some extent.
While integrability is equivalent to existence of a compatible connection with trivial torsion, we can
consider a milder condition of vanishing all but one of the torsion's irreducible components. The torsion tensor of a connection on a Riemannian manifold $(M,g)$ is a section of the bundle
$\Lam^2 TM \otimes TM,$ which decomposes under the action of the orthogonal group into three (for $\dim M>3$) irreducible components:
$$ \Lam^2 TM \otimes TM \simeq \Lam^3 TM \oplus TM \oplus \mathcal{T}. $$ 

Particularily interesting classes of connections are those with completely skew ($\Lam^3 TM$)
and vectorial ($TM$) torsion, where the associated projections act on $T_{ijk} \in (\Lam^2 TM)_{ij} \otimes (TM)_k$ as:
$$ (\pr_{skew} T)_{ijk} = T_{[ijk]},\quad (\pr_{vec} T)_{ijk} = \frac{1}{\dim M} T_{[i}{}^{mm} g_{j]k}. $$
In what follows, we shall restrict our attention to the skew case (information about the vectorial one can be found e.g. in \cite{agricola-2004}), mostly following the exposition given in \cite{agricola-2006}.
A geometric characterisation of this class is given by the following
\begin{pro}[cf. \cite{agricola-2006}] \label{pro-geo}
A metric connection $\nabla$ in the tangent bundle of a Riemannian manifold $(M,g)$ has completely
skew torsion iff its unparametrised geodesics conincide with those of the Levi-Civita connection
on $M.$
\end{pro}
Let us first introduce a simple
\begin{lem}\label{lem-skewa}
Let $\nabla$ be a metric connection in the tangent bundle of a Riemannian manifold $(M,g)$
and $A \in \Omega^1(M,\Lam^2 TM)$ its difference tensor with the Levi-Civita connection:
$$ A(X)(Y) = \nabla_X Y - \nabla^{LC}_X Y. $$
Then $\nabla$ has completely skew torsion $T$ iff $A$ is completely skew, i.e.
$$ A(X)(Y) = - A(Y)(X). $$
Moreover, in such case $A = \frac{1}{2} T$ as a section of $\Lam^3 TM.$
\end{lem}
\begin{proof}
The torsion of $\nabla$ is
\begin{eqnarray*}
T(X,Y) &=& \nabla_X Y - \nabla_Y X - [X,Y] \\ &=& \nabla^{LC}_X Y - \nabla^{LC}_Y - [X,Y]
+ A(X)(Y) - A(Y)(X) \\ &=& A(X)(Y) - A(Y)(X)
\end{eqnarray*}
and
\begin{eqnarray*}
g(T(X,Y),Z) + g(T(X,Z),Y) &=& g(A(X)(Y),Z) - g(A(Y)(X),Z) \\ &+& g(A(X)(Z),Y) - g(A(Z)(X),Y)
\\ &=& g(A(Y)(Z),X) + g(A(Z)(Y),X).
\end{eqnarray*}
Vanishing of l.h.s. is equivalent to complete antisymmetry of $T,$ while vanishing of r.h.s.
is equivalent to complete antisymmetry of $A.$
\end{proof}
\begin{proof}[Proof of Proposition \ref{pro-geo}]
Let $\gamma : ]-\epsilon,\epsilon[ \to M $ be a curve in $M$ and $X$ a vector field defined on some neighbourhood of $\gamma$ such that $X \circ \gamma$ is tangent to $\gamma.$
Then $\gamma$ is (unparametrised) geodesic of $\nabla$ iff $$(\nabla_X X) \circ \gamma
= (\nabla^{LC}_X X + A(X)(X)) \circ \gamma
= f X \circ \gamma $$
for some function $f.$ The latter is equivalent to $\gamma$ being a Levi-Civita geodesic iff
$A(X)(X)$ is a multiple of $X$. However, since $A(X)$ is skew, $A(X,X)$ must be simply zero. Demanding it for every curve $\gamma$ is equivalent to $A(X)(X) = 0$ for each $X\in TM,$
i.e. $A$ completely skew. This in turn, via the previous Lemma, is equivalent to complete antisymmetry of the torsion of $\nabla.$
\end{proof}

\subsection{$G$-structures with skew torsion}
Let us now consider a general $G$-structure $Q$ on an $m$-dimensional Riemannian manifold $(M,g)$
as described in subsection \ref{ss-gstr}. The basic result is the following:
\begin{pro}[cf. \cite{agricola-2006}] \label{pro-exc}
A $G$-structure $Q$ on $M$ admits a compatible connection with completely skew torsion iff
there exists a function
$H\in\Omega^0(Q)\otimes\Lam^3 \RR^m$ (of the natural type) such that for each $X\in TQ$
$$ \alpha^\ft(X) = \pr_\ft\ H(\theta(X))  $$
where $\alpha^\ft$ is the intrinsic torsion of $Q,$ and $\theta$ the soldering form.
\end{pro}
We first give a simple
\begin{lem} \label{lem-skeww}
A connection $\omega$ compatible with a $G$-structure $Q$ has skew torsion iff
$$ [\omega(X)-\omega^{LC}(X)](\theta(Y)) = - [\omega(Y)-\omega^{LC}(Y)](\theta(X)). $$
\end{lem}
\begin{proof}
The torsion of $\omega$ is
\begin{eqnarray*} \Theta &=& d\theta + \omega\wedge\theta 
\\&=& (d\theta + \omega^{LC}\wedge\theta) + (\omega-\omega^{LC})\wedge\theta
\\&=& \alpha\wedge \theta,
\end{eqnarray*}
where we introduced $\alpha=\omega-\omega^{LC},$ so that $\alpha|_Q \in \Omega^1(Q)\otimes\fg$
is horizontal of type $\Ad.$
We now have
\begin{eqnarray*}
\langle \Theta(X,Y),\theta(Z) \rangle + \langle \Theta(X,Z),\theta(Y) \rangle
&=& \alpha(X)(\theta(Y),\theta(Z) - \alpha(Y)(\theta(X),\theta(Z))
\\&+& \alpha(X)(\theta(Z),\theta(Y) - \alpha(Z)(\theta(X),\theta(Y))
\\
&=& \alpha(Y)(\theta(Z),\theta(X)) + \alpha(Z)(\theta(Y),\theta(X)),
\end{eqnarray*}
where vanishing of l.h.s is equivalent to complete antisymmetry of the torsion.
\end{proof}
\begin{proof}[Proof of Proposition \ref{pro-exc}]
Assume that indeed $\alpha^\ft$ satisfies this condition for some $H.$ Then
$$ \omega^{LC}|_Q = \omega^\fg - \alpha^\ft = \omega^\fg|_Q - H(\theta(\cdot)) + \beta $$
where $\beta\in\Omega_{hor}^1(Q)\otimes\fg.$ One easily checks that $ \omega^s = \omega^\fg+\beta $
defines a connection on $Q.$  Now, we have $ \omega^s -\omega^{LC}|_Q = H(\theta(\cdot)), $
and by Lemma \ref{lem-skeww} the torsion of $\omega^s$ is skew.

Conversely, the same lemma implies that the difference of a connection $\omega^s$ with skew torsion and
the Levi-Civita connection is a completely skew tensor, and the function $H$ is simply given by
$$ H_q(\theta(X),\theta(Y),\theta(Z)) = \langle (\omega(X)-\omega^{LC}(X))\theta(Y) ,\theta(Z) \rangle $$for $X,Y,Z\in T_qQ.$ We have
$$ \alpha^\ft(X) = H(\theta(X)) + \omega^g(X) - \omega^s(X), $$
and projecting on $\ft$ proves the Proposition.
\end{proof}

Having stated a condition for existence of a compatible connection with skew torsion, it is natural
to ask how many such connections can be found. A particularily interesting case occurs when the connection is unique. There is a well known, purely algebraic condition:
\begin{pro}[cf. \cite{agricola-2006}]
\label{pro-inter}
A compatible connection with skew torsion on a $G$-structure $Q$ is unique, provided it exists, iff 
\begin{equation} (\RR^m \otimes \fg) \cap \Lam^3 \RR^m = 0. \label{eq-inter} \end{equation}
\end{pro}
\begin{proof}
Assume $\omega^s$ and $\varpi^s$ are two such connections. Then, by Lemma \ref{lem-skeww},
the function $C \in \Omega^0(Q)\otimes(\RR^m \otimes \fg) $ defined by
\begin{eqnarray*} C_q(\theta(X),\theta(Y),\theta(Z)) 
&=& \langle(\omega^s(X) - \varpi^s(X))\theta(Y),\theta(Z)\rangle 
\\ &=& 
\langle(\omega^s(X) - \omega^{LC}(X))\theta(Y),\theta(Z)\rangle \\&-& 
\langle(\varpi^s(X) - \omega^{LC}(X))\theta(Y),\theta(Z)\rangle)
\end{eqnarray*}
for $X,Y,Z\in T_qQ$
is completely skew. Thus, if the intersection of $\RR^m\otimes\fg$ and $\Lam^3 \RR^m$ is trivial,
one has $\omega^s = \varpi^s.$

Conversely, let $\omega^s$ be the unique compatible connection with skew torsion and $C\in\Omega^0(Q)\otimes [(\RR^m\otimes\fg) \cap \Lam^3 \RR^m]$ a function of the natural type. Then
$$ \varpi^s(X) = \omega^s(X) - C(\theta(X)) $$
for $X\in TQ$ defines a connection on $Q$ and by Lemma \ref{lem-skeww} its torsion is skew.
Now, if the intersection of $\RR^m\otimes\fg$ and $\Lam^3\RR^m$ was nontrivial,
$\varpi^s(X) \neq \omega^s(X)$ and there would exist different compatible connections with skew torsion.
As it contradicts the uniqueness of $\omega^s,$ the intersection must be trivial.
\end{proof}

If such a connection is indeed unique, it is called \emph{the characteristic connection} of the
$G$-structure, and its torsion tensor -- \emph{the characteristric torsion.} It is not very difficult
to check the intersections (\ref{eq-inter}) for classical irreducible holonomy groups, and the result
is that a skew-torsion connection is unique in all those cases \cite{agricola-2006}. Moreover, there
is a recent important result of Nagy, which solves the problem of computing the l.h.s. of (\ref{eq-inter}) completely:
\begin{pro}[Nagy \cite{nagy-2007}]\label{pro-nagy}
Let $\fg\subset\fso(m)$ be proper and act irreducibly on $\RR^m$. Then the intersection (\ref{eq-inter})
is trivial, unless $\fg$ compact simple and $\RR^m\simeq \fg$ as a $\fg$-module. In the latter case, it is one dimensional and spanned by the structure constants of $\fg$.
\end{pro}

Considering the case with nontrivial intersection (i.e. no characteristic connection), we have:
\begin{cor}\label{cor-nagy}
Let $G$ be simple, with $\fg$ proper in $\fso(m)$ and $\RR^m\simeq \fg$ as a $G$-module. Then a $G$-structure on $M$ admits either no compatible connections with skew torsion, or precisely a one-parameter family thereof. In the latter case, the torsion tensors of connections in this family differ by a section of a one-dimensional invariant subbundle $\mathcal{T}_0\subset\Lam^3 TM$. Choosing an adapted frame
$e$, an intertwiner $\psi:\fg\to\RR^m$, and defining $f = e \circ \psi$, we have a fibre of $\mathcal{T}_0$:
$$ \mathcal{T}_0(x) = \Span\{C\},\quad C(X,Y,Z) = \langle e_x^{-1} X, \psi[f_x^{-1}Y,f_x^{-1}Z]\rangle $$
for $X,Y,Z\in T_x M$.
\end{cor}
\begin{proof}
Let $Q$ be the $G$-structure. Every connection $\omega\in\Omega^1(Q)\otimes\fg$ on $Q$ with
skew torsion $\Theta\in\Omega^2(Q)\otimes\RR^m$ is uniquely defined by the latter, since (cf. Lemma \ref{lem-skewa})
$$ [\omega(X) - \omega^{LC}(X) ](\theta(Y),\theta(Z)) = \frac{1}{2}\langle \Theta(X,Y),\theta(Z) \rangle.$$
We can introduce a function $H^\omega\in\Omega^0(Q)\otimes\Lam^3\RR^m$ such that
$$ \langle \Theta(X,Y),\theta(Z)\rangle = H^\omega(q)(\theta(X),\theta(Y),\theta(Z))$$
for each $X,Y,Z\in T_q Q.$

On the other hand, two such connections $\omega,\varpi$ differ by a horizontal one-form of type $\Ad$:
$$ \omega-\varpi \in \Omega^1(Q)\otimes\fg, $$
so that 
$$ H^\omega(q)(\theta(X)) - H^\varpi(q)(\theta(X)) \in \fg $$
for each $X\in T_q Q,$ and thus
$$ H^\omega(q) - H^\varpi(q) \in \Lam^3\RR^m \cap (\RR^m \otimes \fg).$$

Assume now that there \emph{exists} a skew-torsion connection $\omega_0$ on $Q$. It thus follows,
that if we identify each skew-torsion connection $\omega$ on $Q$ with the corresponding function
$H^\omega$, the space of all such connections at a point $q\in Q$ is the affine space
$$ H^{\omega_0}(q) + [\Lam^3\RR^m \cap (\RR^m\otimes \fg)]. $$
We can now apply Proposition \ref{pro-nagy} to find that the intersection is spanned by
the map
$$ c: \Lam^3 \RR^m \to \RR,\quad c(X,Y,Z) = \langle X,\psi[\psi^{-1}Y,\psi^{-1}Z]\rangle,$$
for $X,Y,Z\in\RR^m,$ where $\psi:\RR^m\to\fg$ is an intertwiner. Pulling everything back to $M$ via some
adapted frame, i.e. a (local) section $e: M \to Q$, we arrive at the Corollary.
\end{proof}

We shall refer to such a family as \emph{the} one-parameter family of skew-torsion connections.
Their torsion can be considered to be `characteristic modulo $\mathcal{T}_0$'.

\subsection{Invariant tensors and nearly-integrability}
We shall finally focus on structures defined on a Riemannian manifold $(M,g)$ by a \emph{symmetric} tensor $\YM \in \Omega^0(M,\Sym^p TM),$ as in Subsection \ref{ss-its}. The first important fact is a necessary condition for the existence of a compatible connection with skew torsion, first discussed by Nurowski \cite{nurowski-2006}:
\begin{pro}\label{pro-nearly}
Let a local $G-$structure defined by a symmetric tensor
$\YM$ admit a compatible connection with skew torsion. Then the
symmetrized Levi-Civita derivative of $\YM$ vanishes:
\begin{equation} (\nabla^{LC}_X\YM)(X,\dots,X) = 0 \quad\textrm{for each}\ X\in TM. \label{eq-nearly}\end{equation}
\end{pro}
\begin{proof} Let $\nabla$ be such a connection. Clearly, $\nabla\YM = 0.$
Then, recalling Lemma \ref{lem-skewa},
$$ \nabla^{LC}_X\YM = \nabla_X \YM - \frac{1}{2} T(X)(\YM) 
= -\frac{1}{2} T(X)(\YM) $$  where $T$ is
the torsion of $\nabla,$ a section of $\Lam^3 TM\subset TM\otimes\Lam^2 TM.$ Indexing $TM\simeq T^*M$ with $i,j,k,\dots,$
we have
$$ \nabla^{LC}_{(i}\YM_{j_1\dots j_p)} = \frac{p}{2}\ T_{(i}{}^m{}_{j_1} \YM_{j_2\dots j_p) m} = 0$$
due to the antisymmetry of $T.$
\end{proof}
The next natural step is to ask when is (\ref{eq-nearly}) sufficient. After \cite{nurowski-2006},
we give an algebraic condition on the tensor $\Y,$ mapped to $\YM$ in adapted frames:
\begin{lem}\label{lem-nearly}
Let us introduce a map
$$ \Y' : \RR^m \otimes \Lam^2 \RR^m \to \Sym^{p+1}\RR^m $$
$$ \Y'(A)_{i_0 i_1\dots i_p} = A_{(i_0}{}^m{}_{i_1} \Y_{i_2\dots i_p)m}. $$
Then (\ref{eq-nearly}) implies existence of a compatible connection with skew torsion iff
$$ \ker\Y' = \RR^m \otimes \fg + \Lam^3 \RR^m. $$
\end{lem}
\begin{proof} Choose locally an adapted frame $e.$ We have, at each $x\in M,$
\begin{eqnarray*}
\left[e_x^{-1} (\nabla^{LC}\YM)(x)\right]_{(i_0\dots i_p)} &=& \left[ e_x^{-1} \Gamma^{LC}(x) \right]_{(i_0}{}^m{}_{i_1} \Y_{i_2\dots i_p)m}
\\
&=& \Y'(e_x^{-1} \Gamma^{LC}(x))_{i_0\dots i_p}.
\end{eqnarray*}
Thus, $(\nabla^{LC}_X\YM)(X,\dots,X)=0$ for each $X\in TM$ iff $ e_x^{-1} \Gamma^{LC}(x) \in \ker\Y' $
for each $x\in M.$ 

Now, if (\ref{eq-nearly}) implies exsistence of a compatible connection with skew torsion, then $\ker\Y'$ must be contained in $\RR^m\otimes\fg + \Lam^3\RR^m,$ as the Levi-Civita connection
can be expressed as a sum of the compatible connection and half of its torsion. However, Proposition
\ref{pro-nearly} implies that $\RR^m\otimes\fg + \Lam^3 \RR^m \subset \ker\Y'.$ Thus these must be equal.

Conversely, if the kernel of $\Y'$ is $\RR^m\otimes\fg+\Lam^3\RR^m,$ then (\ref{eq-nearly}) implies
that $\Gamma^{LC}$ can be decomposed (not necessarily in a unique way) into a $\fg$-valued local connection form and a skew difference tensor. Lemma (\ref{lem-skewa}) completes the proof.
\end{proof}

The condition (\ref{eq-nearly}) will be referred to as \emph{nearly-integrability} of the
geometric structure defined by $\YM$ (or of the tensor itself). There exists an interesting
geometric interpretation:
\begin{lem}
A $G$-structure defined by $\YM$ is nearly-integrably iff for each parametrised geodesic $\gamma:\RR\supset I \to M$ the value of $\YM$ evaluated on $\dot\gamma$ is constant along the geodesic. \end{lem}
\begin{proof}
Let $\gamma$ be such a geodesic, $\nabla^{LC}_{\dot\gamma}\dot\gamma=0$. We have
$$ \frac{d}{dt} \YM(\dot\gamma(t),\dots,\dot\gamma(t)) = \nabla^{LC}_{\dot\gamma}[\YM(\dot\gamma,\dots,\dot\gamma)](t) = (\nabla^{LC}_{\dot\gamma} \YM)(\dot\gamma,\dots,\dot\gamma)(t).$$
Now, $\YM$ is nearly integrable iff the latter expression vanishes for every parametrised geodesic, at every point. This in turn is equivalent to the evaluation of $\YM$ on $\dot\gamma$ being constant.
\end{proof}

It follows that for a nearly-integrable $G$-structure there is a well-defined notion of geodesics which are null with respect to the tensor. The spaces of such geodesics seem to be interesting in their own right.

\section{$\fg(\KK,\KK')-$geometries with characteristic torsion}

We return now to the $\fg(\KK,\KK')$-geometries defined in Section \ref{sec-kkgeo},
and collect results on
skew-torsion connections compatible with related $G(\KK,\KK')$-structures.

The problem of uniqueness of a skew-torsion connection for the first family has been already
investigated by Nurowski \cite{nurowski-2006}, while uniqueness for the other two families can
be readily estabilished once on knows an analogous result for almost-K\"ahler and almost-quaternion-K\"ahler structures. Currently however, we can present it as a simple corollary of Nagy's general result:
\begin{pro}[Corollary of Proposition \ref{pro-nagy}]
Let $(M,g,\YM)$ be a $\fg(\KK,\KK')$-geometry admitting a skew-torsion compatible connection. The such a
connection is unique, unless $\KK=\CC$ and $\KK'=\CC$. In the latter case, there is a
one-parameter family of such connections, whose torsion differ by a section of a one-dimensional
$G(\CC,\CC)$-invariant section of $\Lam^3 TM$.
\end{pro}
\begin{proof}[Proof (sketch)]
We apply Proposition \ref{pro-nagy} to $\fg(\KK,\KK')$ as subalgebras of $\End V(\KK,\KK')$. These algebras
are clearly proper and act irreducibly. The only case when $V(\KK,\KK')\simeq\fg(\KK,\KK')$ is the
geometry modelled after $\frac{\sSU(3)\times\sSU(3)}{\sSU(3)}$, i.e. $\KK=\CC$ and $\KK'=\CC$ (it suffices
to notice that there is a single 8-dimensional irreducible representation of $\sSU(3)$, and thus both the isotropy and adjoint representations are equivalent). The, we apply Corollary \ref{cor-nagy}
and the Proposition follows.
\end{proof}

As we have mentioned, such unique connection is called \emph{the} characteristic
connection, and its torsion tensor -- the characteristic torsion. Indeed, one may consider
the latter as characterising the geometry in a manner more convenient than the intrinsic torsion (however the two can be clearly obtained from each other). 

Of course, these notions make sense only if the
connection exists. Bobienski and Nurowski \cite{bobienski-2005} proposed nearly integrability
as a candidate for an existence condition, and checked that it is one indeed for the geometry they
considered -- i.e. the one modelled after $ \sSU(3)/\sSO(3).$ It then turned out \cite{nurowski-2006}
that it also works for the next two geometries in the first column, failing however in case of
the last one, i.e. $\sE_6 / \sF_4$:
\begin{pro}[Nurowski \cite{nurowski-2006}]
Let $(M,g,\UpsilonM)$ be a $\fg(\KK,\CC)$-geometry with $\KK\neq\OO.$ Then $M$ admits a skew-torsion
compatible connection iff $\UpsilonM$ is nearly-integrable.
\end{pro}

In the following, we extend the equivalence between nearly-integrability and existence of a characteristic connection onto the second family.
\begin{thm}\label{thm-nearly}
Let $(M,g,\XiM)$ be a $\fg(\KK,\HH)$-geometry. Then $M$ admits a characteristic connection iff
$\XiM$ is nearly-integrable.
\end{thm}
\begin{proof} That nearly-integrability is implied by existence of the characteristic connection follows
immediately from Proposition \ref{pro-nearly}. To prove the converse, we will apply Lemma \ref{lem-nearly} to the map
$$ \Xi' : \cV_2(\KK) \otimes \Lam^2 \cV_2(\KK) \to \Sym^{7} \cV_2(\KK) $$
$$ \Xi'(A)_{a_0a_1\dots a_6} = A_{(a_0}{}^m{}_{a_1} \XiM_{a_2\dots a_6)m}. $$
It is clear that its kernel contains $\cV_2(\KK)\otimes\fg(\KK,\HH) + \Lam^3 \cV_2(\KK).$
In order to check that these are actually equal, we extend $\Xi'$ by complex linearity and consider
a decomposition of a generic element $A \in \cV_2(\KK) \otimes \Lam^2 \cV_2(\KK):$
$$ A = O + B + C + \bar O + \bar B + \bar C $$
$$ O \in \Lam^{1,0} \otimes \Lam^{2,0},\quad B \in \Lam^{0,1}\otimes\Lam^{2,0},
\quad C \in \Lam^{1,0} \otimes \Lam^{1,1} $$
where $ C_{abc} = c_{\alpha\bar\beta\gamma} - c_{\alpha\bar\gamma\beta}. $
The action of extended $\Xi'$ gives then
$$ \Xi'(O) \in\Sym^{5,2},\quad \Xi'(B + C) \in \Sym^{4,3} $$
$$ \Xi'(\bar O) \in\Sym^{2,5},\quad \Xi'(\bar B + \bar C) \in \Sym^{3,4} $$
and $\Xi'(A)=0$ iff each of these vanishes separately.
Applying the isomorphisms (\ref{sym-split}), we find $\Xi'(O) = 0$ and $\Xi'(B+C)=0$ to be equivalent to
\begin{eqnarray*}
\Lambda_{(\alpha\beta\gamma} O_{\delta}{}^{\bar\mu}{}_{\epsilon)} \bar\Lambda_{\bar\phi\bar\kappa\bar\mu} &=& 0 \\
\bar\Lambda_{(\bar\mu\bar\phi\bar\kappa} B_{\bar\epsilon)}{}^{\bar\mu}{}_{(\alpha} \Lambda_{\beta\gamma\delta)}
+ \Lambda_{\mu(\alpha\beta} c_{\gamma}{}^{\mu}{}_{\delta)} -  \Lambda_{(\alpha\beta\gamma} c_{\delta)(\bar\epsilon}{}^{\bar\mu} \bar\Lambda_{\bar\phi\bar\kappa)\bar\mu} &=& 0.
\end{eqnarray*}
Contracting the first equation with $\Lambda^{\phi\kappa\nu}$ and the second one with $\bar\Lambda^{\alpha\beta\gamma},$ we obtain
\begin{eqnarray}
\Lambda_{(\alpha\beta\gamma} O_{\delta}{}^{\bar\nu}{}_{\epsilon)} &=& 0 \label{e1} \\
X_\delta{}^{\bar\mu}{}_{(\bar\epsilon} \bar\Lambda_{\bar\phi\bar\kappa)\bar\mu} + \xi_\delta \bar\Lambda_{\bar\epsilon\bar\phi\bar\kappa} &=& 0 \label{e2}
\end{eqnarray}
where
$$ \xi_\delta = \frac{1}{N+3} [ c_\delta{}^{\mu}{}_\mu + c_\mu{}^{\mu}{}_\delta +
2 c_\alpha{}^\mu{}_\beta \bar\Lambda^{\alpha\beta\gamma} \Lambda_{\gamma\mu\delta} ] $$
$$ X_{\delta\mu\bar\epsilon} = B_{\bar\epsilon\mu\delta} - c_{\delta\bar\epsilon\mu}. $$
Equation (\ref{e1}) gives simply $O_{(\alpha}{}^{\bar\mu}{}_{\beta)}=0,$ so that $ O \in \Lam^{3,0}.$
Writing $X$ in the form
$$ X_{\delta\mu\bar\nu} = -\xi_\delta h_{\bar\nu\mu} + X'_{\delta\mu\bar\nu} $$
equation (\ref{e2}) simplifies to $ X'_\delta(\bar\Lambda) = 0, $ so that 
$$X'_{\delta\mu\bar\nu} - X'_{\delta\nu\bar\mu} \in\Lam^{1,0}_\delta \otimes \cg_2(\KK)_{mn}. $$
and thus
$$X_{\delta\mu\bar\nu} - X_{\delta\nu\bar\mu} \in\Lam^{1,0}_\delta \otimes \fg(\KK,\HH)_{mn}. $$
Now, using $X$ and $B$ to eliminate $c,$ we get
\begin{eqnarray*}
A_{abc} 
&=& O_{\alpha\beta\gamma} + B_{\bar\alpha\beta\gamma} + B_{\bar\beta\gamma\alpha} + B_{\bar\gamma\alpha\beta} + X_{\alpha\beta\bar\gamma} - X_{\alpha\gamma\bar\beta} 
\\
&+& \bar O_{\bar\alpha\bar\beta\bar\gamma} + \bar B_{\alpha\bar\beta\bar\gamma} + \bar B_{\beta\bar\gamma\bar\alpha} + \bar B_{\gamma\bar\alpha\bar\beta} + \bar X_{\bar\alpha\bar\beta\gamma} - \bar X_{\bar\alpha\bar\gamma\beta},
\end{eqnarray*}
where the terms involving $B,$ due to $B_{\bar\gamma\alpha\beta} = -B_{\bar\gamma\beta\alpha},$ give
an element of $\Lam^{2,1} \oplus \Lam^{1,2}.$ We thus finally obtain
$$ A \in \Re [\Lam^{3,0} \oplus\Lam^{2,1} \oplus \Lam^{1,2} \oplus \Lam^{0,3} \oplus \Lam^{1,0} \otimes
\fg(\KK,\HH) \oplus \Lam^{0,1} \otimes \fg(\KK,\HH) ], $$
so that $ \ker\Xi' = \cV_2(\KK)\otimes\fg(\KK,\HH) + \Lam^3\cV_2(\KK),$
and the theorem follows by Lemma \ref{lem-nearly}.
\end{proof}

It is instructive to note that the latter proof relies on the complex structure of $\cV_2(\KK):$
indeed, in spite of complete symmetrization in the definition of $\Xi',$ using the projections (\ref{sym-split}) splits linear and antilinear indices and allows for separate contraction.
That did not occur in the first family, and thus Nurowski had to resort to explicit calculations using computer algebra. Unfortunately, this method fails also for the third family, since none of the complex structures is invariant. The question whether nearly-integrability guarantees existence of a characteristic connection for $\fg(\KK,\OO)$-geometries remains open.

\section{Decompositions of two- and three-forms} \label{sec-dec}
In order to classify possible $\fg(\KK,\KK')$-geometries with skew torsion,
we need to decompose the space of three-forms into $G(\KK,\KK')-$irreducibles.
This is easily done using the computer algebra package
{\tt LiE} by Marc van Leeuwen et al. \cite{leeuwen}.

One needs to note that we are
ultimately interested in \emph{real} representations, while the package we use operates on complex ones.
Thus in the first and third family, where the representations are respectively real and quaternionic, we consider complexifications of $\cV_1(\KK)$ and $\cV_3(\KK),$ and, after decomposing, 
take the real section (recalling that $\Re(V_n \oplus \bar V_n),$ with $V_n$ being some complex irrep, does not decompose).
On the other hand, the representations in the second column are already unitary, and can be dealt with directly.

We only provide here the dimensions of the irreducible subspaces. 
The invariants can be used to construct suitable projection operators.

\subsection{First family}
The following was already done by Nurowski \cite{nurowski-2006}.

\begin{pro} With $W_n$ denoting an $n-$dimensional real representation, we have the following decompositions
into:
\begin{enumerate}
\item $\sSO(3)$ irreps:
\begin{eqnarray*}
\Lam^2 \cV_1(\RR) &\simeq& \fso(3) \oplus W_7 \\
\Lam^3 \cV_1(\RR) &\simeq& \fso(3) \oplus W_7
\end{eqnarray*}
\item $\sSU(3)$ irreps:
\begin{eqnarray*}
\Lam^2 \cV_1(\CC) &\simeq& \fsu(3) \oplus W_{20} \\
\Lam^3 \cV_1(\CC) &\simeq& \fsu(3) \oplus W_{20} \oplus W_{27} \oplus \RR.
\end{eqnarray*}
\item $\sSp(3)$ irreps:
\begin{eqnarray*}
\Lam^2 \cV_1(\HH) &\simeq& \fsp(3) \oplus W_{70} \\
\Lam^3 \cV_1(\HH) &\simeq& \fsp(3) \oplus W_{70} \oplus W_{84} \oplus W_{189}.
\end{eqnarray*}
\item $\sF_4$ irreps:
\begin{eqnarray*}
\Lam^2 \cV_1(\OO) &\simeq& \ff_4 \oplus W_{273} \\
\Lam^3 \cV_1(\OO) &\simeq& W_{273} \oplus W_{1274} \oplus W_{1053}.
\end{eqnarray*}
\end{enumerate}\end{pro}
The $\sSU(3)$ case is special, since the isotropy representation is equivalent to the adjoint one,
$\cV_1(\CC) \simeq \fsu(3).$ In particular, there exists an invariant three-form (spanning the $\RR\subset\Lam^3 \cV_1(\CC)$ subspace) corresponding to the structure constants of $\fsu(3).$ 

\subsection{Second family}
Recall first, that the spaces of two- and three-forms on the second family spaces decompose under $\sU(\cV_2(\KK))$
into
\begin{eqnarray*}
\Lam^2 \cV_2(\KK) &\simeq& \Re \Lam^{1,1} \oplus \Re(\Lam^{2,0} \oplus \Lam^{0,2})  \\
&=& [\RR \oplus \fsu(\cV_2(\KK))] \oplus \fu^\bot(\cV_2(\KK)),
\end{eqnarray*}
\begin{eqnarray*} \Lam^3 \cV_3(\KK) &\simeq& \Re (\Lam^{2,1}\oplus\Lam^{1,2}) \oplus \Re (\Lam^{3,0} \oplus \Lam^{0,3}) \\
&=& [\cW_3 + \cW_4](\cV_2(\KK)) \oplus \cW_1(\cV_2(\KK)), \end{eqnarray*}
where $\cW_1$, $\cW_3$ and $\cW_4$ are 
completely skew analogues of the usual spaces introduced in the Gray-Hervella 
classification of almost-Hermitian structures \cite{gray-1980}.
These are further decomposed when the unitary group is reduced to one of our gropus $G(\KK,\HH).$
\begin{pro} With $V_n$ denoting (a realification of) an $n-$dimensional complex representation, we have the following decompositions into:
\begin{enumerate}
\item $\sU(3)$ irreps:
\begin{eqnarray*}
\fsu(\cV_2(\RR)) &\simeq& \fsu(3) \oplus V_{27} \\
\fu^\bot(\cV_2(\RR)) &\simeq& V_{15} \\
\cW_1(\cV_2(\RR)) &\simeq& V_{10}\oplus \bar V_{10} \\
\cW_3(\cV_2(\RR)) &\simeq& V_3 \oplus V_{15} \oplus V_{24} \oplus V_{42}
\end{eqnarray*}
\item $S(\sU(3)\times\sU(3))$ irreps:
\begin{eqnarray*}
\fsu(\cV_2(\CC)) &\simeq& \{\fsu(3)\oplus\fsu(3)\} \oplus V_{64} \\
\fu^\bot(\cV_2(\CC)) &\simeq& \{V_{18} \oplus V_{18}\} \\
\cW_1(\cV_2(\CC)) &\simeq& \{V_{10}\oplus V_{10}\} \oplus V_{64} \\
\cW_3(\cV_2(\CC)) &\simeq& \cV_2(\CC)\oplus \{V_{18}\oplus V_{18}\} \oplus 
\{ V_{45}\oplus V_{45} \}  \oplus \{ V_{90}\oplus V_{90} \},
\end{eqnarray*}
where in $\{ V_{n} \oplus V_{n} \}$ the first (resp. second) copy of $V_{n},$ being a $\sSU(3)$ irrep, is affected only by the first (resp. second) $\sSU(3)$ in $\sSU(3)\times\sSU(3).$
\item $\sU(6)$ irreps:
\begin{eqnarray*}
\fsu(\cV_2(\HH)) &\simeq& \fsu(6) \oplus V_{189} \\
\fu^\bot(\cV_2(\HH)) &\simeq& V_{105} \\
\cW_1(\cV_2(\HH)) &\simeq& V_{175} \oplus V_{280} \\
\cW_3(\cV_2(\HH)) &\simeq&  V_{21} \oplus \bar V_{105} \oplus V_{384} \oplus V_{1050}
\end{eqnarray*}
\item $\sE_6\cdot\sU(1)$ irreps:
\begin{eqnarray*}
\fsu(\cV_2(\OO)) &\simeq& \fe_6 \oplus V_{650} \\
\fu^\bot(\cV_2(\OO)) &\simeq& V_{351} \\
\cW_1(\cV_2(\OO)) &\simeq& V_{2925} \\
\cW_3(\cV_2(\OO)) &\simeq& \bar V_{351} \oplus V_{1728} \oplus V_{7371}
\end{eqnarray*}

\end{enumerate}
\end{pro}
The space $\cW_4(\cV_2(\KK)) = \cV_2(\KK) \wedge \theta, $
where $\theta_{ab} = i h_{\bar\alpha\beta} - i h_{\bar\beta\alpha},$
is already isomorphic to $\cV_2(\KK)$ and thus irreducible as a $G(\KK,\HH)$-module.

\subsection{Third family}
\begin{pro} With $W_n$ denoting an $n-$dimensional real representation, we have the following decompositions
into:
\begin{enumerate}
\item $\sSp(3)\sSp(1)$ irreps:
\begin{eqnarray*}
\Lam^2 \cV_3(\RR) &\simeq& \fsp(3) \oplus \fsp(1) \oplus W_{84} \oplus W_{270} \\
\Lam^3 \cV_3(\RR) &\simeq& \cV_3(\RR) \oplus W_{56} \oplus W_{128} \oplus W_{432} \oplus W_{1232} \oplus W_{1400}
\end{eqnarray*}
\item $\sSU(6)\sSp(1)$ irreps:
\begin{eqnarray*}
\Lam^2 \cV_3(\CC) &\simeq& \fsu(6) \oplus \fsp(1) \oplus W_{175} \oplus W_{576} \\
\Lam^3 \cV_3(\CC) &\simeq& \cV_3(\CC) \oplus W_{80} \oplus W_{280} \oplus W_{1080} \oplus W_{3920} \oplus W_{4480}
\end{eqnarray*}
\item $\sSO(12)\sSp(1)$ irreps:
\begin{eqnarray*}
\Lam^2 \cV_3(\HH) &\simeq& \fso(12) \oplus \fsp(1) \oplus W_{462} \oplus W_{1485} \\
\Lam^3 \cV_3(\HH) &\simeq& \cV_3(\HH) \oplus W_{128} \oplus W_{704} \oplus W_{3456} \oplus W_{17600} \oplus W_{19712}
\end{eqnarray*}
\item $\sE_7\sSp(1)$ irreps:
\begin{eqnarray*}
\Lam^2 \cV_3(\OO) &\simeq& \fe_7 \oplus \fsp(1) \oplus W_{1463} \oplus W_{4617} \\
\Lam^3 \cV_3(\OO) &\simeq& \cV_3(\OO) \oplus W_{224} \oplus W_{1824} \oplus W_{12960} \oplus W_{102144} \oplus W_{110656}
\end{eqnarray*}

\end{enumerate}\end{pro}
Note that one always has $\cV_3(\KK) \subset \Lam^3 \cV_3(\KK)$ with the intertwiner given
by the natural $\sSp(\cV_3(\KK))\sSp(1)$-invariant 4-form $I\wedge I + J\wedge J + K\wedge K.$

\subsection{A classification}
Let $M$ be a $\fg(\KK,\KK')$-geometry. Then the
 decomposition of $\Lam^3 \cV_n(\KK)$ 
 into $G(\KK,\KK')$-irreducibles
 $$ \Lam^3 \cV_n(\KK) = \bigoplus_{r=1}^s W^{(r)} $$
 (where  $n=1,2,3$ for $\KK'$ being, respectively, $\CC,\HH,\OO$)
 gives rise to a decomposition of the bundle of three-forms 
 $$ \Lam^3 TM = \bigoplus_{r=1}^s \mathcal{T}^{(r)} $$
 into subbundles such that in an adapted frame $e_x : \cV_n(\KK) \to T_x M$ one has $e_x(W^{(r)}) = \mathcal{T}^{(r)}_x.$ As the spaces $W^{(r)}$ are $G(\KK,\KK')$-invariant, the latter decomposition does
 not depend on the choice of a frame.

Assume now that $M$ admits a characteristic connection, with a characteristic torsion
$T^c \in \Omega^0(M,\Lam^3 TM).$ It follows that $M$ falls into one of $2^s$ classes, numbered by 
$$ t(M) = \sum_{r=1}^s 2^{r-1} t_r(M),\quad t_r(M)=\left\lbrace\begin{matrix}
\ 0 &\ & \pr_{\mathcal{T}^{(r)}} T^c = 0 \\
\ 1 &\ & \textrm{otherwise}
\end{matrix}\right.$$
This way we obtain: 
\begin{itemize}
\item 4 classes of $\fg(\RR,\CC)$-geometries with characteristic torsion.
\item 16 classes of $\fg(\HH,\CC)$-geometries with characteristic torsion.
\item 8 classes of $\fg(\OO,\CC)$-geometries with characteristic torsion.
\item 128 classes of $\fg(\KK,\HH)$-geometries with characteristic torsion, where $\KK=\RR,\CC,\HH.$
\item 32 classes of $\fg(\OO,\HH)$-geometries with characteristic torsion.
\item 64 classes of $\fg(\KK,\OO)$-geometries with characteristic torsion.
\end{itemize}
A connection with skew torsion is not unique for a $\fg(\CC,\CC)$-geometry, due to the one-dimensional
subspace
$\RR \subset \Lam^3 \cV_1(\CC)$ which is also contained in $\cV_1(\CC) \otimes \fg(\CC,\CC).$ There is
thus no genuine characteristic torsion -- instead, one has a one-parameter family of connections
whose torsion differs by a section of \emph{the} one-dimensional $G(\CC,\CC)-$invariant subbundle $\mathcal{T}^{(4)} \subset \Lam^3 TM.$
(Nurowski \cite{nurowski-2006} introduces the notion of restricted nearly-integrability to rule this subbundle out).  Anyway, we can still perform analogous classification, projecting the skew torsion
of any compatible connection onto the complement of $\mathcal{T}^{(4)}$. We thus have in addition:
\begin{itemize}
\item 8 classes of $\fg(\CC,\CC)$-geometries admitting a skew-torsion compatible connection.
\end{itemize}

\section{Locally reductive $\fg(\KK,\KK')$-geometries}
While symmetric spaces provide integrable models for $\fg(\KK,\KK')$-geometries, it is the reductive spaces that give homogeneous examples with nontrivial characteristic torsion (cf. \cite{agricola-2002,kndg},). We will first show how
certain locally 
reductive spaces become equipped with a $G(\KK,\KK')$-structure, and how such reductive spaces
can be obtained from the symmetric models. Then, we shall perform a construction of such
spaces at the Lie-algebraic level, having previously shown how to extend results derived for a
single pair $(\KK,\KK')$ onto all the `larger' cases.

\subsection{The existence theorem}
In what follows, we construct a manifold equipped with a $G$-structure from adequate Lie-algebraic data.
\begin{lem}\label{lem-red0}
Let $G$ be a subgroup of $\sO(m)$ with $\fg\subset\fso(m)$ its Lie algebra. Let $\fk$ be a Lie algebra admitting a reductive, non-symmetric decomposition
$$ \fk= \fh \oplus \fl$$
$$ [\fh,\fh]\subset\fh,\quad[\fh,\fl]\subset\fl,\quad \pr_\fl[\fl,\fl]\neq 0 $$
such that $\fl$ possesses a $\fk$-invariant scalar product. Let there moreover exist a Lie algebra
monomorphism $\varphi:\fh\to \fg$ and an orthogonal isomorphism $\psi:\fl\to\RR^m$ satisfying
$$ \varphi(A)\circ\psi = \psi \circ \ad_A $$
for each $A\in\fh.$

Then there exists an $m$-dimensional manifold equipped with a $G$-structure admitting a compatible connection with nontrivial skew torsion.
\end{lem}
\begin{proof}
Let $K$ be a Lie group with Lie alebra $\fk$ (for example the simply-connected one). While
one is tempted to produce a subgroup from $\fh\subset\fk$ and form a reductive homogeneous space by
taking a quotient, it is not trivial to guarantee that the subgroup is closed. Instead, we shall perform
a local construction.

Let $\fh^L\subset TK$ be the left-invariant distribution such that $\fh^L(g) = (T_e L_g)\fh$ 
for each $g\in K.$ Since $\fh$ is a subalgebra in $\fk,$ it follows that $\fh^L$ is integrable
and there exists a neighbourhood $Q_0\subset K$ of identity, such that $\fh^L$ gives rise to a foliation
of $Q_0,$ together with a projection $\pi_0 : Q_0\to M$ onto the space $M$ of leaves, a (smooth) manifold.

Possibly choosing smaller $Q_0,$ consider a global section $\sigma_0 : M\to Q_0\subset K,$ i.e.
a map satisfying $\pi_0\circ\sigma_0 = \id_M.$ The Maurer-Cartan form $\vartheta_{MC}\in\Omega^1(K)\otimes\fk$ pulls back to
$$ \vartheta = \sigma_0^* \vartheta_{MC} \in \Omega^1(M)\otimes\fk, $$
and the latter is further decomposed into
$$\underline\theta\in\Omega^1(M)\otimes\RR^m,\quad \underline\omega\in\Omega^1(M)\otimes\fg$$
$$ \underline\theta = \psi\circ\pr_\fl\vartheta,\quad\underline\omega=\varphi\circ\pr_\fh\vartheta.$$

Moreover $\underline\theta(\pi_0(e)): T_{\pi_0(e)}M \to \RR^m$ is an isomorphism and extends to
a coframe on entire $M$. The structure equation $d\vartheta_{MC} + \frac{1}{2}[\vartheta_{MC},\vartheta_{MC}] = 0$ pulls back to
\begin{eqnarray*}
\underline\Omega(X,Y) &=& -(\varphi\circ\pr_\fh)[\psi^{-1}\underline\theta(X),\psi^{-1}\underline\theta(Y)] \\
\underline\Theta(X,Y) &=& -(\psi\circ\pr_\fl)[\psi^{-1}\underline\theta(X),\psi^{-1}\underline\theta(Y)]
\end{eqnarray*}
for $X,Y\in T_xM$, where $\underline\Omega=d\underline\omega+\underline\omega\wedge\underline\omega$
and $\underline\Theta=d\underline\theta+\underline\omega\wedge\underline\theta.$ Since $\psi$ is
orthogonal, it follows that
\begin{equation}\label{eq-pi}
\langle\underline\Theta(X,Y),Z\rangle = -\langle\underline\Theta(X,Z),Y\rangle.
\end{equation}

Considering $\underline\theta$ as an adapted orthogonal coframe on $M$, we equip the latter with
a compatible metric, and a $G$-structure $\pi:Q\to M$. The fibre over $x$ of
the latter consists of all frames $e_x:\RR^m\to T_x M$ such that there exists $g\in G$ such
that $e_x^{-1}(g\underline\theta(X))=X$ for all $X\in T_xM.$

The frame dual to $\underline\theta$ is a section $\sigma:M\to Q$ with $\sigma^*\theta=\underline\theta$,
where $\theta$ is the soldering form on $Q$. Considering
$$ \Gamma^\omega\in\Omega^1(M,\fg_M),\quad \Gamma^\omega(X)=\sigma_x(\underline\omega(X))\quad\textrm{for}\ X\in T_xM$$
as a connection form relative to the frame $\sigma,$ we obtain a $Q$-compatible connection $\nabla^\omega$ in the tangent bundle, whose torsion tensor
$$ T^\omega\in\Omega^2(M,TM),\quad T^\omega(X,Y)=\sigma_x(\underline\Theta(X,Y))\quad\textrm{for}\ X,Y\in T_x M$$
is completely skew due to (\ref{eq-pi}). Moreover, it is nontrivial, since $\pr_\fl[\fl,\fl]\neq0.$
\end{proof}

The latter lemma reduces the problem of producing locally reductive $\fg(\KK,\KK')$-geometries
with characteristic torsion
to an algebraic one. This in turn can be first dealt with on a Lie-algebraic level. 
Recall first the notation introduced when decomposing the Magic Square algebras into symmetric pairs,
Subsection \ref{ss-symd} of the previous Chapter.
The idea 
is, being given an original symmetric decomposition $\fm(\KK,\KK')= \fg(\KK,\KK') \oplus V(\KK,\KK'),$ to produce a reductive
pair by reducing $\fg(\KK,\KK')$ to a subalgebra $\fh$ and twisting
the $[V(\KK,\KK'),V(\KK,\KK')]$ bracket. 

The following proposition performs such a construction
whenever the decomposition of $V(\KK,\KK')$ into $\fh$-irreducibles includes $\fh$ itself.

\begin{lem} \label{lem-red}
Let $\fm = \fg \oplus V$ be a symmetric pair, with $\fm$ semisimple.
Let $\fh$ be a subalgebra of $\fg$
such that there exists an orthogonal (with respect to the Killing form) map
$$ B : \fh \to V $$
equivariant under the adjoint action of $\fh.$

Let us equip the space $\fk = \fh\oplus V$ with the following bracket:
\begin{eqnarray*}
[A,A']_\fk &=& [A,A'] \\ {}
[A,X]_\fk &=& [A,X] \\ {}
[X,Y]_\fk &=& \pr_\fh[X,Y] \\ 
&+& B(\pr_\fh[X,Y]) + [B^*(X),Y] - [B^*(Y),X] \\&-& B([B^*(X),B^*(Y)])
\end{eqnarray*}
for $A,A'\in\fh$ and $X,Y\in V,$
where the commutators on the r.h.s are those in $\fm$ and $B^*$ is the adjoint of $B$
with respect to the Killing form.

Then $\fk$ becomes a Lie algebra, and 
the decomposition $\fk=\fh\oplus V$ is reductive pair. Moreover,
$\fk$ posesses an invariant nondegenerate quadratic form which,
restricted to $V,$ coincides with the Killing form of $\fm.$
\end{lem}
The proof is by an explicit verification of the Jacobi identities, and can be found
in Section \ref{sec-xprf}.
We finally arrive at the following result, which we shall use in the next section.
\begin{thm}\label{thm-red} Let $\fh\subset\fg(\KK,\KK')$ be a simple subalgebra 
and $W \subset V(\KK,\KK')$
a $\fh$-submodule equivalent to $\fh.$
Let moreover the action of $\fh$ on 
the orthogonal complement of $W$ be nontrivial.

Then there exists a $\fg(\KK,\KK')$-geometry admitting a compatible connection with nonvanishing
skew torsion.
\end{thm}

\begin{proof}
We first apply Lemma \ref{lem-red} to the symmetric decomposition $\fm(\KK,\KK')
= \fg(\KK,\KK') \oplus V(\KK,\KK')$ and the subalgebra $\fh$ with the map $B$ being
the intertwiner between $\fh$ and $W$ (due to semisimplicity both $\fh$ and $W$ are
irreducible, and thus $B$ is automatically orthogonal).

This way we obtained a reductive pair $\fh \oplus V(\KK,\KK')=\fk$ 
such that the action on $\fh$ on $V(\KK,\KK')$ in $\fk$ is the same as in $\fm(\KK,\KK').$
Moreover, for $X,Y\in W^\bot$ (the orthogonal complement of $W$), we have
$$ \pr_{V(\KK,\KK')} [X,Y]_{\fk} = (B\circ\pr_\fh) [X,Y]. $$
Then, since the action on $\fh$ on $W^\bot$ is nontrivial, there exist such 
$X,Y\in W^\bot$ that 
$ \pr_{V(\KK,\KK')} [X,Y]_{\fk} \neq 0. $ It thus follows that the reductive pair we have obtained
is not symmetric.

We can finally apply Lemma \ref{lem-red0} to obtain a Riemannian manifold $(M,g)$ equipped
with a $G(\KK,\KK')$-structure admitting a compatible connection with nontrivial skew torsion. Then,
choosing an adapted frame $e$, we can equip $M$ wich a corresponding tensor $\YM(x) = e_x(\Y)$,
where $\Y$ is one of $\Upsilon,\Xi,\mho,$ depending on $\KK'$. Clearly, $\YM$ does not depend on the 
choice of a frame.
\end{proof}

\subsection{A generalizing proposition}
We are now going to
give several examples of subalgebras $\fh\subset\fg(\KK,\KK')$ satisfying the
conditions of Theorem \ref{thm-red}. While we tried to keep previous sections
of the present chapter possibly independent of the Jordan-algebraic constructions,
here we will need the structures introduced in Sections \ref{sec-tits} and \ref{sec-fts} of
the previous chapter.

It useful to note that an example for a
pair $(\KK,\KK')$ gives rise to examples for all `further' such pairs, with the ordering
given by inclusions:
\begin{pro} \label{pro-incl}
Let $\tilde\KK\supset\KK$ and $\tilde\KK'\supset\KK'$ where $\KK,\tilde\KK$ are chosen
from $\RR,\CC,\HH,\OO$ and $\KK',\tilde\KK'$ from $\CC,\HH,\OO.$ Assume there is a subalgebra $\fh\subset\fg(\KK,\KK')$ and a subspace $W\subset V(\KK,\KK')$
satisfying the conditions of Theorem \ref{thm-red}. 

Then there exist natural inclusions $\fh\hookrightarrow\fg(\tilde\KK,\tilde\KK')$ and $W\hookrightarrow V(\KK,\KK')$ such that their images satisfy the conditions of Theorem \ref{thm-red} for $\fg(\tilde\KK,\tilde\KK')$-geometries.
\end{pro}
We need the following Lemma, deriving form the properties of the Magic Square:
\begin{lem} $ $ \begin{enumerate}
\item
Let $\KK,\tilde\KK$ and $\KK'$ be as in Proposition \ref{pro-incl}. 
Then there exist natural monomorphisms (of Lie algebras and of vector spaces):
\begin{eqnarray*}
i_\fg &:& \fg(\KK,\KK') \to \fg(\tilde\KK,\KK') \\
i_V &:& V(\KK,\KK') \to V(\tilde\KK,\KK')
\end{eqnarray*}
such that
$$ [i_\fg(E), i_V (X)] = i_V [E,X] $$
for $E\in\fg(\KK,\KK')$ and $X\in V(\KK,\KK'),$
with the l.h.s bracket taken in $\fm(\KK,\KK')$ and the r.h.s. one in $\fm(\tilde\KK,\KK').$
\item
Let $\tilde\KK$ and $\KK',\tilde\KK'$ be as in Proposition \ref{pro-incl}. 
Then there exist natural monomorphisms (of Lie algebras and of vector spaces):
\begin{eqnarray*}
j_\fg &:& \fg(\tilde\KK,\KK') \to \fg(\tilde\KK,\tilde\KK') \\
j_V &:& V(\tilde\KK,\KK') \to V(\tilde\KK,\tilde\KK')
\end{eqnarray*}
such that
$$ [j_\fg(E), j_V (X)] = j_V [E,X] $$
for $E\in\fg(\tilde\KK,\KK')$ and $X\in V(\tilde\KK,\KK'),$
with the l.h.s bracket taken in $\fm(\tilde\KK,\KK')$ and the r.h.s. one in $\fm(\tilde\KK,\tilde\KK').$
\end{enumerate}
\end{lem}
\begin{proof} $ $
\begin{enumerate}
\item Recall that
\begin{eqnarray*}
\fg(\KK,\KK') =& \der\hxk\ \oplus & \deg\KK' \oplus \KK'_0 \otimes \sxk \\
V(\KK,\KK') =& & \dev\KK' \oplus \KK'_1 \otimes \sxk,
\end{eqnarray*}
while
\begin{eqnarray*}
\fg(\tilde\KK,\KK') =& \der\hx\tilde\KK\ \oplus & \deg\KK' \oplus \KK'_0 \otimes \sx\tilde\KK \\
V(\tilde\KK,\KK') =& & \dev\KK' \oplus \KK'_1 \otimes \sx\tilde\KK.
\end{eqnarray*}
One easily checks that the inclusion $\KK\hookrightarrow\tilde\KK$ induces a Jordan algebra
monomorphism $$\mu : \hxk\hookrightarrow\hx\tilde\KK.$$
Recalling that the derivations of $\hxk$
are defined as an image of the $\der\hxk-$equivariant map $\cD$ defined in Lemma \ref{lem-dmapj}: 
$$\der\hxk = \{ \cD_{X,Y}\ |\ X,Y\in \sxk\},$$
we see that $\mu$ induces a map 
\begin{equation}\label{map-ddt}
\der\hxk \ni \cD_{X,Y} \mapsto \cD_{\mu(X),\mu(Y)} \in \der\hx\tilde\KK
\end{equation}
such that 
\begin{equation}\cD_{\mu(X),\mu(Y)}(\mu(Z)) = \mu(\cD_{X,Y}(Z)) \label{eq-i0} \end{equation}
for $X,Y\in\sxk$ and $Z\in\hxk.$
Due to equivariance of $\cD,$ the map (\ref{map-ddt}) gives a monomorphism of Lie algebras, 
$$\gamma : \der\hxk \to \der\hx\tilde\KK,$$ and (\ref{eq-i0}) translates to $ \gamma(D)(\mu(X)) = \mu(D(X)) $
for $D\in\der\hxk$ and $X\in\sxk.$

It then follows that the maps
\begin{eqnarray*}
i_\fg =& \gamma\ \oplus& \id \oplus \id \otimes \mu|_{\sxk} \\
i_V =& & \id \oplus \id\otimes \mu|_{\sxk}
\end{eqnarray*}
are such as claimed by the Lemma.
\item Recalling Proposition \ref{pro-isoreps}, we have the following identifications:
\begin{eqnarray*}
V(\tilde\KK,\CC) &\simeq& \cV_1(\tilde\KK) = \sx\tilde\KK \\
V(\tilde\KK,\HH)&\simeq&\cV_2(\tilde\KK) = \CC\otimes\hx\tilde\KK \\
V(\tilde\KK,\OO)&\simeq&\cV_3(\tilde\KK) \simeq (\RR \oplus \hx\tilde\KK)\otimes\CC^2 
\end{eqnarray*}
and the algebras $\fg(\tilde\KK,\KK')$ seen as endomorphisms of the latter spaces are
\begin{eqnarray*}
\fg(\tilde\KK,\CC) &=& \der\hx\tilde\KK \\
\fg(\tilde\KK,\HH) &=& \der\hx\tilde\KK \oplus i L_{\hx\tilde\KK} \\
\fg(\tilde \KK,\OO) &=& \cH_0(\der\hx\tilde\KK \oplus i\hx\tilde\KK) \\
&\oplus& \cH_1(1\otimes\hx\tilde\KK \oplus i\otimes\hx\tilde\KK) \oplus \fsp(1),
\end{eqnarray*}
where $\cH_1$ and $\cH_2$ are the maps defined in Lemma \ref{lem-derfts} and extended by complex linearity, and $\fsp(1)$ denotes the additional algebra related to a quaternion-hermitian structure, as described in Proposition \ref{pro-isoreps}. In $\fg(\tilde\KK,\HH)$ we have included the additional $\fu(1)$
as spanned by $i L_1 \in i L_{\hx\tilde\KK}.$

Let us now consider monomorphisms
$$ \begin{CD} \cV_1(\tilde\KK) @>{\mu}>> \cV_2(\tilde\KK) @>{\nu}>> \cV_3(\tilde\KK) \end{CD} $$
$$ \begin{CD} \fg(\tilde\KK,\CC) @>{\alpha}>> \fg(\tilde\KK,\HH) @>{\beta}>> \fg(\tilde\KK,\OO) \end{CD} $$
given by
$$ \mu(X) = 1 \otimes X,\quad \nu(z\otimes Y) = (0,Y) \otimes (z,0) $$
for $X\in\sx\tilde\KK,$ $Y\in\hx\tilde\KK$ and $z\in\CC,$ and
$$
\alpha(D) =  D,\quad \beta( D + i L_X ) = \cH_0(D + i X) $$
for $D\in\der\hx\tilde\KK$ and $ X\in\hx\tilde\KK$.

It is then readily checked that $\alpha(D)(\mu(X)) = \mu(D(X))$ for $D\in\der\hx\tilde\KK$ and $X\in\sx\tilde\KK.$ Comparing with the formula for $\cH_0$ given in Lemma \ref{lem-derfts}, one
also easily verifies $\beta(E)(\nu(Z)) = \nu(E(Z))$ for $E\in\fg(\tilde\KK,\HH)$ and $Z\in\cV_2(\tilde\KK).$

Finally, the map $j_\fg$ is given by either $\id$, $\alpha$, $\beta$ or $\beta\circ\alpha$ (depending
on $\KK'$ and $\tilde\KK'$),
while $j_V$ is respectively $\id$, $\mu$, $\nu$ or $\nu\circ\mu.$ It follows from the previous paragraph, that $j_\fg$ and $j_V$ are
such as claimed by the Lemma.
\end{enumerate}
\end{proof}
By composing $m_\fg=j_\fg\circ i_\fg$ and $m_v=j_V\circ i_V$ we have the obvious
\begin{cor}\label{cor-incl}
Let $\KK,\tilde\KK,\KK',\tilde\KK'$ be as in Proposition \ref{pro-incl}. 
Then there exist natural monomorphisms (of Lie algebras and of vector spaces):
\begin{eqnarray*}
m_\fg &:& \fg(\KK,\KK') \to \fg(\tilde\KK,\tilde\KK') \\
m_V &:& V(\KK,\KK') \to V(\tilde\KK,\tilde\KK')
\end{eqnarray*}
such that
\begin{equation} [m_\fg(E), m_V (X)] = m_V [E,X] \label{eq-itw} \end{equation}
for $E\in\fg(\KK,\KK')$ and $X\in V(\KK,\KK'),$
with the l.h.s bracket taken in $\fm(\KK,\KK')$ and the r.h.s. one in $\fm(\tilde\KK,\tilde\KK').$
\end{cor}
This finally leads us to a proof of the Proposition:
\begin{proof}[Proof of Proposition \ref{pro-incl}]
The inclusions are simply the restrictions $ m_\fg|_\fh $ and $m_V|_W$ of the maps described by Corollary \ref{cor-incl}. That $m_V(W) \simeq m_\fg(\fh)$ as a $m_\fg(\fh)$-module is guaranteed by the intertwining
property (\ref{eq-itw}) of $m_V.$ Finally, $m_\fg(\fh)$ acts nontrivially at least on $m_V(W^\bot) \subset m_V(W)^\bot.$
\end{proof}

\subsection{Examples}
We will now give
a couple of examples of subalgebras satisfying the conditions
of Theorem \ref{thm-red}. As the examples of $\fg(\KK,\CC)$-geometries
have been given by Nurowski in \cite{nurowski-2006}, we concentrate on constructions applicable
to $\fg(\KK,\HH)$- and $\fg(\KK,\OO)$-geometries.

\begin{eg}[Subalgebra of type $\fsu(2)$]
Set $$\fg = \fg(\RR,\HH) = \der\hx\RR \oplus iL_{\hx\RR},\quad V=V(\RR,\HH)=\CC\otimes\hx\RR.$$ 
Define $X,Y\in\hx\RR:$
$$ X = \left(\begin{matrix} 0 & 1 & 0 \\ 1 & 0 & 0 \\ 0 & 0 & 0 \end{matrix}\right),
\quad
 Y = \left(\begin{matrix} -1 &0  & 0 \\ 0 & 1 & 0 \\ 0 & 0 & 0 \end{matrix}\right) $$
and $Z=X^2 = Y^2.$ We have:
$$ Z\circ X = X,\quad Z\circ Y = Y,\quad X\circ Y = 0. $$
Let $$\fh = \Span\{ D, iL_X, -iL_Y \} \subset\fg, $$ where $D=[L_X,L_Y].$ 
One easily checks that
$\fh$ is a subalgebra of $\fg$ isomorphic to $\fsu(2).$ 
Let $$ W = \Span\{1\otimes X,1\otimes Y,i\otimes Z\}\subset V. $$ One easily checks that $\fh$ preserves
$W$ (and $W^\bot$) and that the map
$$ B : \fh \to W,\quad B(ix L_X - iy L_Y + zD) = x Y + y X + iz Z $$
satisfies $ B([A',A]) = A'(B(A)) $ for $A,A' \in \fh.$
The action of $\fh$ on $W^\bot$ is evidently nontrivial.

Theorem \ref{thm-red} for $\fh$ and $W$, extended by Proposition \ref{pro-incl} yields an example
of a $\fg(\KK,\KK')$-geometry with characteristic torsion for all $\KK=\RR,\CC,\HH,\OO$ and $\KK'=\HH,\OO.$
\end{eg}

\begin{eg}[Subalgebra of type $\fsu(3)$]
Set $$\fg = \fg(\CC,\HH) = \der\hx\CC \oplus iL_{\hx\CC},\quad V=V(\CC,\HH)=\CC\otimes\hx\CC.$$ 
Let $$\fh = \der\hx\CC \subset \fg,$$ being of course isomorphic to $\fsu(3).$ 
Let $$W = 1 \otimes \sx\CC \subset V.$$ One easily checks, that $\fh$ preserves $W$ (and $W^\bot$) and the map
$$ B : \fh \to W,\quad B([L_X,L_Y]) = 1 \otimes i(XY-YX) $$
is bijective and satisfies $ B([A',A]) = A'(B(A)) $ for $A,A' \in \fh.$
The action of $\fh$ on $W^\bot$ is evidently nontrivial.

Theorem \ref{thm-red} for $\fh$ and $W$, extended by Proposition \ref{pro-incl} yields an example
of a $\fg(\KK,\KK')$-geometry with characteristic torsion for all $\KK=\CC,\HH,\OO$ and $\KK'=\HH,\OO.$
\end{eg}

\begin{eg}[Subalgebra of type $\fso(8)$]
Set
$$ \fg = \fg(\OO,\OO) \simeq \fe_7\oplus\fsp(1),\quad V = V(\OO,\OO) = \CC \otimes \FF(\hx\OO). $$
The compact algebra $\fe_7$ has $\fsu(8)$ among its maximal subalgebras, and one can check (e.g. using
{\tt LiE} \cite{leeuwen}) that
the complex 56-dimensional module $V$ decomposes under the action of $\fsu(8)\subset\fe_7$ into
$$ V \simeq \Lam^2 \CC^8 \oplus \overline{\Lam^2 \CC^8}, $$
where $\CC^8$ is the defining representation of $\fsu(8).$

Regarding now the usual inclusion $\fso(8)\subset\fsu(8)$ (given by some generic nondegenerate quadratic form on $\CC^8),$ we find that $V$ decomposes under $\fso(8)\subset\fe_7$ into two copies of the complexified adjoint
representation:\footnote{
Since $\fso(8,\CC) \simeq \Lam^2 \CC^8 \simeq \overline{\Lam^2 \CC^8}$ as $\fso(8)$-modules.
}
$$ V \simeq \fso(8,\CC) \oplus \fso(8,\CC) $$
The explicit form of the intertwiner is rather complicated \cite{yokota-1982}.
Nevertheless, taking $$\fh = \fso(8) \subset \fsu(8) \subset \fe_7 \subset \fg$$
and $$ W = \fso(8) \subset \fso(8,\CC) \subset V, $$ 
the real part of one of the two copies of $\fso(8,\CC)$, we can apply Theorem \ref{thm-red} to
obtain an example of a $\fg(\OO,\OO)$-geometry with characteristic torsion.

\end{eg}

\section{Proofs of Lemmas 34, 35 \& 41}
\label{sec-xprf}

\begin{proof}[Proof of Lemma \ref{lem-dxi}] To find correct combinatorial factors, note that
$$ \Xi_{abcdef} =  \frac{1}{20} (
\Lambda_{\alpha\beta\gamma} \bar\Lambda_{\bar\epsilon\bar\phi\bar\kappa}
+ \Lambda_{\beta\gamma\epsilon} \bar\Lambda_{\bar\alpha\bar\phi\bar\kappa} + \dots ) $$
with ${6\choose 3}=20$ terms on the r.h.s.,
so that we have
\begin{eqnarray*}
20^2\ \Xi_{\alpha defkl} \Xi^{\beta defkl} &=& 10\ \Lambda_{\alpha\delta\epsilon}\bar\Lambda_{\bar\phi\bar\kappa\bar\lambda} \Lambda^{\bar\phi\bar\kappa\bar\lambda} \bar\Lambda^{\beta\delta\epsilon} \\
&=& 10\  N\delta^\beta_\alpha,
\end{eqnarray*}
so that $\Xi_{adefkl}\Xi^{bdefkl} = \frac{N}{40} \delta^b_a.$ Moreover,
\begin{eqnarray*}
20^2\ \Xi_{\alpha}{}^{\delta efkl} \Xi^{\gamma}{}_{\beta efkl} 
&=&6\ \Lambda_{\alpha}{}^{\bar\epsilon\bar\phi} \bar\Lambda^{\delta\kappa\lambda} \Lambda_{\beta\kappa\lambda} \bar\Lambda^\gamma{}_{\bar\epsilon\bar\phi} \\
&=&6\ \delta^\gamma_\alpha \delta^\delta_\beta
\end{eqnarray*}
\begin{eqnarray*}
20^2\ \Xi_{\alpha}{}^{\bar\delta efkl} \Xi^{\gamma}{}_{\bar\beta efkl}
&=& 4\ \Lambda_{\alpha}{}^{\bar\delta\bar\epsilon} \bar\Lambda^{\phi\kappa\lambda} \Lambda_{\phi\kappa\lambda} \bar\Lambda^\gamma{}_{\bar\beta\bar\epsilon} \\
&=& 4\ N\Lambda_{\alpha}{}^{\bar\delta\bar\epsilon} \bar\Lambda^\gamma{}_{\bar\beta\bar\epsilon}
\end{eqnarray*}
\begin{eqnarray*}
20^2\ \Xi_{\alpha}{}^{\delta efkl} \Xi^{\bar\gamma}{}_{\bar\beta efkl}
&=& 6\ \Lambda_\alpha{}^{\bar\epsilon\bar\phi} \bar\Lambda^{\delta\kappa\lambda} \Lambda^{\bar\gamma}{}_{\kappa\lambda}\bar\Lambda_{\bar\beta\bar\epsilon\bar\phi} \\
&=& 6\ h_{\bar\beta\alpha} h^{\delta\bar\gamma}
\end{eqnarray*}
$$ \Xi_{\alpha}{}^{\bar\delta efkl} \Xi^{\bar\gamma}{}_{\beta efkl} = 0, $$
so that
\begin{eqnarray*}
\cD_\Xi|_{\Re(\Lam^{2,0}\oplus\Lam^{0,2})} &=& \frac{3}{200} \\
\cD_\Xi|_{\Re\Lam^{1,1}} &=& \frac{N}{100} \cD_\Lam + \frac{3}{200} \pr_0,
\end{eqnarray*}
where $\pr_0$ is the orthonormal projection onto $\fu(1).$ The Lemma is then proved recalling the formula for $\cD_\Lam$ given at the beginning of the present chapter. 
\end{proof}

\begin{proof}[Proof of Lemma \ref{lem-dmho}]
Note first that the map $\cD_\mho,$ being by construction $G(\KK,\OO)$-equivariant, is given
by a multiple of identity when restricted to each of the irreducible subspaces of $\Lam^2 \cV_3(\KK),$
namely:
$$ \Lam^2 \cV_3(\KK) = \fsp(1) \oplus \cg_3(\KK) \oplus [\ft\cap\fsp(\cV_3(\KK,\omega))] \oplus 
\bot. $$
Distinguishing as usually one of the complex structures of $\fsp(1),$ by convention $I,$ to  consider
$\cV_3(\KK)$ as a complex vector space, one notes that each of the former subspaces has a nonzero
intersection with $\Lam^{1,1}\simeq \fu(\cV_3(\KK)).$ It is therefore sufficient to check
the eigenvalue of $\cD_\mho$ on  each of the following:
$$ \Re\Lam^{1,1} = \fu(1) \oplus \cg_3(\KK) \oplus [\ft\cap\fsp(\cV_3(\KK),\omega)] \oplus \fsp^\bot_0(\cV_3(\KK),\omega), $$
where $\fu(1)$ is generated by the distinguished complex structure and $\fsp^\bot_0(\cV_3(\KK),\omega)$
is the complement of $\fsp(\cV_3(\KK),\omega)$ in $\fsu(\cV_3(\KK)).$

Recall now the formula for the orthorgonal projection of $\mho \in \Sym^{4,4}$ onto $\Sym^{4,0}\otimes\Sym^{0,4}:$
$$
\mho_{\mu_1\dots\mu_4\bar\mu_5\dots\bar\mu_8} = \frac{256}{70} P_{44}(q\otimes q)_{\mu_1\dots\mu_4\sigma_1\dots\sigma_4} \bar\omega^{\sigma_1}{}_{\bar\mu_5}\cdots\bar\omega^{\sigma_4}{}_{\bar\mu_4},
$$
where the operator $P_{44} : \Sym^{4,0}\otimes\Sym^{4,0} \to \Sym^{4,0} \otimes \Sym^{4,0}$ 
is such that the image in $\Sym^{4,0}_{\mu_1\dots\mu_4}\otimes\Sym^{4,0}_{\nu_1\dots\nu_4}$
of a tensor
$t_{\tau_1\dots\tau_1\sigma_1\dots\sigma_4} \in\Sym^{4,0}_{\tau_1\dots\tau_4}\times\Sym^{4,0}_{\sigma_1\dots\sigma_4}$ is:
$$ P_{44}(t)_{\mu_1\dots \mu_4 \nu_1 \dots \nu_4} =
\delta^{[\tau_1}_{(\mu_1} \delta^{\sigma_1]}_{(\nu_1}
\delta^{[\tau_2}_{\mu_2} \delta^{\sigma_2]}_{\nu_2}
\delta^{[\tau_3}_{\mu_3} \delta^{\sigma_3]}_{\nu_3}
\delta^{[\tau_4}_{\mu_4)} \delta^{\sigma_4]}_{\nu_4)}
\ t_{\tau_1\dots\tau_4\sigma_1\dots\sigma_4},
$$
where the index sets $\mu_1\dots\mu_4$ and $\nu_1\dots\nu_4$ are symmetrized \emph{separately}.
Note that, introducing $c = \frac{256}{70}$ to avoid clutter,
\begin{eqnarray*}
c^{-1} \mho_{m_1\dots m_8} &=& 
P_{44}(q\otimes q)_{\mu_1\mu_2\mu_3\mu_4 \sigma_1\sigma_2\sigma_3\sigma_4} 
\bar\omega^{\sigma_1}{}_{\bar\mu_5}\bar\omega^{\sigma_2}{}_{\bar\mu_6}\bar\omega^{\sigma_3}{}_{\bar\mu_7}\bar\omega^{\sigma_4}{}_{\bar\mu_8}
\\ &+&
P_{44}(q\otimes q)_{\mu_2\mu_3\mu_5 \sigma_1\sigma_2\sigma_3\sigma_4} 
\bar\omega^{\sigma_1}{}_{\bar\mu_1}\bar\omega^{\sigma_2}{}_{\bar\mu_6}\bar\omega^{\sigma_3}{}_{\bar\mu_7}\bar\omega^{\sigma_4}{}_{\bar\mu_8}
\\ &+& \dots
\end{eqnarray*}
with ${8\choose 4}=70$ terms on the r.h.s., so that we have
\begin{eqnarray*}
c^{-2} \mho_{\alpha m_2\dots m_8} \mho^{\beta m_2\dots m_8} &=& 35\ P_{44}(q\otimes q)_{\alpha \mu_2\dots\mu_8} P_{44}(\bar q\otimes \bar q)^{\beta\mu_2\dots \mu_8}.
\end{eqnarray*}
Moreover,
\begin{eqnarray*} 
 c^{-2}\ \mho_{\alpha\bar\delta m_3\dots m_8} \mho^{\bar\gamma\beta m_3\dots m_8} &=& 20\ 
 P_{44}(q\otimes q)_{\alpha\mu\nu\rho\xi\eta\zeta\sigma} P_{44}(\bar q\otimes \bar q)^{\beta\mu\nu\rho\xi\eta\zeta\tau} \bar\omega^\sigma{}_{\bar\delta} \omega_\tau{}^{\bar\gamma} 
\\
 c^{-2}\ \mho_{\alpha\delta m_3\dots m_8} \mho^{\gamma\beta m_3\dots m_8} &=& 15\ 
P_{44}(q\otimes q)_{\alpha\delta\mu\nu\rho\xi\eta\zeta} P_{44}(\bar q\otimes \bar q)^{\gamma\beta\mu\nu\rho\xi\eta\zeta}
\end{eqnarray*}
Using computer algebra to keep track of the combinatorial factors, we expand the symmetrizers and
find that:
\begin{eqnarray*}
&& 16^2 P_{44}(q\otimes q)_{\alpha\mu_1\mu_2\mu_3\nu_1\dots\nu_4} P_{44}(\bar q\otimes \bar q)^{\beta\mu_1\mu_2\mu_3\nu_1\dots\nu_4} 
\\ &=& -2 (K_0)_\alpha^\beta + 12 (K_1)_\alpha^\beta -72 (K_2)_\alpha^\beta + 62 (K_3)_\alpha^\beta
\end{eqnarray*}
\begin{eqnarray*}
&& 16^2 P_{44}(q\otimes q)_{\alpha\beta\mu_1\mu_2\nu_1\dots\nu_4} P_{44}(\bar q\otimes \bar q)^{\gamma\delta\mu_1\mu_2\nu_1\dots\nu_4} 
\\ &=& 
 4 (H_5)_{\alpha\beta}{}^{\gamma\delta}
+4 (\bar H_5)^{\gamma\delta}{}_{\alpha\beta}
-4 (H_3)_{\beta\alpha}{}^{\gamma\delta}
-4 (\bar H_3)^{\delta\gamma}{}_{\alpha\beta}
\\ &+& 32 (H_1)_{\alpha\beta}{}^{\gamma\delta}
-32 (H_7)_{\alpha\beta}{}^{\gamma\delta}
+2 (H_6)_{\alpha\beta}{}^{\gamma\delta}
-24 (H_4)_{\alpha\beta}{}^{\gamma\delta}
+22 (H_2)_{\alpha\beta}{}^{\gamma\delta}
\end{eqnarray*}
\begin{eqnarray*}
&& 16^2 P_{44}(q\otimes q)_{\alpha\mu_1\mu_2\mu_3\nu_1\nu_2\nu_3\beta} P_{44}(\bar q\otimes \bar q)^{\gamma\mu_1\mu_2\mu_3\nu_1\nu_2\nu_4\delta}
\\ &=& 
 29 (H_1)_{\alpha\beta}{}^{\gamma\delta}
+ 21 (H_7)_{\beta\alpha}{}^{\gamma\delta}
+ 9 (H_2)_{\alpha\beta}{}^{\gamma\delta}
\\ &-& \frac{31}{2} (H_1)_{\beta\alpha}{}^{\gamma\delta}
-\frac{69}{2} (H_7)_{\alpha\beta}{}^{\gamma\delta}
-9 (H_4)_{\alpha\beta}{}^{\gamma\delta}
\end{eqnarray*}
where the contractions of $q,q,\bar q,\bar q$ are:
$$ (K_0)_\alpha^\beta = q_{\alpha\epsilon\phi\kappa} q_{\lambda\mu\nu\rho} \bar q^{\beta\mu\nu\rho} \bar q^{\epsilon\phi\kappa\lambda} = \left(\frac{N+1}{2}\right)^2 \delta_\alpha^\beta $$
$$ (K_1)_\alpha^\beta =  q_{\alpha\xi\epsilon\phi} q_{\kappa\lambda\mu\nu} \bar q^{\beta\xi\mu\nu} \bar q^{\epsilon\phi\kappa\lambda} = (1+\chi^2)\frac{N+1}{2}\delta^\beta_\alpha  $$
$$ (K_2)_\alpha^\beta =  q_{\alpha\xi\eta\epsilon} q_{\phi\kappa\lambda\mu}\bar q^{\beta\xi\eta\mu} \bar q^{\epsilon\phi\kappa\lambda}  = \left(\frac{N+1}{2}\right)^2 \delta_\alpha^\beta $$
$$ (K_3)_\alpha^\beta =  q_{\alpha\xi\eta\zeta} q_{\epsilon\phi\kappa\lambda} \bar q^{\beta\xi\eta\zeta} \bar q^{\epsilon\phi\kappa\lambda} = N\left(\frac{N+1}{2}\right)^2 \delta^\beta_\alpha $$
\begin{eqnarray*} (H_1)_{\alpha\beta}{}^{\gamma\delta} =
q_{\alpha\mu\nu\rho} q_{\beta\xi\eta\zeta} \bar q^{\gamma\mu\nu\rho} \bar q^{\delta\xi\eta\zeta} 
&= \frac{N+1}{2} \delta^\gamma_\alpha \delta^\delta_\beta \end{eqnarray*}
\begin{eqnarray*} (H_2)_{\alpha\beta}{}^{\gamma\delta} &=&
q_{\alpha\beta\mu\nu} q_{\rho\xi\eta\zeta} \bar q^{\gamma\delta\mu\nu} \bar q^{\rho\xi\eta\zeta} 
\\ &=& \frac{N(N+1)}{2} \left[\frac{1}{2}( \delta^\gamma_\alpha \delta^\delta_\beta + \delta^\delta_\alpha \delta^\gamma_\beta ) + \chi q_{\alpha\beta}{}^{\gamma\delta} \right] \end{eqnarray*}
\begin{eqnarray*} (H_3)_{\alpha\beta}{}^{\gamma\delta} &=&
q_{\alpha\mu\nu\rho} q_{\beta\xi\eta\zeta} \bar q^{\gamma\delta\mu\nu} \bar q^{\rho\xi\eta\zeta} 
\\ &=& \frac{N+1}{2} \left[\frac{1}{2}( \delta^\gamma_\alpha \delta^\delta_\beta + \delta^\delta_\alpha \delta^\gamma_\beta ) + \chi q_{\alpha\beta}{}^{\gamma\delta} \right] \end{eqnarray*}
\begin{eqnarray*} (H_4)_{\alpha\beta}{}^{\gamma\delta} &=&
q_{\alpha\beta\mu\nu} q_{\phi\kappa\lambda\rho} \bar q^{\gamma\delta\mu\rho} \bar q^{\phi\kappa\lambda\nu} 
\\ &=& \frac{N+1}{2} \left[\frac{1}{2}( \delta^\gamma_\alpha \delta^\delta_\beta + \delta^\delta_\alpha \delta^\gamma_\beta ) + \chi q_{\alpha\beta}{}^{\gamma\delta} \right] \end{eqnarray*}
\begin{eqnarray*} (H_5)_{\alpha\beta}{}^{\gamma\delta} &=&
q_{\alpha\mu\nu\rho} q_{\beta\xi\eta\zeta} \bar q^{\gamma\delta\zeta\mu} \bar q^{\nu\rho\xi\eta} 
\\ &=&
(1+\chi^2) \left[\frac{1}{2} (\delta^\gamma_\alpha \delta^\delta_\beta + \delta^\delta_\alpha \delta^\gamma_\beta ) + \chi q_{\alpha\beta}{}^{\gamma\delta} \right] - \frac{\chi}{2} q_{\alpha\beta}{}^{\gamma\delta} \end{eqnarray*}
\begin{eqnarray*} (H_6)_{\alpha\beta}{}^{\gamma\delta} &=&
q_{\alpha\beta\mu\nu} q_{\epsilon\phi\kappa\lambda} \bar q^{\gamma\delta\epsilon\phi} \bar q^{\kappa\lambda\mu\nu} 
\\&=& (1+\chi^2) \left[\frac{1}{2} (\delta^\gamma_\alpha \delta^\delta_\beta + \delta^\delta_\alpha \delta^\gamma_\beta ) + \chi q_{\alpha\beta}{}^{\gamma\delta} \right] + \chi q_{\alpha\beta}{}^{\gamma\delta} \end{eqnarray*}
\begin{eqnarray*} (H_7)_{\alpha\beta}{}^{\gamma\delta} &=&
q_{\alpha\mu\nu\rho} q_{\beta\xi\eta\zeta} \bar q^{\gamma\zeta\mu\nu} \bar q^{\delta\rho\xi\eta} 
\\&=&\frac{N+2}{4} \delta^\gamma_\alpha \delta^\delta_\beta + 
\frac{1+2\chi^2}{4} \delta^\delta_\alpha \delta^\gamma_\beta
- \frac{\chi^2}{2} \omega_{\alpha\beta} \bar\omega^{\gamma\delta} 
+ \chi(1-\chi^2) q_{\alpha\beta}{}^{\gamma\delta}\end{eqnarray*}
and by slight abuse of notation $ q_{\alpha\beta}{}^{\gamma\delta} := q_{\alpha\beta\mu\nu} \bar\omega^{\mu\gamma} \bar\omega^{\nu\delta}.$
Collecting terms yields
\begin{eqnarray*}
&& 16^2 P_{44}(q\otimes q)_{\alpha\mu_1\mu_2\mu_3\nu_1\dots\nu_4} P_{44}(\bar q\otimes \bar q)^{\beta\mu_1\mu_2\mu_3\nu_1\dots\nu_4} 
\\ &=&  \frac{N+1}{2} [ 25(N-1) + 12\chi^2 ]\ \delta^\beta_\alpha
\end{eqnarray*}
and
\begin{eqnarray*}
&& 16^2 P_{44}(q\otimes q)_{\alpha\beta\mu_1\mu_2\nu_1\dots\nu_4} P_{44}(\bar q\otimes \bar q)^{\gamma\delta\mu_1\mu_2\nu_1\dots\nu_4} 
\\ &=& \frac{11N(N+1)-6+10\chi^2}{2}\ \delta^\gamma_\alpha\delta^\delta_\beta
\\ &+& \frac{N(11N-5)-22(1+\chi^2)}{2}\ \delta^\delta_\alpha\delta^\gamma_\beta
\\&+& 16\chi^2\ \omega_{\alpha\beta}\bar\omega^{\gamma\delta}
+ \left[N(11N-5) -40+42\chi^2\right]\ \chi q_{\alpha\beta}{}^{\gamma\delta}
\end{eqnarray*}
and
\begin{eqnarray*}
&& 16^2 P_{44}(q\otimes q)_{\alpha\mu_1\mu_2\mu_3\nu_1\nu_2\nu_3\beta} P_{44}(\bar q\otimes \bar q)^{\gamma\mu_1\mu_2\mu_3\nu_1\nu_2\nu_4\delta}
\\ &=& \frac{N(18N+47)+2+84\chi^2}{8} \ \delta^\gamma_\alpha\delta^\delta_\beta
\\ &+& \frac{2N(9N-10)-65+138\chi^2}{8}\ \delta^\delta_\alpha\delta^\gamma_\beta
\\&+& \frac{111\chi^2}{4}\ \omega_{\alpha\beta}\bar\omega^{\gamma\delta}
+ \frac{9}{2}\left[N^2-4+3\chi^2\right]\ \chi q_{\alpha\beta}{}^{\gamma\delta}.
\end{eqnarray*}
Thus applying $ \mho_a{}^{dm_3\dots m_8} \mho^c{}_{bm_3\dots m_8} $ to $ \Re\Lam^{1,1}_{ab} \ni F_{ab} = f_{\alpha\bar\beta}
+ \bar f_{\bar\alpha\beta},$ with $ \bar f_{\bar\alpha\beta} = - f_{\beta\bar\alpha},$ we have
$\cD_\mho(F)_{ab} = f'_{\alpha\bar\beta} + \bar f'_{\bar\alpha\beta}$ where
\begin{eqnarray*}  f'_{\alpha\bar\beta} &=& 
\mho_\alpha{}^{\bar\delta m_3\dots m_8} \mho^\gamma{}_{\bar\beta m_3\dots m_8} f_{\gamma\bar\delta}
+
\mho_\alpha{}^{\delta m_3\dots m_8} \mho^{\bar\gamma}{}_{\bar\beta m_3\dots m_8} \bar f_{\bar\gamma\delta}
\\&=& \left(
\mho_\alpha{}^{\bar\delta m_3\dots m_8} \mho^\gamma{}_{\bar\beta m_3\dots m_8} -
\mho_\alpha{}^{\gamma m_3\dots m_8} \mho^{\bar\delta}{}_{\bar\beta m_3\dots m_8}
\right) f_{\gamma\bar\delta}
\\&=& \frac{c^2} {16^2} f_{\gamma\bar\delta} \left\lbrace
15 \left[
\frac{11N(N+1)-6+10\chi^2}{2} \delta^\gamma_\alpha \delta^{\bar\delta}_{\bar\beta} \right.\right.
\\ &&\quad\quad + \frac{N(11N-5)-22(1+\chi^2)}{2} h_{\bar\beta\alpha} h^{\gamma\bar\delta}
\\ && \left. \quad\quad + 16\chi^2 \omega_\alpha{}^{\bar\delta} \bar\omega^\gamma{}_{\bar\beta}
+ (N(11N-5)-40+42\chi^2)\chi q_{\alpha\bar\beta}{}^{\gamma\bar\delta} \right]
\\& & -20 \left[
\frac{N(18N+47)+2+84\chi^2}{8} h_{\bar\beta\alpha} h^{\gamma\bar\delta}\right.
\\&&\quad\quad - \frac{2N(9N-10)-65+138\chi^2}{8} \omega_\alpha{}^{\bar\delta} \bar\omega^\gamma{}_{\bar\beta}
\\&&\left.\left.\quad\quad + \frac{111\chi^2}{4} \delta_\alpha^\gamma \delta_{\bar\beta}^{\bar\delta}
-\frac{9}{2} (N^2-4+3\chi^2) \chi q_{\alpha\bar\beta}{}^{\gamma\bar\delta}\right]
\right\rbrace.
\end{eqnarray*}
Again by abuse of notation we have introduced 
$ q_{\alpha\bar\beta}{}^{\gamma\bar\delta} := q_{\alpha\mu\nu}{}^{\bar\delta}
\bar\omega^\mu{}_{\bar\beta} \bar\omega^{\nu\gamma}.$
Recollecting terms  we finally have
\begin{eqnarray*} \frac{2}{5} 16^2 c^{-2} \cD_\mho|_{\fu(\cV_3(\KK))} 
&=& 3(11N^2+1N-6-84\chi^2)\ \id	\\
&+& (15N^2-62N-68-150\chi^2)\ \pr_0	\\
&+& (18N^2-20N-65+234\chi^2)\ \sigma	\\
&+& 6(17N^2-5N-64+60\chi^2) \cD_q	
\end{eqnarray*}
where $\cD_q$ is trivially extended to the unitary complement of $\fsp(\cV_3(\KK),\omega)$
and $\sigma(E) = J^{-1} EJ,$ so that it is $1$ on $\fsp(\cV_3(\KK))$ and $-1$ on
its unitary complement. Recalling the eigenvalues of $\cD_q$ on $\cg_3(\KK)$ and its
symplectic complement, the result is
\begin{eqnarray*}
\cD_\mho|_{\fu(1)} &=& \frac{5c^2}{512} \left(30N^2 - 9N - 21 - 636\chi^2\right) \\
\cD_\mho|_{\cg_3(\KK)} &=& \frac{5c^2}{512} \left(51N^2 + 13N -83 -18\chi^2 - \sqrt{\kappa+3}\right) \\
\cD_\mho|_{\ft\cap\fsp(\cV_3(\KK),\omega)} &=& \frac{5c^2}{512} \left(51N^2+13N-83-18\chi^2+\frac{1}{\sqrt{\kappa+3}}\right) \\
\cD_\mho|_{\fsp_0^\bot(\cV_3(\KK),\omega)} &=& \frac{5c^2}{512} (15N^2+53N+47-386\chi^2).
\end{eqnarray*}
We also have
$$ \cD_{am_2\dots m_8} \cD^{bm_2\dots m_8} = \frac{35 c^2}{256} \frac{N+1}{2} [25(N-1)+12\chi^2]\ \delta^b_a. $$
Substituting $c=\frac{256}{70},$ and using equivariance of $\cD_\mho$ and irreducibility of the subspaces the Lemma refers to, the proof is complete.
\end{proof}

\begin{proof}[Proof of Lemma \ref{lem-red}]

\def\E{\epsilon}
\def\R{\rho}

We use the Killing form $m :\fm\to\fm^*$ to identify $\fm$ with $\fm^*$ in what follows.
We will moreover, \emph{independently of the previous indexing conventions}, use $a,b,c,\dots$
to index $\fg,$ $i,j,k,\dots$ to index $V$ and $\alpha,\beta\,\gamma,\dots$ to index $\fh$
in such a way that greek letters denote subalgebras of the algebras indexed by corresponding
latin ones:
$$ \fh_\alpha \subset \fg_a,\quad \fh_\beta\subset\fg_\beta,\quad \dots $$
Let the symbols $\epsilon$ and $\rho$ express the
bracket on $\fm$:
\begin{eqnarray*}
[E,F]_c &=& \E_{abc} E^a F^b \\ {}
[E,X]_j &=& \R_{aij} E^a X^i \\ {}
[X,Y]_a &=& -\R_{aij} X^i Y^j
\end{eqnarray*}
for $E,F\in\fg$ and $X,Y\in V.$
The bracket on $\fk$ is then given by:
\begin{eqnarray*}
[E,F]_\gamma &=& \E_{\alpha\beta\gamma} E^\alpha E^\beta \\ {}
[E,X]_j &=& \R_{\alpha ij} E^\alpha X^i \\ {}
[X,Y]_\alpha &=& c_{\alpha ij} X^i Y^j \\ {}
[X,Y]_k &=& c_{ijk} X^i Y^j
\end{eqnarray*}
where $\R_{\alpha ij}$ denotes the restriction of $\R_{a ij} \in \fg_a \otimes (\Lam^2 V)_{ij}$ 
to $\fh_\alpha\otimes (\Lam^2 V)_{ij}$ and
\begin{eqnarray*}
c_{\alpha ij} &=& \R_{\alpha ij} \\
c_{ijk} &=& 3 B^\alpha_{[i} \R^\alpha_{jk]} 
-  \E_{\alpha\beta\gamma} B^\alpha_i B^\beta_j B^\gamma_k
\end{eqnarray*}
and $B(E)_i = B^\alpha_i E^\alpha $ satisfies
\begin{eqnarray*}
B^\alpha_m B^\beta_m &=& m^{\alpha\beta} \\
\E_{\alpha\beta\gamma} B^\gamma_i - \R^\alpha_{mi} B^\beta_m &=& 0. \end{eqnarray*}

We now investigate the Jacobi identities for $\fk:$
\begin{enumerate}
\item{$[[\fh,\fh],\fh]$ : ${\E^\delta}_{[\alpha\beta}{\E^\delta}_{\gamma]\varepsilon} = 0, $ satisfied since $\fh$ is a subalgebra in $\fg.$
}
\item{$[[\fh,\fh],V]$ : 
$\E_{\alpha\beta\gamma} \R^\gamma_{ik} + 2 \R^\alpha_{m[i}\R^\beta_{k]m} = 0,$
satisfied since $\rho$ is equivariant.}
\item{$[[V,V],\fh]$:
\begin{eqnarray*}
c^\beta_{ij} \E_{\beta\alpha\gamma} + 2 \R^\alpha_{m[i} c^\gamma_{j]m} &= 0&
\textrm{i.e.}\ c^\alpha_{ij}\ \textrm{equivariant}
\\
c_{m[ij} \R^\alpha_{k]m} &= 0& 
\textrm{i.e.}\ c_{ijk}\ \textrm{equivariant}
\end{eqnarray*}
}
\item{$[[V,V],V]$:
\begin{eqnarray*}
c_{m[ij}c^\alpha_{k]m} &=& 0 \\
c_{m[ij}c_{k]lm} - c^\alpha_{[ij} \R^\alpha_{k]l} &=& 0.
\end{eqnarray*}
}
\end{enumerate}

Equivariance of $B,\ \R$ and $\E$ automatically ensures satisfaction of all the equations except
for the last one, namely:
\[{c^m}_{[ij} {c^m}_{k]l} = \R^\alpha_{[ij}\R^\alpha_{k]l}. \]

We shall compute the first term. 
In the following formulas the indices $ijk$ are implicitly antisymmetrized (the parentheses
have been supressed for the sake of readability):
\begin{eqnarray*} 
 c_{mij} c_{mkl} &=&
(B^\alpha_m \R^\alpha_{ij} + 2 B^\alpha_{i} \R^\alpha_{jm}
- \E_{\alpha\beta\gamma} B^\alpha_m B^\beta_i B^\gamma_j)\\
&\times& (B^\lambda_m \R^\lambda_{kl} + B^\lambda_{k} \R^\lambda_{lm} - B^\lambda_{l} \R^\lambda_{km}
- \E_{\lambda\mu\nu} B^\lambda_m B^\mu_k B^\nu_l). \end{eqnarray*}
Calculating each term (using the antisymmetry in $ijk$) yields:
\begin{eqnarray*}
B^\alpha_m B^\beta_m \R^\alpha_{ij} \R^\lambda_{kl} &=& \R^\alpha_{ij} \R^\alpha_{kl},
\\
2 B^\alpha_i B^\lambda_k \R^\alpha_{jm}\R^\lambda_{lm} &=& \\
= 
B^\alpha_i B^\lambda_k \R^\alpha_{jm}\R^\lambda_{lm} -
B^\lambda_k B^\alpha_i \R^\lambda_{jm}\R^\alpha_{lm} 
&=& \\
= - 2 B^\alpha_i B^\lambda_k \R^\alpha_{m[j} \R^\lambda_{l]m} 
&=& \E_{\alpha\lambda\gamma} B^\alpha_i B^\lambda_k \R^\gamma_{jl},
\\ 
-2 B^\alpha_i B^\lambda_l \R^\alpha_{jm} \R^\lambda_{km} 
= 2 B^\alpha_i B^\lambda_l \R^\alpha_{m[j} \R^\lambda_{k]m} 
&=& -\E_{\alpha\lambda\gamma} B^\alpha_i B^\lambda_l \R^\gamma_{jk},
\\
2 B^\alpha_i \R^\alpha_{jm} B^\lambda_m \R^\lambda_{kl}
= - 2\E_{\alpha\lambda\gamma} B^\gamma_j B^\alpha_i \rho^\lambda_{kl} 
&=& - 2\E_{\alpha\lambda\gamma} B^\alpha_i B^\lambda_k \rho^\gamma_{jl},
\\
B^\lambda_k \R^\lambda_{lm} B^\alpha_m \R^\alpha_{ij}
= - \E_{\lambda\alpha\gamma} B^\gamma_l B^\lambda_k \R^\alpha_{ij} 
&=& \E_{\alpha\lambda\gamma} B^\alpha_i B^\lambda_l \R^\gamma_{jk},
\\
-B^\lambda_l \R^\lambda_{km} B^\alpha_m \R^\alpha_{ij}
= \E_{\lambda\alpha\gamma} B^\gamma_k B^\lambda_l \R^\alpha_{ij} 
&=& \E_{\alpha\lambda\gamma} B^\alpha_i B^\lambda_l \R^\gamma_{jk},
\\
- \E_{\alpha\beta\gamma} B^\alpha_m B^\beta_i B^\gamma_j B^\lambda_m \R^\lambda_{kl} 
= - \E_{\alpha\beta\gamma} B^\beta_i B^\gamma_j \R^\alpha_{kl} 
&=& \E_{\alpha\lambda\gamma} B^\alpha_i B^\lambda_k \R^\gamma_{jl},
\\
- \E_{\lambda\mu\nu} B^\lambda_m B^\mu_k B^\nu_l B^\alpha_m \R^\alpha_{ij}
= -\E_{\lambda\mu\nu} B^\mu_k B^\nu_l \R^\lambda_{ij}
&=& -\E_{\alpha\lambda\gamma} B^\alpha_i B^\lambda_l \R^\gamma_{jk},
\\
-2 B^\alpha_i \E_{\lambda\mu\nu} \R^\alpha_{jm} B^\lambda_m B^\mu_k B^\nu_l =
2 B^\alpha_i {\E^\lambda}_{[\alpha\gamma} {\E^\lambda}_{\mu]\nu} B^\gamma_j B^\mu_k B^\nu_l &=& 0,
\\
-B^\lambda_k \E_{\alpha\beta\gamma} \R^\lambda_{lm} B^\alpha_m B^\beta_i B^\gamma_j =
-B^\lambda_k \E^\alpha_{[\beta\gamma} \E^\alpha_{\lambda]\delta} B^\delta_l B^\beta_i B^\gamma_j &=& 0,
\\
B^\lambda_l \E_{\alpha\beta\gamma} \R^\lambda_{km} B^\alpha_m B^\beta_i B^\gamma_j =
-B^\lambda_l \E^\alpha_{[\beta\gamma} \E^\alpha_{\delta]\lambda} B^\delta_k B^\beta_i B^\gamma_j &=& 0,
\\
\E_{\alpha\beta\gamma} \E_{\lambda\mu\nu} B^\alpha_m B^\lambda_m B^\beta_{i} B^\gamma_j B^\mu_{k} B^\nu_l
= {\E^\alpha}_{[\beta\gamma} {\E^\alpha}_{\mu]\nu} B^\beta_i B^\gamma_j B^\mu_k B^\nu_l &=& 0.
\end{eqnarray*}
Summing the expressions on the r.h.s. we obtain $\R^\alpha_{[ij} \R^\alpha_{k]l}.$ 

Thus the Jacobi equations are all satisfied and $\fk$ is a Lie algebra. Then $\fh\oplus V$
is a reductive pair by construction.
Finally we note that the following quadratic form $k$ is invariant:
$$k(A,B) = -m(A,B),\quad k(A,X) = 0,\quad k(X,Y) = m(X,Y) $$
for $A,B\in\fh$ and $X,Y\in V.$ It clearly coincides with $m$ on $V.$
\end{proof}

\backmatter
\newpage
\bibliographystyle{utphys}
\addcontentsline{toc}{chapter}{Bibliography}
\bibliography{refs}{}

\end{document}